\documentclass[11pt,reqno]{amsart}
\usepackage{amsmath,amscd,amssymb,amsfonts,amsthm,courier,relsize,bm}
\usepackage{hyperref,enumerate,mathrsfs,mathtools,slashed}
\usepackage[backend=bibtex,style=alphabetic,maxcitenames=50,maxnames=50]{biblatex} 
\renewbibmacro{in:}{} %removes "In:" that would appear in bibliography
\usepackage{xcolor}  %the next few lines set the colors of the hyperlinks, and remove the color boxes
\hypersetup{
    colorlinks,
    linkcolor={red!50!black},
    citecolor={blue!70!black},
    urlcolor={blue!80!black}
}

\usepackage{tikz-cd}
\newtheorem{theorem}{Theorem}[section]
\newtheorem{corollary}[theorem]{Corollary}
\newtheorem{proposition}[theorem]{Proposition}

\newtheorem{lemma}[theorem]{Lemma}
\theoremstyle{definition}    
\newtheorem{definition}[theorem]{Definition}
\theoremstyle{remark}

\newtheorem{remark}[theorem]{Remark}

\newtheorem{example}[theorem]{Example}

\newcommand{\pair}[2]{\langle #1, #2 \rangle}
\newcommand{\ignore}[1]{}

\newcommand\m[1]{\begin{pmatrix}#1\end{pmatrix}} %this does matrices of any size nicely: simply write \m{a&b&c\\d&e&f} to get the 2x3 matrix with the corresponding entries, etc.
\newcommand{\ol}[1]{\overline{#1}}

\renewcommand{\sf}[1]{\mathsf{#1}}
\newcommand{\wh}[1]{\widehat{#1}}
\newcommand{\scr}[1]{\mathscr{#1}}
\newcommand{\mf}[1]{\mathfrak{#1}}

\newcommand{\tn}[1]{\textnormal{#1}}
\renewcommand{\i}{{\mathrm{i}}}

\def\la{\ensuremath{\langle}}
\def\ra{\ensuremath{\rangle}}

\def\d{\ensuremath{\mathrm{d}}}

\def\Ad{\ensuremath{\textnormal{Ad}}}

\def\Spin{\ensuremath{\textnormal{Spin}}}

\def\C{\ensuremath{\mathcal{C}}}
\def\D{\ensuremath{\mathcal{D}}}
\def\E{\ensuremath{\mathcal{E}}}
\def\F{\ensuremath{\mathcal{F}}}

\def\I{\ensuremath{\mathcal{I}}}

\def\L{\ensuremath{\mathcal{L}}}

\def\R{\ensuremath{\mathcal{R}}}

\def\T{\ensuremath{\mathcal{T}}}
\def\U{\ensuremath{\mathcal{U}}}

\def\W{\ensuremath{\mathcal{W}}}

\def\Z{\ensuremath{\mathcal{Z}}}

\def\bC{\ensuremath{\mathbb{C}}}
\def\bD{\ensuremath{\mathbb{D}}}

\def\bR{\ensuremath{\mathbb{R}}}

\def\bT{\ensuremath{\mathbb{T}}}

\def\bZ{\ensuremath{\mathbb{Z}}}

\def\dE{\ensuremath{\bm{E}}}
\def\dF{\ensuremath{\bm{F}}}
\def\dG{\ensuremath{\bm{G}}}
\def\dK{\ensuremath{\bm{K}}}
\def\dL{\ensuremath{\bm{L}}}
\def\dN{\ensuremath{\bm{N}}}
\def\dR{\ensuremath{\bm{R}}}
\def\dV{\ensuremath{\bm{V}}}
\def\dW{\ensuremath{\bm{W}}}
\def\dX{\ensuremath{\bm{X}}}
\def\dY{\ensuremath{\bm{Y}}}
\def\dnabla{\ensuremath{\bm{\nabla}}}
\def\dz{\ensuremath{\bm{0}}}

\def\dsigma{\ensuremath{\bm{\sigma}}}
\def\dDelta{\ensuremath{\bm{\Delta}}}
\def\dTheta{\ensuremath{\bm{\Theta}}}
\def\dgamma{\ensuremath{\bm{\gamma}}}

\def\vep{\ensuremath{\varepsilon}}

\def\End{\ensuremath{\textnormal{End}}}
\def\Hom{\ensuremath{\textnormal{Hom}}}

\def\ker{\ensuremath{\textnormal{ker}}}
\def\supp{\ensuremath{\textnormal{supp}}}
\def\id{\ensuremath{\textnormal{id}}}

\def\Tr{\ensuremath{\textnormal{Tr}}}
\def\tr{\ensuremath{\textnormal{tr}}}

\def\Cl{\ensuremath{\textnormal{Cl}}}
\def\bCl{\ensuremath{\bC\textnormal{l}}}

\def\gr{\ensuremath{\textnormal{gr}}}

\def\Ahat{\ensuremath{\widehat{\textnormal{\textbf{A}}}}}

\newcommand\fakeslant[1]{%
  \pdfliteral{1 0 0.167 1 0 0 cm}#1\pdfliteral{1 0 -0.167 1 0 0 cm}}
\newcommand\mathbbsl[1]{\mathbb{\fakeslant{#1}}}
\def\sbN{\ensuremath{\mathbbsl{N}}}
\def\sbW{\ensuremath{\mathbbsl{W}}}
\def\sbE{\ensuremath{\mathbbsl{E}}}
\def\sbG{\ensuremath{\mathbbsl{G}}}

\title{A fixed-point formula for Dirac operators on Lie groupoids}

\author{Ahmad Reza Haj Saeedi Sadegh}
\address{Department of Mathematics, Dartmouth College, U.S.}
\email{ahmadreza.hajsaeedisadegh@dartmouth.edu}
\author{Shiqi Liu}
\address{School of Mathematics and Statistics, University of New South Wales, Australia}
\email{shiqi.liu3@unsw.edu.au}
\author{Yiannis Loizides}
\address{Department of Mathematical Sciences, George Mason University, U.S.}
\email{yloizide@gmu.edu}
\author{Jesus Sanchez}
\address{Department of Mathematics and Statistics, Washington University in St. Louis, U.S.}
\email{jesuss@wustl.edu}

\dedicatory{Dedicated to Nigel Higson on the occasion of his $60^{th}$ birthday.}

\begin{document}
\sloppy
\maketitle

\begin{abstract}
We study equivariant families of Dirac operators on the source fibers of a Lie groupoid with a closed space of units and equipped with an action of an auxiliary compact Lie group. We use the Getzler rescaling method to derive a fixed-point formula for the pairing of a trace with the K-theory class of such a family. For the pair groupoid of a closed manifold, our formula reduces to the standard fixed-point formula for the equivariant index of a Dirac operator. Further examples involve foliations and manifolds equipped with a normal crossing divisor.
\end{abstract}

%\tableofcontents

\section{Introduction}
Let $M$ be an even-dimensional closed oriented Riemannian manifold equipped with an isometric action of a compact Lie group $\sf{K}$. Let $(E=E^+\oplus E^-,c\colon T^*M\rightarrow \End(E))$ be a $\sf{K}$-equivariant $\bZ/2$-graded Clifford module with compatible connection $\nabla$. The corresponding Dirac operator $D$, given by the composition
\[ \Gamma(E)\xrightarrow{\nabla}\Gamma(T^*M\otimes E)\xrightarrow{c}\Gamma(E),\]
is a $\sf{K}$-equivariant elliptic differential operator, and therefore has a well-defined $\sf{K}$-equivariant index, defined as the class of the $\bZ/2$-graded kernel $[\ker(D)]=[\ker(D^+)]-[\ker(D^-)]\in R(\sf{K})$ in the representation ring of $\sf{K}$. The Atiyah-Bott-Segal-Singer theorem gives a formula for the character of this representation evaluated at $\gamma \in \sf{K}$ as an integral of characteristic classes over the fixed-point set $M^\gamma$.

One approach to the index formula is the heat kernel method. By the McKean-Singer formula, the character of $[\ker(D)]$ evaluated at $\gamma \in \sf{K}$ equals the supertrace $\Tr_s(\gamma \exp(-tD^2))$ for all $t>0$, and the fixed-point formula for the index can be obtained by calculating the $t\rightarrow 0^+$ limit. The method leads to a local refinement of the index theorem. Various heat kernel proofs of the fixed-point formula were given in well-known work by several authors beginning in the 1970s through the 1980s, notably Patodi, Donnelly, Gilkey, Bismut, Berline and Vergne (cf. \cite[Chapter 6]{BerlineGetzlerVergne}, the historical notes and references therein).

The Getzler rescaling method for calculating the $t\rightarrow 0^+$ limit of the heat supertrace was carried out in the equivariant case by Lafferty-Yu-Zhang \cite{lafferty1992lefschetz}. One of the aims of this article is to give a new treatment of this result; this may be of interest even to readers mostly interested in the classical setting and less so in the general Lie groupoid case. Such readers may go through the article making appropriate substitutions throughout: $G$ with $M\times M$, $A$ with $TM$, $\tau$ with integration over the diagonal in $M\times M$, and so on.

Our approach makes use of a reformulation of Getzler's method introduced recently by Higson and Yi \cite{higsonyi} (see also \cite{ludewig2020short} for another interesting recent application). The idea in this approach is to build the desirable features of Getzler's rescaling into the definition of a vector bundle over Connes' tangent groupoid $\bT M$ in such a way that all the Getzler-rescalings of a section/operator, along with its Getzler-rescaled limit, fit together into a smooth global section/operator. We adapt this approach to the equivariant context and find that it usefully systematizes the fixed-point calculations. For example the fixed-point contributions involve an easy-to-overlook correction factor that does not vanish in the Getzler-rescaling limit, and is related to the holonomy of a geodesic triangle connecting a triple of points $x,y,\gamma x$ near a $\gamma$-fixed point (see for example \cite[Lemma 3.3]{lafferty1992lefschetz}). The Higson-Yi framework neatly organizes the effects of holonomy and is such that one is prevented from accidentally omitting this factor.

We apply Getzler's rescaling method to derive a fixed-point formula in a more general context involving a longitudinal family of Dirac operators on the source fibers of a Lie groupoid. Examples of Lie groupoids are abundant; to name but a few one has (i) $\tn{Pair}(M)=M\times M$ the pair groupoid of $M$, (ii) $\tn{Hol}(\F)$ the holonomy groupoid of a foliation $\F$ of $M$, (iii) the standard Lie groupoid integrating Melrose's $b$-tangent bundle ${}^bTM$ for some embedded hypersurface $Z$, as well as variations thereof, for example involving a fibration of $Z$ over some base, or replacing $Z$ with a normal crossing divisor.

The classical fixed-point formula for the index is associated with the pair groupoid. The smooth convolution algebra $\Psi^{-\infty}(\tn{Pair}(M))$ of the pair groupoid is the algebra of smoothing operators on $M$. It carries a canonical trace given by integration of smoothing kernels along the diagonal in $M\times M$. The Dirac operator $D$ on $M$ determines a class $\tn{ind}(D) \in K_0^\sf{K}(\Psi^{-\infty}(\tn{Pair}(M)))$ in the $\sf{K}$-equivariant algebraic K-theory group, given by a standard formula (recalled in Section \ref{s:equivtracepairing}) involving a parametrix for $D$. The equivariant index equals the evaluation of the canonical trace on the class $\tn{ind}(D)$. Reformulated in this way, the setup generalizes easily to Lie groupoids.

Let $(r,s)\colon G\rightrightarrows M$ be a Lie groupoid with a closed space of units $M$ (the manifold $G$ is not required to be compact). Let $A=\ker(Ts)|_M \rightarrow M$ be the Lie algebroid of $G$. In the examples above, $A$ is (i) $TM$, (ii) $T\F$, (iii) the $b$-tangent bundle ${}^bTM$, respectively. Many familiar constructions for $TM$ carry over immediately to Lie algebroids. Given a fiber metric $h$ on $A$ and a bundle of modules $E\rightarrow M$ for the bundle of Clifford algebras $\Cl(A^*,h^*)$, one can define the `$A$-Dirac operator' $D$ as the composition
\[ \Gamma(E)\xrightarrow{\nabla} \Gamma(A^*\otimes E)\xrightarrow{c}\Gamma(E) \]
where $\nabla$ is an $A$-connection compatible with the Levi Civita $A$-connection. 

The operator $D$ is elliptic if and only if the anchor map $A\rightarrow TM$ is surjective. Regardless $D$ gives rise to a $G$-equivariant family $D^R$ of ordinary (in particular, elliptic) Dirac operators on the source fibers of the Lie groupoid (the relation between $D$ and $D^R$ is just a version, with vector bundle coefficients, of the correspondence between sections $X \in \Gamma(A)$ and right-invariant vector fields $X^R$ on $G$). In a suitable sense $\tn{ind}(D)\in K_0(\Psi^{-\infty}(G))$ may be thought of as the class associated to this elliptic family. Indeed in the special case $G=M\times_B M$ is the Lie groupoid associated to a fibration $M\rightarrow B$ with compact fibers, $D$ is a family of Dirac operators on the fibers, and the K-class $\tn{ind}(D)$ is equivalent to the class associated to the family by Atiyah and Singer. More generally for $G=\tn{Hol}(\F)$, $D$ is a family of Dirac operators on the leaves of the foliation, and the K-class $\tn{ind}(D)$ was introduced and studied by Connes (cf. \cite{ConnesBook}). In the example $A={}^bTM$ initiated by Melrose (cf. \cite{melrose1993atiyah}), $D$ restricts to an ordinary Dirac operator on the complement $M\backslash Z$, where the latter carries a complete metric with cylindrical ends (more generally with geometry at infinity depending on the chosen variant of ${}^bTM$).

A bisection of $G$ is an embedded submanifold $\gamma \subset G$ such that the source and range maps $(r,s)\colon G\rightrightarrows M$ restrict to diffeomorphisms from $\gamma$ to $M$. The set $\tn{Bis}(G)$ of bisections forms a (Fr{\'e}chet or diffeological) group with group operation induced by the composition law in $G$. For the pair groupoid $\tn{Bis}(\tn{Pair}(M))\simeq \tn{Diff}(M)$ is the diffeomorphism group of $M$. Let $\sf{K}$ be a compact Lie group. We refer to a smooth group homomorphism $\sf{K} \rightarrow \tn{Bis}(G)$ as an action of $\sf{K}$ on $G$ by bisections. In this setting it is straight-forward to define an equivariant K-class $\tn{ind}(D)\in K_0^\sf{K}(\Psi^{-\infty}(G))$.

A trace $\tau$ on $\Psi^{-\infty}(G)$ is a $\bC$-linear functional that vanishes on commutators. In this article we only consider traces that factor as restriction to the unit space $M\subset G$ composed with a continuous $\bC$-linear functional. A trace determines an `equivariant trace pairing'
\[ \pair{\tau}{-}\colon K_0^\sf{K}(\Psi^{-\infty}(G))\rightarrow C^\infty(\sf{K})^\sf{K},\]
where $C^\infty(\sf{K})^\sf{K}$ denotes conjugation-invariant smooth functions. We already mentioned the canonical trace for $G=\tn{Pair}(M)$. In the case $G=\tn{Hol}(\F)$, an example of such a trace is provided by a transverse measure, cf. \cite{ConnesBook}. For the example associated to the $b$-tangent bundle ${}^bTM$ of a hypersurface or more general normal crossing divisor, one has a trace defined using Cauchy principal value integration that we describe in Section \ref{s:normalcrossing}.

Our main result is a fixed-point formula for the equivariant trace pairing.
\begin{theorem}
\label{t:fixedptformulaintro}
Let $G$ be a Lie groupoid over a closed unit space $M$ with even rank $n$ oriented metrised Lie algebroid $(A,h)$. Let $\sf{K}$ be a compact Lie group that acts on $G$ by bisections preserving the metric and orientation on $A$. Let $\tau \colon \Psi^{-\infty}(G)\rightarrow \bC$ be a continuous trace that factors through restriction to $M$. Let $(E,c,\nabla)$ be a $\sf{K}$-equivariant $\Cl(A^*)$-module with Clifford $A$-connection and let $D=c\circ \nabla$ be the corresponding $A$-Dirac operator. Let $\gamma \in \sf{K}$ and $j\colon M^\gamma \hookrightarrow M$ the inclusion. Then
\begin{equation} 
\label{e:fixedptformulaintro}
\pair{\tau}{\tn{ind}(D)}(\gamma)=\bigg\la \frac{j^*\tau}{|\nu|},\bigg(\frac{\Ahat(j^!A)\tn{\textbf{ch}}^{\gamma}(E/S)}{(2\pi \i)^{n_0/2}\i^{n_1/2}\tn{det}^{1/2}(1-\gamma_1 e^{-R_1})}\bigg)_{[n]}\bigg\ra.
\end{equation}
\end{theorem}
The characteristic forms are Lie algebroid analogues of the familiar ones, and our notation is similar to Berline-Getzler-Vergne \cite{BerlineGetzlerVergne}. The subscript `$1$' denotes quantities associated to the normal bundle $\dN_MM^\gamma$ to $M^\gamma=\{m\in M\mid \gamma m=m\}$, for example $R_1$ is its curvature and $n_1$ is its rank. The vector bundle $j^!A=A\times_{TM}TM^\gamma$ is the Lie algebroid pullback of $A$ to $M^\gamma$ (Proposition \ref{p:transversefixedpt} below shows that this is automatically a transverse pullback). The trace $\tau$ is defined in terms of a generalized section on $M$, and $j^*\tau$ denotes its pullback to $M^\gamma$, which is well-defined by wave-front set considerations.

In the classical setting $G=\tn{Pair}(M)$, $j^*\tau/|\nu|$ coincides with the functional $\Gamma(j^*\Lambda)\rightarrow \bC$ given by Berezin integration followed by integration over $M^\gamma$ with respect to the Riemannian volume density, and the formula in Theorem \ref{t:fixedptformulaintro} becomes precisely that given in \cite[Theorem 6.16]{BerlineGetzlerVergne}. Another special case is for $G$ the holonomy groupoid of a foliation of a closed manifold $M$, with trace $\tau$ given by a transverse measure. In this instance Theorem \ref{t:fixedptformula} recovers Connes' \cite{connes1979theorie} index theorem for measured foliations in the non-equivariant case and the leafwise-isometric case of an equivariant generalization due to Heitsch-Lazarov \cite{heitsch1990lefschetz}.

In Section \ref{s:examples} we describe a new example where $A={}^bTM$ is the $b$-tangent bundle associated to a hypersurface in $M$ (more generally to a simple normal crossing divisor). The corresponding $A$-Dirac operator is not elliptic in the usual sense because its symbol degenerates in a prescribed way along the hypersurface. There is a further straightforward generalization to the case of a measured foliation transverse to the divisor. We also describe the analogue of the Atiyah-Hirzebruch vanishing theorem in our context.

We remark that in the study of $b$ (and related) geometry, initiated by Melrose and further developed by many authors, the focus is usually on the case where the hypersurfaces sit at the boundary of a compact manifold with corners. We emphasize that in the examples described here, the hypersurfaces are always embedded in a \emph{closed} manifold. The latter situation has been studied in Poisson geometry (cf. \cite{guillemin2014symplectic, gualtieri2017tropical, LLSS2}), where the b-tangent bundle is used to study Poisson structures that are symplectic everywhere except for controlled degeneracies along the hypersurfaces of the divisor.

Throughout we work with smooth objects and in the smooth category. To our knowledge it is not known in general whether the longitudinal heat kernel $\exp(-t\Delta^R)$, $\Delta^R=(D^R)^2$ is smooth transversely to the source fibers (cf. \cite[Conjecture 1.6]{bkso2014exponential}, \cite[Remark 4.8]{bohlen2018getzler}). To avoid this obstacle we instead work with an object which is essentially a parametrix for the heat operator $(\partial_t+\Delta)$. We give a relatively `low tech' construction of this operator that avoids the heat operator pseudodifferential calculus (almost); in brief we apply Borel summation to the formal perturbative solution of the heat equation (described in the textbooks \cite{roe1988index} or \cite{BerlineGetzlerVergne} for example) to obtain a convergent series which is a solution of the heat equation modulo an error which is smoothing and vanishes to infinite order at $t=0$. 

Part of the inspiration for this work came from a series of articles by Pflaum, Posthuma and Tang \cite{pflaum2014index, pflaum2015localized, pflaum2015transverse}. In these articles the authors studied longitudinal elliptic (not necessarily Dirac) operators on Lie groupoids $G$ (and more generally manifolds equipped with a proper co-compact $G$-action), and obtained formulas not just for traces but for the evaluation of a wider collection of cyclic cocycles on the index class. The methods are based on algebraic index theory and Fedosov-type deformation quantization, and are completely different from those used here. Our focus is more restricted. Nevertheless our Theorem \ref{t:fixedptformulaintro} generalizes a special case of the Pflaum-Posthuma-Tang theorem in two ways (i) we treat the equivariant case with an auxiliary action of a compact Lie group $\sf{K}$; (ii) we allow traces that are given by distributions on the unit space (this will be crucial for the example described in Section \ref{s:examples}).

The contents of the article are as follows. In Section 2 we summarize needed background on Lie groupoids and Lie algebroids, including Lie algebroid analogues of objects such as the Levi-Civita connection, Chern-Weil forms, and Dirac operators. We develop the basic properties of compact Lie group actions on Lie groupoids by bisections. In Section 3 we briefly recall the convolution algebra $\Psi^{-\infty}(G)$ of a Lie groupoid and the equivariant trace pairing. In Section 4 we construct a heat parametrix (`asymptotic heat kernel') $K_t$ using Borel summation of the formal solution as alluded to above. We construct a parametrix for $D$ using $K_t$ and deduce that the equivariant trace pairing $\pair{\tau}{\tn{ind}(D)}(\gamma)$ equals the constant term in the asymptotic expansion of $\tau^\gamma_s(K_t)$ as $t\rightarrow 0^+$. In Section 5 we develop the appropriate version of the Higson-Yi vector bundle incorporating the Getzler rescaling, which in this case lives over the deformation to the normal cone $\sbN_GM^\gamma$ for $M^\gamma$ inside of $G$. Also included is a rapid treatment of Schwartz functions on the deformation to the normal cone. Section 6 contains the calculation of the fixed-point contributions and the proof of Theorem \ref{t:fixedptformulaintro}. In a brief Section 7, we explain two modest extensions of the main theorem: to the case of non-Hausdorff Lie groupoids, and to proper actions of possibly non-compact Lie groups. Finally in Section 8 we describe examples (in particular $A={}^bTM$), and give the analogue of the Atiyah-Hirzebruch vanishing theorem.

\bigskip

\noindent \textbf{Notation}. Throughout $G$ denotes a Lie groupoid over a closed manifold $M$ with source $s$ and range $r$. The Lie algebroid of $G$ is $A=\ker(Ts)$ and $h$ is a fiber metric on $A$. $\sf{K}$ is a compact Lie group that acts on $G$ by bisections, and $\gamma$ denotes either an element of $\sf{K}$ or the corresponding bisection of $G$. The letter $\tau$ is used both to denote a trace on $\Psi^{-\infty}(G)$ and for the corresponding generalized section over $M$. The operator $D$ is an $A$-Dirac operator and $\Delta=D^2$. The `asymptotic heat kernel' or heat parametrix is denoted $K_t$ (defined in Section 4). Objects associated to the deformation to the normal cone $\sbN_GM^\gamma$ are denoted with black-board bold. The corresponding objects associated to the normal bundle $\dN_GM^\gamma$ are denoted with ordinary bold.

\bigskip

\noindent \textbf{Acknowledgments}. J. Sanchez is supported by NSF grants DMS-1952551, DMS-1952557.

\section{Geometric structures on Lie groupoids}\label{s:liegpd}
In this section we introduce basic definitions and notation for geometric structures on Lie groupoids and Lie algebroids that we shall need in the rest of the article. In particular we introduce $A$-Dirac operators for any metrised Lie algebroid $(A,h)$. We also develop the basic properties of compact Lie group actions on Lie groupoids by bisections.

\subsection{Background on Lie groupoids}
Let $(r,s)\colon G \rightrightarrows M$ be a (Hausdorff) Lie groupoid with compact unit space $M \subset G$ and inverse $\iota \colon G \rightarrow G$. Let $G^{(2)}\subset G^2$ denote the subset of composable pairs of arrows, i.e. pairs $(g_1,g_0)$ with $s(g_1)=r(g_0)$. There are three face maps 
\[ (\partial_1,\partial,\partial_0)\colon G^{(2)} \rightarrow G^3 \]
where $\partial_1$, $\partial_0$ are projection to the first and second factors respectively, and $\partial$ is the groupoid multiplication. The inverse and groupoid multiplication will also be written $\iota(g)=g^{-1}$, $\partial(g_1,g_0)=g_1g_0$.

Let $\rho \colon A=\ker(Ts)|_M \rightarrow M$ denote the Lie algebroid of $G$. There are canonical isomorphisms 
\[ \ker(Ts)|_M\simeq \dN_GM\simeq \ker(Tr)|_M,\] 
where $\dN_GM$ denotes the normal bundle to $M$ in $G$. Moreover
\[ s^*A\simeq \ker(Tr) \qquad r^*A\simeq \ker(Ts), \]
by left (resp. right) translation to the unit space; we shall use these isomorphisms frequently without comment. A section $X \in \Gamma(A)$ determines a corresponding left (resp. right) invariant vector field on $G$ denoted $X^L=s^*X$ (resp. $X^R=r^*X$). In terms of the action of the groupoid on itself by right (resp. left) multiplication,
\[  X^L(g)=-\frac{\d}{\d u}\bigg|_{u=0} g\ell(u)^{-1}, \qquad X^R(g)=\frac{\d}{\d u}\bigg|_{u=0} \ell(u)g, \]
where $\ell$ is any smooth curve in $s^{-1}(r(g))$ with $\ell(0)=r(g)$, $\ell'(0)=X(r(g))$.

The action of right and left invariant vector fields on $G$ are conveniently described in terms of a Lie algebroid
\[ \delta A=r^*A\oplus s^*A \rightarrow G, \]
generated by sections $X^R\oplus Y^L$, $X,Y \in \Gamma(A)$. Such sections will be used frequently, and we will often use the simplified notation
\[ X\oplus Y:=X^R\oplus Y^L \in \Gamma(\delta A).\]
The anchor map $\delta \rho \colon \delta A\rightarrow TG$ of $\delta A$ is given on generators by
\[ \delta \rho(X\oplus Y)=X^R-Y^L.\]
Note that $\delta \rho(X\oplus X)=X^R-X^L$ is generated by the adjoint action of $G$ on itself, and in particular is tangent to the unit space $M$.

The Lie algebroid $A$ has a universal enveloping algebra, the algebra generated by $C^\infty(M)$, $\Gamma(A)$ with relations $f\cdot X=fX$, $X\cdot f=fX+\rho(X)f$, $X\cdot Y-Y\cdot X=[X,Y]$. When $A=TM$ this recovers the algebra of scalar differential operators on $M$. The map $X\mapsto X^R$ between sections of $\Gamma(A)$ and right invariant vector fields along the $s$-fibers extends to an isomorphism $T \mapsto T^R$ from the universal enveloping algebra of $A$ to the algebra of $G$-equivariant families of scalar differential operators along the $s$-fibers.

\subsection{Splitting theorem for Lie algebroids}\label{s:splitting}
For further background, see for example \cite{meinrenkenliegroupoids} and references therein. Let $f\colon Q\rightarrow M$ be a smooth map such that the anchor map $\rho \colon A \rightarrow TM$ is transverse to the tangent map $Tf\colon TQ \rightarrow TM$. Then the fiber product
\[ f^!A=A\times_{TM}TQ \]
is a smooth Lie algebroid over $Q$ known as the \emph{pullback Lie algebroid}. As a special case when $j\colon Q \hookrightarrow M$ is an embedded submanifold such that $\rho$ is transverse to $TQ$, then the pullback Lie algebroid $j^!A$ is defined.

An \emph{Euler-like vector field} for a submanifold $j\colon Q\hookrightarrow M$ is a smooth vector field $\R$ defined on a neighborhood of $Q$, vanishing along $Q$, and such that its linearization $\dN(\R)\in \mf{X}(\dN_MQ)$ is the Euler vector field on the normal bundle. An \emph{Euler-like section of} $A$ for $Q$ is a smooth section $\E$ of $A$ defined on a neighborhood of $Q$, vanishing along $Q$, and such that $\rho(\E)$ is an Euler-like vector field for $Q$.

When $\rho$ is transverse to $TQ$, Euler-like sections exist. Moreover a choice of Euler-like section $\E$ determines an isomorphism of Lie algebroids
\[ A|_N \simeq p^!j^!A \]
where $N\supset Q$ is the tubular neighborhood determined by the Euler-like vector field $\rho(\E)$ and $p\colon N \rightarrow Q$ is the projection map of the tubular neighborhood embedding. This statement is known as the splitting theorem for Lie algebroids (cf. \cite[Theorem 4.1]{bursztyn2019splitting}). If the normal bundle is trivial so that $N\simeq V\times Q$ then the statement simplifies: $A|_N\simeq TV\times j^!A$. Of course since the normal bundle is locally trivial, one always has this description locally.

\subsection{Vector bundles, connections, curvature}
Let $E\rightarrow M$ be a vector bundle. An $A$-connection $\nabla$ on $E$ is a differential operator
\[ \nabla \colon \Gamma(E)\rightarrow \Gamma(A^*\otimes E) \]
satisfying the Leibniz rule
\[ \nabla_X (fs)=(\rho(X)f)s+f\nabla_X s, \qquad f \in C^\infty(M), X \in \Gamma(A), s\in \Gamma(E).\]
An $A$-connection on $E$ induces an $A$-connection on the dual $E^*$.
The curvature $F \in \Gamma(\wedge^2 A^*\otimes \End(E))$ of an $A$-connection is defined by the usual formula 
\[F(X,Y)=[\nabla_X,\nabla_Y]-\nabla_{[X,Y]},\qquad X,Y\in \Gamma(A).\]

An differential $A$-form is a smooth section of $\wedge A^*$. There is a de Rham $A$-differential $\d_A$ making $\Gamma(\wedge A^*)$ into a complex (graded by degree), and whose cohomology is the Lie algebroid cohomology of $A$. Given an invariant polynomial $p$ on $\End(\bC^r)$, $r=\tn{rank}(E)$, the Chern-Weil construction produces a $\d_A$-closed differential $A$-form 
\[ p(E)=p(E,\nabla):=p(F) \in \Gamma(\wedge A^*), \] 
whose $\d_A$-cohomology class is independent of the choice of connection. We will sometimes omit the connection $\nabla$ to simplify the notation. For example if $\gamma$ is an endomorphism of $E$ that preserves $\nabla$, then one has the $\gamma$-twisted $A$-Chern character 
\[ \tn{\textbf{ch}}^\gamma(E)=\tn{\textbf{ch}}^\gamma(E,\nabla)=\tr(\gamma\exp(-F)) \in \Gamma(\wedge A^*).\]

The pullback $r^*E$ is a $G$-equivariant vector bundle over $G$, where $G$ acts on itself by right multiplication. Under the isomorphism $r^*A\simeq \ker(Ts)$, an $A$-connection on $E$ pulls back to a $G$-equivariant family of ordinary connections on the restriction of $r^*E$ to each $s$ fiber of $G$, or in other words, to a $G$-equivariant partial connection on $r^*E$, defined on the vertical subbundle $\ker(Ts)\subset TG$ for the submersion $s$. Under this correspondence, $r^*F$ is the curvature of the partial connection. 

There is a straight-forward generalization of the universal enveloping algebra of $A$ incorporating $\End(E)$-coefficients. As before an element $T$ of this algebra gives rise to a smooth $G$-equivariant family $T^R$ of source-wise differential operators acting on $r^*E$, and this correspondence is one-one.

Given $E$ as above, let
\[ \delta E=r^*E\otimes s^*E^* \rightarrow G.\]
Note that $\delta E$ is an example of a multiplicative vector bundle over $G$, in the sense that there is a natural smooth bundle morphism $\partial_1^*\delta E\otimes \partial_0^*\delta E\rightarrow \delta E$ covering the groupoid multiplication $\partial \colon G^{(2)}\rightarrow G$. Define a $\delta A$-connection on $\delta E$ by
\begin{equation} 
\label{e:deltaAconn}
\nabla_{X\oplus Y}(r^*\sigma_1\otimes s^*\sigma_2)=r^*\nabla_X\sigma_1\otimes s^*\sigma_2+r^*\sigma_1\otimes s^*\nabla_Y\sigma_2,
\end{equation}
on generators, and then extending to arbitrary sections by $C^\infty(G)$-linearity in the $\delta A$-slot and the Leibniz rule in $\delta E$.

\subsection{Metrics}
Let $h$ be a fiber metric on $A$. Under the isomorphism $r^*A=\ker(Ts)$, the pullback $r^*h$ defines a family of Riemannian metrics on the source fibers of $G$. The metric determines an $A$-connection $\nabla$ on $A$ itself, given by the Koszul formula \cite[Chapter 1]{BerlineGetzlerVergne}. We refer to this connection as the Levi-Civita $A$-connection. The induced $G$-equivariant family of connections on the source fibers are the usual Levi-Civita connections for the $G$-equivariant family $r^*h$ of Riemannian metrics on the $s$ fibers. The curvature of $\nabla$ is denoted
\[ R \in \Gamma(\wedge^2 A^*\otimes \mf{so}(A)).\]
Its pullback $r^*R$ identifies with the family of Riemann curvature tensors of the Levi-Civita connections along the $s$ fibers. In particular $R$ has the familiar symmetries of the Riemann curvature tensor, including the well-known identity
\begin{equation}
\label{e:swapsymmetry}
h(W,R(X,Y)Z)=h(X,R(W,Z)Y).
\end{equation}
Later on it is convenient to use notation similar to \cite{BerlineGetzlerVergne} for contractions of the $\mf{so}(A)$ indices of the tensor $R$ using the metric $h$. For $X,Y \in A$ we define
\[ R|Y\ra=h(R(-,-)Y,-) \in \Gamma(\wedge^2A^*\otimes A^*), \quad \la X|R|Y\ra=h(R(-,-)Y,X) \in \Gamma(\wedge^2A^*).\]
For example with this notation \eqref{e:swapsymmetry} becomes
\begin{equation} 
\label{e:swapsymmetry2}
h(-,R(X,Y)-)=\la X|R|Y\ra \in \Gamma(\wedge^2A^*).
\end{equation}

An invariant analytic germ $f$ at the origin in $\mf{so}(n)$ determines, via the Chern-Weil construction, an $A$-differential form
\[ f(R) \in \Gamma(\wedge A^*).\]
For example the $\Ahat$ form $\Ahat(A)$ for $(A,h)$ is associated to the function $\tn{det}^{1/2}(\frac{x/2}{\sinh(x/2)})$.

The metric also determines an exponential map
\[ \exp^h\colon U \rightarrow G \]
where $U\subset A$ is a suitable tubular neighborhood of the $0$-section of some radius $\kappa>0$. We mention the following property of the $s$-fibers for completeness, although we shall not need it.
\begin{proposition}
\label{p:bdedgeom}
The induced metrics on the $s$ fibers of $G$ have bounded geometry.
\end{proposition}
\begin{proof}
The boundedness of the curvature tensor and its covariant derivatives is clear because these tensors are pulled back under $r$ from sections of tensor powers of $A,A^*$ defined over the compact manifold $M$. Right translation of the geodesic balls $\exp^h(U)\cap s^{-1}(x)$, $x \in M$ shows that the injectivity radius is bounded below by $\kappa$.
\end{proof}

\subsection{Clifford modules and $A$-Dirac operators}
Let $(E=E^+\oplus E^-,c\colon \bCl(A^*)\rightarrow \End(E))$ be a $\bZ_2$-graded $\bCl(A^*)$-module bundle. Let $\nabla$ be a Clifford $A$-connection on $E$, i.e. an $A$-connection $\nabla$ on $E$ preserving the $\bZ_2$-grading and satisfying
\[ [\nabla_X, c(\alpha)]=c(\nabla_X \alpha), \qquad \forall X \in \Gamma(A), \alpha \in \Gamma(A^*) \]
where on the right $\nabla_X\alpha$ denotes the Levi-Civita $A$-connection (we use the same symbol $\nabla$ for both the Clifford $A$-connection on $E$ and the Levi-Civita $A$-connection on $A$). 

The usual construction of Clifford connections for $A=TM$ generalizes immediately to this setting. Indeed, locally on $M$, $E$ can be decomposed into a graded tensor product $S\hat{\otimes} W$, where $S$ is an irreducible $\bCl(A^*)$-module associated to a $\Spin(n)$ reduction of structure group of the oriented orthonormal frame bundle of $A$. The Levi-Civita $A$-connection induces a canonical spin $A$-connection on $S$. A local Clifford $A$-connection on $E$ is obtained by tensoring the spin $A$-connection on $S$ with an arbitrary $\bZ_2$-grading compatible connection on $W$. One obtains a global Clifford $A$-connection by patching together local Clifford $A$-connections using a partition of unity on $M$.

The endomorphism bundle
\[ \End(E)\simeq \bCl(A^*)\otimes \End_{\Cl}(E) \]
where $\End_{\Cl}(E)$ denotes endomorphisms commuting with the Clifford action. With respect to this tensor product decomposition, the curvature $F$ of $(E,\nabla)$ is a sum
\begin{equation} 
\label{e:barF}
F=R\otimes 1+1\otimes F^{E/S}
\end{equation}
where we identify $R$ with a section of $\wedge^2A^*\otimes \Cl^{[2]}(A^*)$ using the standard isomorphism $\mf{so}(A)\simeq \Cl^{[2]}(A^*)$.

The data $(E,c,\nabla)$ is used to construct an $A$-Dirac operator $D \colon \Gamma(E)\rightarrow \Gamma(E)$ given by the usual composition
\[ \Gamma(E)\xrightarrow{\nabla}\Gamma(A^*\otimes E)\xrightarrow{c} \Gamma(E).\]
This is a first-order differential operator. Its $A$-symbol $\sigma^A_D(\xi)=c(\xi)$, $\xi \in A^*$, is invertible away from the zero section in $A^*$. On the other hand the ordinary symbol of $D$ is
\[ \sigma_D(\d f)=[D,f]=c(\rho^\top(\d f)) \]
where $\rho^\top \colon T^*M\rightarrow A^*$ is the adjoint morphism. Thus $D$ is elliptic in the ordinary sense if and only if $\rho^\top$ is injective if and only if $\rho$ is surjective.

Using the identification $r^*A\simeq \ker(Ts)$, the pullback $r^*(E,c,\nabla)$ is a $G$-equivariant family of Clifford modules with connection for the $s$-fibers of $G$ equipped with the pullback metric $r^*h$. This is the Clifford data for the $G$-equivariant family of (ordinary) Dirac operators $D^R$ along the $s$-fibers.

The square $D^2=\Delta$ is given by the Lichnerowicz formula:
\begin{equation}
\label{e:Lichnerowicz}
\Delta=D^2=-\Tr_h(\nabla^2)+\frac{\kappa(R)}{4}+c(F^{E/S}).
\end{equation}
where $\kappa(R) \in C^\infty(M)$ is the scalar constructed from the tensor $R$ in the same way as the scalar curvature in the case $A=TM$, $F^{E/S}\in \Gamma(\wedge^2A^*_\bC\otimes \End_{\Cl}(E))$ is viewed as a section of $\bCl^{[2]}(A^*)\otimes \End_{\Cl}(E)$ via the isomorphism $\wedge A^*\simeq \Cl(A^*)$, and $\Tr_h(\nabla^2)$ denotes the contraction of the operator 
\[ \Gamma(E)\xrightarrow{\nabla} \Gamma(A^*\otimes E)\xrightarrow{\nabla\otimes 1+1\otimes \nabla} \Gamma(A^*\otimes A^*\otimes E) \]
with the inverse of the metric $h^{-1} \in \Gamma(\tn{Sym}^2(A))$. By the same calculation as in the case $A=TM$ (cf. \cite[Chapter 3]{BerlineGetzlerVergne}), one has the following formula for the operator $-\Tr_h(\nabla^2)$ in terms of a local orthonormal frame $e_1,...,e_n$ of $(A,h)$:
\begin{equation}
\label{e:Bochner}
-\Tr_h(\nabla^2)=-\sum_{i=1}^n (\nabla_{e_i}^2-\nabla_{\nabla_{e_i} e_i}).
\end{equation}
Equation \eqref{e:Lichnerowicz} may be shown by repeating the proof in the case $A=TM$ (cf. \cite[Chapter 3]{BerlineGetzlerVergne}); alternatively, working with the $G$-equivariant family $D^R$, \eqref{e:Lichnerowicz} is an immediate consequence of its classical counterpart.

\subsection{Group actions}\label{s:groupactions}
A \emph{bisection} of $G$ is an embedded submanifold $\gamma$ of $G$ that is a section of both $r,s$; equivalently $r_\gamma:=r|_{\gamma}$ and $s_\gamma:=s|_{\gamma}$ are diffeomorphisms from $\gamma$ to $M$. The space of bisections of $G$ forms a group denoted $\tn{Bis}(G)$, where the group multiplication is induced by multiplication of arrows in $G$. $\tn{Bis}(G)$ is typically infinite dimensional, for example, if $G=\tn{Pair}(M)=M\times M$, $\tn{Bis}(G)\simeq\tn{Diff}(M)$ is isomorphic to the diffeomorphism group of $M$, via the map taking a diffeomorphism to its graph in $M\times M$. For any $G$, there is a group homomorphism
\[ \tn{Bis}(G)\rightarrow \tn{Diff}(M), \qquad \gamma \mapsto \tilde{\gamma}:=r_\gamma\circ s_\gamma^{-1}.\]

Bisections act on $G$ by left and right multiplication: given $g \in G$, let $\gamma g$ denote the composition of $g$ with the unique arrow contained in $\gamma$ having source equal to $r(g)$, and similarly let $g\gamma$ denote the composition of $g$ with the unique arrow in $\gamma$ having range equal to $s(g)$. The conjugation action $g\mapsto \Ad_\gamma(g)=\gamma g \gamma^{-1}$ is a Lie groupoid automorphism, and restricts to the diffeomorphism $\tilde{\gamma}$ on the unit space $M$. In the example $G=\tn{Pair}(M)=M\times M$, the two actions are those of $\tn{Diff}(M)$ on the first and second factors respectively, and conjugation is the diagonal action. The action of $\tn{Bis}(G)$ on $G$ by left multiplication commutes with the right action of $G$ on itself, and thus any $\gamma \in \tn{Bis}(G)$ determines, via left multiplication, a $G$-equivariant family of diffeomorphisms of the $s$-fibers. Right multiplication by $\gamma$ does not preserve $s$-fibers, but maps $s$-fibers diffeomorphically to other $s$-fibers. For a function $f \in C^\infty(G)$ and $\gamma \in \tn{Bis}(G)$ let
\[ \gamma f=f^\gamma \in C^\infty(G), \qquad f^\gamma(g)=f(\gamma^{-1}g) \]
be the pullback action for left multiplication. There is a similar pullback (left) action for right multiplication denoted $f\mapsto f\gamma$ (we shall use this less).

More generally let $E \rightarrow M$ be a vector bundle. Let $\tn{Bis}(G,E)$ be the set of pairs $(\gamma,\gamma^E)$ consisting of a bisection $\gamma\subset G$ together with an isomorphism of vector bundles
\[ \gamma^E \colon s_\gamma^*E\xrightarrow{\sim} r_\gamma^*E.\]
Composition of morphisms makes $\tn{Bis}(G,E)$ into a group. The additional data of an isomorphism $\gamma^E \in \Gamma(\Hom(s_\gamma^*E,r_\gamma^*E))=\Gamma(\delta E|_\gamma)$ over $\gamma$ is equivalent to specifying an automorphism, also denoted $\gamma^E$, of $E$ covering the diffeomorphism $\tilde{\gamma}$. There is a left action of $\tn{Bis}(G,E)$ on $r^*E$ compatible with the natural forgetful map $\tn{Bis}(G,E)\rightarrow \tn{Bis}(G)$ and the action of $\tn{Bis}(G)$ on $G$ by left multiplication, and commuting with the $G$-action on $r^*E$ covering right multiplication. By taking adjoints there is a similar right action of $\tn{Bis}(G,E)$ on $s^*E^*$. The bundle $\delta E=r^*E\otimes s^*E^*$ thus carries actions of $\tn{Bis}(G,E)$ on both the left and the right. For a section $f \in \Gamma(\delta E)$ and $(\gamma,\gamma^E)\in \tn{Bis}(G,E)$ let
\[ \gamma f=f^\gamma \in \Gamma(\delta E), \qquad f^\gamma(g)=\gamma^Ef(\gamma^{-1}g), \]
be the pullback action for left multiplication. There is a similar pullback action for right multiplication denoted $f\mapsto f\gamma$ (we shall use this less).

As a special case, any $\gamma \in \tn{Bis}(G)$ prolongs to an element $(\gamma,\gamma^A)\in \tn{Bis}(G,A)$, where the corresponding automorphism of $A$ covering $\tilde{\gamma}$ is obtained by applying the Lie functor to $\Ad_\gamma$. From the definitions one has that if $f \in C^\infty(G)$, $X,Y\in \Gamma(A)$, then $(X^Rf)^\gamma=(\gamma^A X)^R f^\gamma$, $(Y^Lf)^\gamma=Y^Lf^\gamma$, or more succinctly
\begin{equation} 
\label{e:vectgamma}
\big(\delta \rho(X\oplus Y) f\big)^\gamma=\delta \rho(\gamma^AX\oplus Y) f^\gamma
\end{equation}
where recall $\delta \rho(X\oplus Y)=X^R-Y^L$ is the anchor map for $\delta A$.

An element $(\gamma,\gamma^E)\in \tn{Bis}(G,E)$ is said to preserve a geometric structure on $E$ (where geometric structure could be a metric, orientation, $A$-connection, Clifford module structure, etc.) if the corresponding automorphism of $E$ covering $\tilde{\gamma}$ preserves the geometric structure (for $A$-connections and Clifford module structures the prolongation $(\gamma,\gamma^A)\in \tn{Bis}(G,A)$ is involved as well). Assuming $(\gamma,\gamma^E)$ preserves an $A$-connection $\nabla$ on $E$, equation \eqref{e:vectgamma} generalizes:
\begin{equation}
\label{e:connvectgamma}
\big(\nabla_{X\oplus Y} f\big)^\gamma=\nabla_{\gamma^AX\oplus Y} f^\gamma,
\end{equation}
where $\nabla_{X\oplus Y}$ is the induced $\delta A$-connection \eqref{e:deltaAconn} on $\delta E$. This can be checked on generators of $\delta E$ using \eqref{e:deltaAconn}. Alternatively the connection $\nabla$ determines a lift of $X^R$ (resp. $Y^L$) to an infinitesimal automorphism $X_E^R$ (resp. $Y_E^L$) of $\delta E$ that exponentiates to a 1-parameter subgroup of automorphisms given by left (resp. right) multiplication by a 1-parameter subgroup in $\tn{Bis}(G,E)$. Then \eqref{e:connvectgamma} is obtained by differentiating $\gamma\exp(tX_E^R)f\exp(-tY_E^L)=\exp(t\Ad_{\gamma}X_E^R)\hat{\gamma}f\exp(-tY_E^L)$ and using $\Ad_{\gamma}X_E^R=(\gamma^A X)^R_E$.

Let $\sf{K}$ be a Lie group. An \emph{action of $\sf{K}$ on $G$ by bisections} is a group homomorphism
\begin{equation} 
\label{e:grphom}
\sf{K} \rightarrow \tn{Bis}(G), 
\end{equation}
such that the map
\[ \sf{K} \times M \rightarrow G, \qquad (\gamma,m)\mapsto \gamma m \]
is smooth, where $\gamma m$ is short-hand for the action of the bisection corresponding to $\gamma$ under \eqref{e:grphom} on $m \in M\subset G$ by left multiplication. An action of $\sf{K}$ on $G$ by bisections determines, via the action of a bisection on $G$ by left multiplication, a left action of $\sf{K}$ on the manifold $G$,
\[ \sf{K}\times G \rightarrow G, \qquad (\gamma,g)\mapsto \gamma g. \]
An action of $\sf{K}$ on $(G,E)$ by bisections is defined similarly. A metric/orientation/$A$-connection/Clifford module structure on $E$ is $\sf{K}$-equivariant if each $(\gamma,\gamma^E)$ preserves the metric/orientation/$A$-connection/Clifford module structure. An action of $\sf{K}$ on $G$ by bisections induces a smooth family of $\sf{K}$-actions on the $s$ fibers of $G$, equivariant for the action of $G$ on itself by right multiplication. $\sf{K}$-equivariance of geometric structures in the sense above translates into the usual notion of $\sf{K}$-equivariance along each of the $s$ fibers. 

For $\gamma \in \sf{K}$, let
\begin{equation}
\label{e:Mgamma}
M^\gamma=\{m \in M\mid \gamma m=m\}=\{m \in M\mid m\gamma=m\}=\gamma\cap M, \qquad j \colon M^\gamma \hookrightarrow M.
\end{equation}
Viewing $G$ as a $\sf{K}$-space where $\gamma \in \sf{K}$ acts as left multiplication by the corresponding bisection, we also have
\begin{equation} 
\label{e:Mgamma2}
M^\gamma=G^\gamma \cap M 
\end{equation}
where $G^\gamma$ denotes the $\gamma$-fixed point locus. On the other hand one has the diffeomorphism $\tilde{\gamma}=\Ad_\gamma|_M$ of $M$; its fixed-point set $M^{\tilde{\gamma}}$ is the set of $m \in M$ such that $\gamma m$ belongs to the isotropy group $G_m$ of $m$. Since $M^\gamma$ is the set of $m \in M$ such that $\gamma m=m$ is the identity element of that isotropy group, $M^\gamma \subset M^{\tilde{\gamma}}$. 
\begin{example}
In case $G=\tn{Pair}(M)$ where $M$ carries an $H$ action, $M^\gamma=M^{\tilde{\gamma}}\subset M \subset M\times M$ is the fixed-point subset for the action of $\gamma$ on $M$, embedded in the diagonal of $M\times M$.
\end{example}
\noindent In general $M^\gamma$ can be a proper subset of $M^{\tilde{\gamma}}$; we give a number of examples of this below.
\begin{example}
Suppose $G=M\times \sf{K}$ is the action Lie groupoid for a smooth action of $\sf{K}$ on $M$, and $\sf{K} \rightarrow \tn{Bis}(G)$ sends $\gamma$ to the bisection $M\times \{\gamma\}$. Then for $\gamma \ne 1$, $M^\gamma=\emptyset$ while $M^{\tilde{\gamma}}$ is the fixed point set for the action of $\gamma$ on $M$. 
\end{example}
\begin{example}
Consider the mapping torus of the disk $\bD^2$ and the map $\gamma$ given by rotation by $\pi/3$. The mapping torus carries a natural foliation $\F$ such that the intersections of the leaves and the transversal $\bD^2$ are the orbits of the $\sf{K}=\bZ_3$ action. Let $G$ be the holonomy groupoid of $\F$. Then $M^{\tilde{\gamma}}$ is a single leaf, while $M^\gamma$ is empty since the latter leaf has non-trivial holonomy.
\end{example}
\begin{example}
Consider the unit circle $M=S^1 \subset \bR^2$ and the hypersurface $Z=\{(1,0)\}$. The b-tangent bundle ${}^bTM$ admits an integration such that the isotropy group of $(1,0)$ is isomorphic to $\bR^\times$. Consider the $\sf{K}=\bZ_2$-action given by reflection across the $x$-axis. Then $M^{\tilde{\gamma}}=\{(1,0),(-1,0)\}$, while $M^\gamma=\{(-1,0)\}$ since the corresponding bisection equals $-1 \in \bR^\times$ in the source fiber over $(1,0)$.
\end{example}
\begin{example}
Consider $M=\bR P^2$ and the $\bR P^1$ at infinity $Z \subset \bR P^2$. The b-tangent bundle ${}^bTM$ admits an integration such that the isotropy group over points in $Z$ is $\bR^\times$. View $\bR P^2$ as $\bR^2 \sqcup Z$ and consider the diffeomorphism $\gamma$ given by rotation by $\pi$ about the origin. Similar to the previous example $M^{\tilde{\gamma}}=Z\cup \{(0,0)\}$ while $M^\gamma=\{(0,0)\}$ since the corresponding bisection is not the identity along $Z$. In this case $\gamma$ occurs in the smooth family of bisections given by rotation by $\theta \in [0,\pi]$.
\end{example}

%In the next two propositions we shall use the following properties of compact Lie group actions: let $\sf{K}$ be a compact Lie group acting smoothly on a manifold $X$ and let $\gamma \in \sf{K}$, then
%\begin{enumerate}
%\item the fixed point set $X^\gamma$ is smooth;
%\item if $x \in X^\gamma$ and $V\subset T_xX$ is a subspace such that the linearized action $T_x\gamma$ fixes $V$ point-wise, then $V\subset T_xX^\gamma$.
%\end{enumerate}
\begin{proposition}
\label{p:fixedpoint}
Let $\sf{K}$ be a compact Lie group acting on $G$ by bisections and let $\gamma \in \sf{K}$. Then $G^\gamma=r^{-1}(M^\gamma)$ and $M^\gamma$ is a smooth union of connected components of the fixed-point set $M^{\tilde{\gamma}}$ of the diffeomorphism $\tilde{\gamma}$. The induced map $\dN r \colon \dN_GG^\gamma|_{M^\gamma} \rightarrow \dN_MM^\gamma$ is an isomorphism of vector bundles, intertwining the linearized actions of $\gamma$ and $\tilde{\gamma}$ on the normal bundles. These linearized actions have no non-zero fixed vectors.
\end{proposition}
\begin{proof}
Being fixed-point loci of compact Lie group actions, $G^\gamma$ and $M^{\tilde{\gamma}}$ are smooth. Note that
\[ \gamma g=g \Leftrightarrow \gamma r(g) g=r(g)g \Leftrightarrow \gamma r(g)=r(g) \]
since $r(g)$ is a unit for the groupoid multiplication. Thus $g \in G^\gamma \Leftrightarrow r(g) \in M^\gamma$, or in other words $G^\gamma=r^{-1}(M^\gamma)$ is a union of $r$-fibers. As $r$ fibers are transverse to $M$, the intersection $M^\gamma=G^\gamma\cap M$ is transverse and hence smooth.

Transversality, proved above, implies that the normal bundle functor $\dN$ applied to $r$ induces an isomorphism of vector bundles
\begin{equation} 
\label{e:Nriso}
\dN r \colon \dN_GG^\gamma|_{M^\gamma} \rightarrow \dN_MM^\gamma.
\end{equation}
Since $\sf{K}$ is compact, the linearized action
\[ \dN \gamma \colon \dN_GG^\gamma \rightarrow \dN_GG^\gamma \]
has no non-zero fixed vectors, and hence by \eqref{e:Nriso}, neither does
\[ \dN r \circ \dN \gamma \circ (\dN r)^{-1} \colon \dN_MM^\gamma \rightarrow \dN_MM^\gamma.\]
But for any $m \in M$,
\[ r\circ \gamma \circ r^{-1}(m)=r(\gamma m)=r_\gamma \circ s_\gamma^{-1}(m)=\tilde{\gamma}(m). \]
Hence
\[ \dN r \circ \dN \gamma \circ (\dN r)^{-1}=\dN \tilde{\gamma} \colon \dN_MM^\gamma \rightarrow\dN_MM^\gamma \]
has no non-zero fixed vectors, and it follows that $M^\gamma$ is an open subset of $M^{\tilde{\gamma}}$. Since $M^\gamma$ is also closed, it must be a union of connected components of $M^{\tilde{\gamma}}$.
\end{proof}

\begin{proposition}
\label{p:transversefixedpt}
Let $\sf{K}$ be a compact Lie group acting on $G$ by bisections and let $\gamma \in \sf{K}$. Then $\rho \colon A \rightarrow TM$ is transverse to $TM^\gamma$.
\end{proposition}
\begin{proof}
Let $m \in M$ have isotropy group $G_m=s^{-1}(m)\cap r^{-1}(m)$, and let $O\subset M$ be the orbit of $G$ through $m$. Recall (cf. \cite[p.56]{fernandes2015normal}) that the normal representation of $G_m$ on $\dN_MO|_m$ is defined as follows. For $g \in G_m$ and $[v] \in \dN_MO|_m$ with representative $v \in T_mM$, 
\[ g\cdot [v]=[Tr(\hat{v})], \quad \text{where} \quad \hat{v}\in T_gG, \quad Ts(\hat{v})=v,\]
and $Tr,Ts$ are the derivatives of $r,s$ respectively (one checks that the result is independent of the choice of lift $\hat{v}$ of $v$). If $\gamma$ is a bisection such that $\{g\}=\gamma \cap G_m$, then $\hat{v}=Ts_\gamma^{-1}(v)$ is a lift, hence
\begin{equation} 
\label{e:normalrep}
g\cdot [v]=[Tr(\hat{v})]=[T(r_\gamma\circ s_\gamma^{-1})(v)]=\dN\tilde{\gamma}([v]).
\end{equation}
Supposing in addition that $m \in M^\gamma$ so that $\gamma m=m=m\gamma$, then $g \in \gamma \cap G_m$ is the identity element of the isotropy group, and in particular, it acts trivially in the normal representation. Thus in this case equation \eqref{e:normalrep} implies $\dN\tilde{\gamma}|_m \in \End(\dN_MO|_m)$ is the identity. Since $\sf{K}$ is compact, the action of $T\tilde{\gamma}|_m$ on $T_mM$ is completely reducible, and hence there is a subspace $V_m\subset T_mM$ fixed by $T\tilde{\gamma}|_m$ that is mapped isomorphically to $\dN_MO|_m$ by the quotient map $T_mM\rightarrow \dN_MO|_m$. By linearizability of compact Lie group actions, $V_m\subset T_mM^{\tilde{\gamma}}$. It follows that $T_mO$, $T_mM^{\tilde{\gamma}}$ are transverse. By Proposition \ref{p:fixedpoint}, $M^\gamma$ is a union of connected components of $M^{\tilde{\gamma}}$, and hence $M^\gamma$ is transverse to $O$.
\end{proof}
Proposition \ref{p:transversefixedpt} implies that the pullback Lie algebroid $j^!A\rightarrow M^\gamma$ exists. Recall that the action of $\sf{K}$ on $G$ by bisections corresponds to a $G$-equivariant family of $\sf{K}$-actions on the $s$-fibers of $G$. The submanifold $G^\gamma\subset G$ is the union of all the $s$-fiber $\gamma$-fixed point loci. Under the identification $r^*A\simeq \ker(Ts)$ given by right translation, the $G$-equivariant subbundle
\[ (r|_{G^\gamma})^*(j^!A)\rightarrow r^{-1}(M^\gamma)=G^\gamma \]
identifies with the kernel of the derivative of $s|_{G^\gamma}\colon G^\gamma \rightarrow M$. Equivalently the restriction of $(r|_{G^\gamma})^*(j^!A)$ to each $s$-fiber is canonically identified with the tangent bundle of the $\gamma$-fixed point locus of that $s$-fiber.

By transversality, and using the orthogonal splitting determined by the metric $h$, one obtains an isomorphism of vector bundles
\begin{equation} 
\label{e:jsplit}
A|_{M^\gamma}=j^!A \oplus (j^!A)^\perp\simeq j^!A\oplus \dN_MM^\gamma.
\end{equation}
The isomorphism $\gamma^A \colon s_\gamma^*A\rightarrow r_\gamma^*A$ of vector bundles over $\gamma$ restricts to a bundle endomorphism $\gamma^A|_{M^\gamma} \in \End(A|_{M^\gamma})$ along the submanifold $M^\gamma=\gamma \cap M$. Assuming $h$ is $\sf{K}$-invariant, the splitting \eqref{e:jsplit} is preserved by $\gamma^A|_{M^\gamma}$,
\begin{equation}
\label{e:jsplit2}
\gamma^A|_{M^\gamma}=\gamma_0\oplus \gamma_1,
\end{equation}
where $\gamma_0\in \Gamma(\End(j^!A))$, $\gamma_1 \in \Gamma(\End(\dN_MM^\gamma))$.
\begin{proposition}
\label{p:gammaAaction}
With respect to the splitting \eqref{e:jsplit}, \eqref{e:jsplit2}, the action of $\gamma^A|_{M^\gamma}$ is
\[ \gamma_0=\id, \qquad \gamma_1=\dN\tilde{\gamma}.\]
In particular $\gamma_1$ has no non-zero fixed vector.
\end{proposition}
\begin{proof}
Recall $\gamma^A$ results from applying the Lie functor to the Lie groupoid automorphism $\Ad_\gamma$. It follows that $\gamma^A|_{M^\gamma}$ is the restriction of the derivative $T\Ad_\gamma$ to $A|_{M^\gamma}\subset TG$. The range $r \colon G \rightarrow M$ intertwines $\Ad_\gamma$ with $\tilde{\gamma}$, and consequently the anchor map $\rho=Tr|_A$ intertwines $T\Ad_\gamma$ with $T\tilde{\gamma}$. Since $\rho$ restricts to an isomorphism $(j^!A)^\perp \xrightarrow{\sim}\dN_MM^\gamma$, we deduce that $\gamma_1=\dN\tilde{\gamma}$ under the identification. That there are no non-zero fixed vectors then follows from Proposition \ref{p:fixedpoint}.

Let $m \in M^\gamma$ and let $v \in j^!A_m$. Since $\rho(v)=Tr(v)\in T_mM^\gamma$, we may choose a smooth curve $g(t)$ with image contained in $s^{-1}(m)\cap r^{-1}(M^\gamma)=s^{-1}(m)\cap G^\gamma$ such that $g'(0)=v$. Then 
\[ \gamma^Av=\frac{\d}{\d t}\bigg|_{t=0} \gamma g(t)\gamma^{-1}. \]
But $g(t)\gamma^{-1}=g(t)$ since $s(g(t))=m$ and $m\gamma^{-1}=m$ as $m \in M^\gamma$. Likewise $\gamma g(t)=g(t)$ since $g(t)\subset G^\gamma$ is contained in the fixed point locus for left multiplication by $\gamma$. Thus $\gamma^Av=v$ proving $\gamma_0=\id$.
\end{proof}

\begin{corollary}
Let $A$ be oriented with metric $h$, and let $\sf{K}$ be a compact Lie group that acts on $G$ by bisections preserving $h$ and the orientation. Then the codimension of $M^\gamma$ is even.
\end{corollary}
\begin{proof}
By assumption $\gamma^A|_{M^\gamma}=\id\oplus \gamma_1$ preserves the orientation and metric $h$, thus $\gamma_1$ is an orientation-preserving isometry of the subbundle $(j^!A)^\perp\simeq\dN_MM^\gamma$. Since $\gamma_1$ has no non-zero fixed vector, the possible eigenvalues are $-1$ and complex conjugate pairs $e^{\pm \i \theta}$. Then $\gamma_1$ preserves orientation if and only if $-1$ occurs with even multiplicity, which can only occur if $\dN_MM^\gamma$ has even rank.
\end{proof}

\section{Longitudinal index theory}
There is an algebra $\Psi^\infty(G)$ of $G$-pseudodifferential operators. Elliptic pseudodifferential operators yield classes in the K-theory of the smoothing operators $\Psi^{-\infty}(G)$. Traces give rise to maps from K-theory to $\bC$. By longitudinal index theory we mean the study of the pairing of traces with the classes of elliptic $G$-$\Psi$DO. In this section we briefly recall these notions. 

\subsection{Convolution algebra of $G$}
We briefly recall the smooth convolution algebra of $G$ and the algebra of $G$-pseudodifferential operators, and refer the reader to \cite{nistor1999pseudodifferential, monthubert1997indice, Vassout2006, van2019groupoid} for further details.

Let 
\[ \Lambda=|\det(A^*)|\boxtimes \bC \]
denote the (complexified) bundle of densities for the vector bundle $A$, and let $\Lambda^{1/2}$ denote the corresponding bundle of half-densities. The smooth convolution algebra $\Psi^{-\infty}(G)$ consists of smooth compactly supported sections of 
\[ \delta \Lambda^{1/2}=r^*\Lambda^{1/2}\otimes s^*\Lambda^{1/2}.\]
The convolution product is
\begin{equation} 
\label{e:conv}
f_1\star f_0=\partial_\ast(\partial_1^*f_1 \otimes \partial_0^*f_0),
\end{equation}
where the push-forward involves integration of a density along the fibers of $\partial$; slightly less formally,
\[ f_1\star f_0(g)=\int_{g_1g_0=g} f_1(g_1)f_0(g_0).\]
Later on we shall frequently omit the `$\star$' notation for brevity, especially in equations involving multiple compositions and in cases where there is little risk of confusion. $\Psi^{-\infty}(G)$ also admits an involution $f\mapsto f^*=\ol{\iota^*f}$ given by pullback under the inversion $\iota$ and complex conjugation.

There is a larger algebra of compactly supported generalized sections of $\delta \Lambda^{1/2}$ having wavefront set contained in the conormal bundle to $M$, with the product given by the same formula \eqref{e:conv}, and containing $\Psi^{-\infty}(G)$ as a 2-sided ideal. The algebra of $G$-pseudodifferential operators $\Psi^\infty(G)$ is a subalgebra of the $M$-conormal generalized sections consisting of those which, in a tubular neighborhood $\exp^h(U)$ of $M \subset G$, are given by inverse Fourier transform of a smooth function on $A^*$ lying in a suitable symbol space---see the references mentioned at the beginning of this section for the detailed definition. The identity element $\delta_M \in \Psi^\infty(G)$ is a generalized section supported along $M\subset G$.

Choosing a trivialization of $\Lambda$, the universal enveloping algebra of $A$ is identified with the subalgebra of $G$-differential operators in $\Psi^\infty(G)$, via the map that sends $T$ to its kernel in $\Psi^\infty(G)$, which will be denoted by the same symbol $T$, and is determined implicitly by the equality
\begin{equation} 
\label{e:convdiff}
T\star f=T^R f, \quad \forall f \in \Psi^{-\infty}(G).
\end{equation}
If $A$ is equipped with a metric $h$, there is a preferred trivialization of $\Lambda$.

If $E\rightarrow M$ is a complex vector bundle then there is a similar convolution $^*$-algebra $\Psi^{-\infty}(G,E)$ with coefficients, consisting of smooth compactly supported sections of $\delta (E\otimes \Lambda^{1/2})$, where the convolution product involves an additional contraction of elements of $E,E^*$. One likewise defines $\Psi^\infty(G,E)$, the algebra of $G$-pseudodifferential operators on $E$. More generally still given two complex vector bundles $E_1,E_0\rightarrow M$ then one has a $\Psi^\infty(G,E_1)$-$\Psi^\infty(G,E_0)$ bimodule $\Psi^{\infty}(G,E_1,E_0)$ of $G$-pseudodifferential operators.

The space of composable arrows $G^{(2)}$ is diffeomorphic to the $s$-map fiber product $G\times_M G$ via the map 
\[ \id\times \iota \colon G\times_M G \xrightarrow{\sim}G^{(2)}.\]
Elements of $\Psi^{\infty}(G)$ (resp. $\Psi^{-\infty}(G)$) have a secondary interpretation as $G$-equivariant smooth families of pseudodifferential operators (resp. smoothing operators) on the $s$-fibers of $G$. This correspondence $f\mapsto k^f$ is given by pullback under the map
\[ \partial\circ (\id \times \iota) \colon G\times_M G \rightarrow G, \qquad (g_1,g_0)\mapsto g_1g_0^{-1}.\]
More informally $f(g)$ corresponds to the Schwartz kernel $k^f(g_1,g_0)=f(g_1g_0^{-1})$, which is $G$-invariant under the diagonal right multiplication: $k^f(g_1g,g_0g)=k^f(g_1,g_0)$. The correspondence takes convolution to composition of Schwartz kernels, the involution $f\mapsto f^*$ to the usual involution on Schwartz kernels, and takes the kernel of an element $T$ of the universal enveloping algebra to the kernel of the $G$-equivariant family $T^R$. In the case of $\Psi^\infty(G,E)$, the corresponding family of Schwartz kernels have coefficients in the vector bundle
\[ (\partial\circ(\id \times \iota))^*\delta (\Lambda^{1/2}\otimes E)=(r\circ \partial_1)^*(\Lambda^{1/2}\otimes E)\otimes (r\circ \partial_0)^*(\Lambda^{1/2}\otimes E^*).\]
Pulling back to $s^{-1}(m)\times s^{-1}(m)\subset G\times_M G$, the maps $\partial_1$, $\partial_0$ become the projection maps onto the two factors, hence the above bundle becomes the usual exterior tensor product
\[ r^*(\Lambda^{1/2}\otimes E)|_{s^{-1}(m)}\boxtimes r^*(\Lambda^{1/2}\otimes E^*)|_{s^{-1}(m)}\rightarrow s^{-1}(m)\times s^{-1}(m).\]

\subsection{Traces}
Let $\tau \colon \Psi^{-\infty}(G)\rightarrow \bC$ be a trace, i.e. a linear functional that vanishes on commutators. We shall assume throughout that $\tau$ is given by restriction to the unit space $M$ composed with a continuous linear functional $\Gamma(\Lambda)\rightarrow \bC$. We will denote this linear functional by the same symbol $\tau$, as it will be clear from context which is meant. One checks that the trace property is satisfied if and only if $\tau(s_*f)=\tau(r_*f)$ for all smooth compactly supported sections $f$ of $\delta \Lambda$.

The functional $\Gamma(\Lambda)\rightarrow \bC$ may be regarded as a generalized (distributional) section, also denoted $\tau$, of the complex line bundle $\Lambda_M\otimes \Lambda^{-1}$, where $\Lambda_M$ denotes the bundle of $1$-densities on $M$. The line bundle $\Lambda_M\otimes \Lambda^{-1}$ carries a natural representation of $G$, see for example \cite{crainic2020transverse, crainic2020measures, evens1996transverse}. The trace property implies that $\tau$ is invariant with respect to this representation \cite[Proposition 3.6]{crainic2020measures}. Put in other terms, there is a natural action of $\tn{Bis}(G)$ on $\Lambda_M\otimes \Lambda^{-1}$ covering the conjugation action on $M$, and $\tau$ is invariant under this action.
\begin{example}
Let $G=\tn{Pair}(M)$. There is a trace $\tau \colon \Psi^{-\infty}(G)\rightarrow \bC$ given by integration over the diagonal in $M\times M$. Invariance under $\tn{Bis}(G)\simeq \tn{Diff}(M)$ holds because integration of densities is invariant under arbitrary diffeomorphisms. The line bundle $\Lambda_M\otimes \Lambda^{-1}\simeq M\times \bC$ is canonically trivial, and $\tau$ corresponds to the constant section $1$.
\end{example}
Infinitesimally, invariance implies
\begin{equation} 
\label{e:PsiInv}
\L_X\tau=0, \qquad \forall X\in \Gamma(A),
\end{equation}
where $\L_X$ denotes the canonical representation of $A$ on $\Lambda^{-1}\otimes \Lambda_M$, given by
\[ \L_X(\Xi\otimes \mu)=\L_X\Xi\otimes \mu+\Xi\otimes \L_{\rho(X)}\mu.\]
That this is a representation amounts to the fact that the formula is $C^\infty(M)$-linear in $X$.

\subsection{Equivariant trace pairing}\label{s:equivtracepairing}
A trace $\tau$ defines a homomorphism
\[ \pair{\tau}{-}\colon K_0^\sf{K}(\Psi^{-\infty}(G))\rightarrow C^\infty(\sf{K})^\sf{K} \]
that we shall refer to as the equivariant trace pairing. For $f \in \Psi^{-\infty}(G)$ let $\tau^\gamma(f)=\tau(\gamma f)=\tau(f^\gamma)$. On an element of the form 
\[ [e]-[f] \in K_0^\sf{K}(\Psi^{-\infty}(G))=\ker(K_0^\sf{K}(\Psi^{-\infty}(G)^+)\rightarrow K_0^\sf{K}(\bC))\] 
where $e,f$ are $\sf{K}$-invariant idempotents in the unitization $\Psi^{-\infty}(G)^+=\Psi^{-\infty}(G)\oplus \bC 1$, it is given by
\[ \pair{\tau}{[e]-[f]}(\gamma)=\tau^\gamma(e)-\tau^\gamma(f), \]
where $\tau^\gamma$ is extended to the unitization by defining $\tau^\gamma(1)=0$.

Let $d \in \Psi^\infty(G)^\sf{K}$ be a $\sf{K}$-equivariant elliptic $G$-pseudodifferential operator. Such an operator defines a class
\[ [d] \in K^\sf{K}_0(\Psi^{-\infty}(G)).\]
From a conceptual point of view, the image of $d$ in the quotient $\Psi^\infty(G)/\Psi^{-\infty}(G)$ is invertible, hence defines an element in the algebraic $K^\sf{K}_1(\Psi^\infty(G)/\Psi^{-\infty}(G))$, and $[d]$ is the image of this class under the boundary map to $K^\sf{K}_0(\Psi^{-\infty}(G))$. One has a more explicit description of $[d]$ as follows. Let $p$ be a $\sf{K}$-equivariant parametrix for $d$, i.e. $pd=1-q$, $dp=1-r$ and $q,r \in \Psi^{-\infty}(G)$. Equivalently $p$ is a choice of representative for the inverse of the image of $d$ in the quotient algebra $\Psi^\infty(G)/\Psi^{-\infty}(G)$. Define the following idempotents in the algebra of $2\times 2$ matrices with entries in the unitization of $\Psi^{-\infty}(G)$:
\[ e=\m{q^2&q(1+q)p\\rd&1-r^2}, \quad e_0=\m{0&0\\0&1} \in M_2(\Psi^{-\infty}(G)^+).\] 
By definition
\[ [d]=[e]-[e_0]\in K^\sf{K}_0(\Psi^{-\infty}(G)),\]
and one finds that for $\gamma \in \sf{K}$,
\begin{equation} 
\label{e:tauind}
\pair{\tau}{[d]}(\gamma)=\tau^\gamma(e-e_0)=\tau(\gamma q^2)-\tau(\gamma r^2).
\end{equation}
For the reader's benefit we provide a proof of the standard result that the pairing is independent of the choice of parametrix, using the following lemma.

\begin{lemma}
Let $k \in \Psi^{-\infty}(G)$ and $d \in \Psi^\infty(G)$. Then $\tau([d,k])=0$.
\end{lemma}
\begin{proof}
Apply the Dixmier-Malliavin theorem for Lie groupoids \cite{francis2020, lescure2020evolution} to write $k$ as a finite sum of convolutions $ll'$ with $l \in C^\infty_c(G)$, $l' \in \Psi^{-\infty}(G)$. Then
\[ \tau([d,k])=\tau(dll'-ll'd)=\tau(l'dl-ll'd)=\tau([l'd,l])=0.\]
\end{proof}

\begin{proposition}
\label{p:tracepairingindependent}
The right hand side of \eqref{e:tauind} is independent of the choice of parametrix. 
\end{proposition}
\begin{proof}
One has
\[ q^2-r^2=2[p,d]+dpdp-pdpd.\]
Replacing $p$ with a different parametrix $p+k$, $k \in \Psi^{-\infty}(G)$ changes this by the operator
\[ 2[p,k]+(dpdk-pdkd)+(dkdp-kdpd)+(dkdk-kdkd) \]
and each of the bracketed terms vanishes under $\tau$ using the previous lemma. The same argument applies in the equivariant case since $\tau^\gamma$ is a trace on the subalgebra of operators commuting with $\gamma$.
\end{proof}

\begin{remark}
In the case with coefficients we may have for example $d\in \Psi^\infty(G,E_0,E_1)$ in which case $p\in \Psi^\infty(G,E_1,E_0)$, $q \in \Psi^{-\infty}(G,E_1)$, $r\in \Psi^{-\infty}(G,E_0)$.
\end{remark}

Let $A$ be equipped with a metric and orientation. Let $(E,c\colon \Cl(A^*)\rightarrow \End(E),\nabla)$ be a $\sf{K}$-equivariant $\Cl(A^*)$-module equipped with a $\gamma$-equivariant Clifford $A$-connection $\nabla$. Let $D$ be the corresponding $\sf{K}$-equivariant $A$-Dirac operator. The metric and orientation on $A$ define a trivialization of the half-density bundle $\Lambda^{1/2}$, and hence we may regard $D$ as an element of $\Psi^1(G,E)$. The operator $D$ is odd so takes the form
\[ D=\m{0&D^-\\D^+&0}.\]
The $G$-pseudodifferential calculus guarantees existence of a $\sf{K}$-invariant parametrix $P\in \Psi^{-1}(G,E)$, which we may take to be odd with off diagonal entries $P^{\pm}$. Let
\[ PD=1-Q, \qquad DP=1-R \]
where $Q,R\in \Psi^{-\infty}(G,E)$ are even smoothing operators with diagonal components $Q^{\pm}$, $R^{\pm}$.

In particular $d=D^+\in \Psi^\infty(G,E^+,E^-)$ is a $\sf{K}$-equivariant elliptic $G$-pseudodifferential operator hence defines a class $\tn{ind}(D):=[d]\in K_0^\sf{K}(\Psi^{-\infty}(G))$, and the equivariant $\tau$-index $\pair{\tau}{\tn{ind}(D)} \in C^\infty(\sf{K})^\sf{K}$ is defined. One has $p=P^-$, $q=Q^+$, $r=R^-$ and hence
\begin{equation} 
\label{e:gradedtauind}
\pair{\tau}{\tn{ind}(D)}(\gamma)=\tau^\gamma(Q^{+,2})-\tau^\gamma(R^{-,2}).
\end{equation}

\section{Asymptotic heat kernels}\label{s:asymptoticheat}
Let $\Delta=D^2\in \Psi^2(G,E)$ be the generalized $A$-Laplacian associated to an $A$-Dirac operator $D\in \Psi^1(G,E)$. In this section we describe how to construct a parametrix $K$ for the heat operator $\partial_t+\Delta^R$, by Borel summation of a formal solution to the heat equation. We use $K$ to construct a parametrix $P$ for $D$. The nature of the construction allows the equivariant trace pairing \eqref{e:gradedtauind} to be expressed in terms of the asymptotic expansion of the heat kernel.

Before coming to the construction of $K$, we briefly explain why we have not worked with the true heat kernel. This discussion will not be used in the sequel. By Proposition \ref{p:bdedgeom} the $s$ fibers of $G$ are Riemannian manifolds of bounded geometry. The restriction of $\Delta^R$ to each $s$ fiber $s^{-1}(m)$ has a well-defined heat kernel $\exp(-t\Delta^R|_{s^{-1}(m)})$ with a small-time asymptotic expansion along the diagonal similar to the case of a compact Riemannian manifold, cf. \cite[Proposition 2.11]{roe1988index}. These operators may be constructed by solving the heat equation or by functional calculus. The family of $s$-fiber heat kernels fit together into an element $\exp(-t\Delta^R)\in C^*_r(G,E)$, which may be viewed as the result of applying functional calculus to the unbounded regular self-adjoint multiplier $D^R$ of the $C^*$-algebraic completion $C_r^*(G,E)$ of $\Psi^{-\infty}(G,E)$ (for regularity see for example \cite[Section 3.6]{Vassout2006}). 

The full list of desired properties of the heat kernel $\exp(-t\Delta^R)$ for a general Lie groupoid $G$ have not, to our knowledge, been established. For example it is not known whether $\exp(-t\Delta^R)$ is smooth transversely to the $s$ fibers (cf. \cite[Conjecture 1.6]{bkso2014exponential}, \cite[Remark 4.8]{bohlen2018getzler}). Smoothness is known in important special cases, including for foliations \cite{heitsch1995bismut} and b-manifolds \cite{melrose1993atiyah}; see also \cite{bkso2014exponential} for further comments and results. Since we aim to compute the pairing of $\tn{ind}(D)$ with a possibly distributional trace $\tau$, this becomes important, and it is for this reason that we work instead with the smooth `asymptotic heat kernel' $K$ below. A similar object could be constructed using more machinery: pseudodifferential calculus for the heat operator $\partial_t+\Delta^R$ (see Section \ref{s:heatparametrix} for further comments).

\subsection{Formal heat kernels}
Following \cite[Chapter 2]{BerlineGetzlerVergne} or \cite{roe1998elliptic}, the first step in the construction of a heat kernel is to find a formal algebraic solution. Let $U\subset A$ be a radius $\kappa>0$ disk bundle, where $\kappa$ is chosen such that the exponential map $\exp^h$ defined by the metric on $A$ restricts to a diffeomorphism from $U$ to an open neighborhood of $M$ in $G$. By shrinking $U$ if necessary, we may assume without loss of generality that $\exp^h$ is a diffeomorphism on a slightly larger disk bundle containing the compact closure of $U$. For $X \in U_m \subset A_m$, it is convenient to label the point $\exp^h_m(X)$ by the ordered pair $(X,m)$. We frequently use $\exp^h$ to identify $U$ with its image under the exponential map, and likewise $U_m$ is identified with an open ball around $m$ in $s^{-1}(m)$. 

Following \cite[Chapter 2]{BerlineGetzlerVergne}, the use of half-densities slightly simplifies the discussion of the formal solution. The metric on $A$ determines a trivializing section $\nu$ of $\Lambda$. The pullback $s^*\nu$ is a trivializing section of $s^*\Lambda$. Let $\tilde{\nu}$ be the pullback of $\nu$ to the total space of $A$, regarded as a density for the vertical bundle of the projection $A \rightarrow M$. It pushes forward to a density
\[ \d X=\exp^h_*\tilde{\nu} \] 
for the vertical bundle of the projection $s|_U\colon U\rightarrow M$, or in other words, to a section of $r^*\Lambda|_U$. We use $\d X^{1/2}\otimes s^*\nu^{1/2}$ to trivialize $r^*\Lambda^{1/2}\otimes s^*\Lambda^{1/2}|_U$ (compare \cite[Chapter 2]{BerlineGetzlerVergne}).

A \emph{formal heat kernel} is a formal series of the form
\begin{equation} 
\label{e:formal}
F_t(X,m)=q_t(X)\sum_{i=0}^\infty t^i\Theta_i(X,m)\nu_m^{1/2}, \qquad q_t(X)=(4\pi t)^{-n/2}e^{-|X|^2/4t}\d X^{1/2}
\end{equation}
that satisfies the heat equation (with initial condition) for the family of generalized Laplacian operators $\Delta^R|_U$ along the fibers of $s|_U \colon U \rightarrow M$. Using parallel translation along radial geodesics in $U_m$, one has an isomorphism $(r^*E)_{\exp^h_m(X)}=(r^*E)_{(X,m)}\simeq E_m$, and under this isomorphism $\Theta_i(-,m)$ may be regarded as a $\End(E_m)$-valued function on $U_m$. Applying the discussion in \cite[Chapter 2]{BerlineGetzlerVergne} to each fiber of $s|_U$ shows that the $\Theta_i$ are uniquely determined by recursively solving a sequence of ODE:
\begin{equation} 
\label{e:recursive}
(\nabla_\R+i)\Theta_i(-,m)=-B\Theta_{i-1}(-,m), \qquad \Theta_0(0_m,m)=\id,
\end{equation}
where $\R$ is the Euler vector field of $A$ (or if one prefers, its pushforward under $\exp^h$), and $B=\d X^{-1/2}\Delta^R\d X^{1/2}$. The solution is unique and given by explicit formulas as in \cite[Chapter 2]{BerlineGetzlerVergne}, and in particular it follows that the $\Theta_i$ are smooth in both arguments $(X,m)$.

\subsection{Borel summation}\label{s:Borel}
The formal series \eqref{e:formal} is not guaranteed to converge. To fix this we use Borel summation. Let $\beta \colon \bR\rightarrow [0,1]$ be a smooth function with $\beta(t)=1$ when $|t|\le 1/2$ and $\beta(t)=0$ when $|t|\ge 1$. Define
\begin{equation} 
\label{e:bformal}
F^b_t(X,m)=q_t(X)\sum_{i=0}^\infty \beta(b_it)t^i\Theta_i(X,m)\nu_m^{1/2}, 
\end{equation}
where $b=(b_i)_{i=0}^\infty$, $b_i\ge 1$ is a suitably chosen increasing sequence of positive real numbers. Note that $F_t^b$ also depends on the bump function $\beta$, although we omit this dependence from the notation.
\begin{proposition}
\label{p:suitableb}
For a suitably chosen sequence $b$, the formal sum $\sum_{i=0}^\infty \beta(b_it)t^i\Theta_i$ in \eqref{e:bformal} converges on $\bR\times U$, uniformly in each $C^k$-norm. In particular $F^b_t$ is smooth on $(0,\infty)\times U$.
\end{proposition}
\begin{proof}
Recall that the compact closure $\ol{U}$ of $U$ is contained in a larger disk bundle in $A$ on which $\exp^h$ remains a diffeomorphism. The equations \eqref{e:recursive} can be solved over this larger neighborhood, and hence the $\Theta_i$ extend smoothly to an open set containing $\ol{U}$. By compactness of $\ol{U}$, the $C^k$-norm $\|\Theta_i|_U\|_{C^k}=C_{k,i}<\infty$. Let $b_i=\tn{max}(2C_{i,i},1)$. To bound arbitrary partial derivatives with respect to $t$, one must consider series of the form
\[ \sum_{i=0}^\infty b_i^l \frac{i!}{(i-q)!}\beta^{(l)}(b_it)t^{i-q}\Theta_i.\]
Note $i!/(i-q)!\le i^q$. The $C^k(U)$ norm of the $i\ge k$ tail of this series is bounded by
\[ \sum_{i=k}^\infty b_i^li^q\beta^{(l)}(b_it)|t|^{i-q}C_{k,i}.\]
Let $B_l=\|\beta^{(l)}\|_{C^0}$. The function $\beta^{(l)}(b_it)$ has support contained in $[-b_i^{-1},b_i^{-1}]$, so we obtain an upper bound
\[ B_l\sum_{i=k}^\infty b_i^{l+q-i}i^qC_{k,i}.\]
For $i\ge k$ we have $C_{k,i}\le C_{i,i}\le b_i/2$, which leads to the simpler bound
\[ B_l\sum_{i=k}^\infty b_i^{l+q+1-i}\frac{i^q}{2^i}, \]
which converges rapidly once $i\gg l+q$.
\end{proof}
With $b$ chosen as in Proposition \ref{p:suitableb}, we will refer to $F^b_t$ as a \emph{Borel summation of the formal heat kernel}.

Below it will be convenient to carry out various calculations modulo $O(t^\infty)$ error. With this in mind we make the following definition.
\begin{definition}
Let $A_t,B_t$ be smooth families (in $t \in (0,\infty)$) of distributional sections of a normed vector bundle defined on $(0,\infty)\times U$. We shall write $A_t=B_t+O(t^\infty)$ if $A_t-B_t$ is smooth, and for each $k=0,1,2,...$, the $C^k(U)$-norm $\|A_t-B_t\|_{C^k(U)}\rightarrow 0$ as $t\rightarrow 0^+$ faster than any power of $t$. In this case $A_t-B_t$ extends smoothly by $0$ to $t\le 0$.
\end{definition}

\begin{lemma}
\label{l:Fbheat}
A Borel summation of the formal heat kernel $F^b_t$ satisfies
\[ (\partial_t+\Delta^R)F^b_t=O(t^\infty). \]
\end{lemma}
\begin{proof}
Let $F_t^{(N)}$ be the approximate heat kernel (cf. \cite[Chapter 2]{BerlineGetzlerVergne}) that one obtains by cutting off the sum over $i=0,1,2,...$ in \eqref{e:formal} at $i=N+n/2$ (no Borel summation). In \cite[Chapter 2]{BerlineGetzlerVergne} it is proved that the $C^k(U)$ norm of the corresponding error $(\partial_t+\Delta^R)F^{(N)}_t$ on the right-hand-side of the heat equation is $O(t^{N-k/2})$. By construction the $C^l(U)$-norm of the difference $F^b_t-F_t^{(N)}$ is $O(t^{N+1-l/2})$. Setting $l=k+2$ we deduce that the $C^k(U)$-norm of $(\partial_t+\Delta^R)(F^b_t-F_t^{(N)})$ is $O(t^{N-k/2})$. Hence the $C^k(U)$-norm of $(\partial_t+\Delta^R)F^b_t$ is also $O(t^{N-k/2})$. Since this holds for all $N$, the result follows.
\end{proof}
\begin{remark}
\label{r:Otinf}
The proof of Lemma \ref{l:Fbheat} highlights a simplifying property of the class of $O(t^\infty)$ families of smoothing kernels that we shall use again below: If $R_t \in C^\infty((0,\infty)\times U)$ and is $O(t^\infty)$, then so is the result of applying any left or right invariant differential operator to $R_t$.% For contrast, in the context of the theory of the heat equation, the finite order $N$ analogue of this property might say that the $C^k(U)$-norm of $R_t$ is $O(t^{N-k/2})$, and thus taking $l$ derivatives typically reduces $N$ to $N-l/2$. We shall give a more sophisticated perspective in Section \ref{s:heatparametrix}.
\end{remark}

Fix a smooth bump function $\chi\colon [0,\infty)\rightarrow [0,1]$ such that $\chi(r)=0$ for $r\ge \kappa$ and $\chi(r)=1$ for $r\le \kappa/2$. The pullback of $\chi$ under the map $U \ni \exp^h_m(X)\mapsto |X|\in [0,\infty)$ is a compactly supported smooth function on $U$, equal to $1$ on a neighborhood of $M$; it extends smoothly by $0$ to all of $G$. We shall abuse notation slightly and denote the extension-by-$0$ of the pull back by $\chi$ as well. For any smooth section $F$ on $U$, we regard the product $\chi F$ as a smooth section on $G$ equal to $0$ outside $U$.
\begin{definition}
Let $F_t^b$ be a Borel-summation of the formal heat kernel and let $\chi$ be a bump function as above. The \emph{asymptotic heat kernel} defined by this data is
\[ K_t=\chi F_t^b \in C^\infty((0,\infty)\times G,\delta(E\otimes \Lambda^{1/2})).\]
For each $t \in (0,\infty)$, $K_t \in \Psi^{-\infty}(G)$. We also set $K_0=\delta_M$.
\end{definition}
\begin{remark}
A comparable object is used by Costello \cite{costello2011renormalization}, where $K_t$ is referred to as a `fake heat kernel'.
\end{remark}

\begin{corollary}
\label{c:AHKeqn} 
The asymptotic heat kernel $K_t$ satisfies
\[ (\partial_t+\Delta^R)K_t=O(t^\infty). \] 
Moreover for any $s \in \Psi^{-\infty}(G)$, $K_ts\rightarrow s$ as $t\rightarrow 0^+$ in the $C^k$ norm for all $k$.
\end{corollary}
\begin{proof}
The first statement follows from Lemma \ref{l:Fbheat}. The second statement follows from \cite[Theorem 2.29(1)]{BerlineGetzlerVergne}.
\end{proof}

\subsection{Asymptotic heat kernel as a heat parametrix}\label{s:heatparametrix}
In this section we will again leave out the coefficient bundle $E$ to simplify notation. So far we have taken the point of view that the asymptotic heat kernel is a smoothly varying family $(K_t)_{t>0}$ of compactly supported distributions on $G$, or in a similar vein, as a smoothly varying family of multipliers of $\Psi^{-\infty}(G)$. It is sometimes convenient to take the point of view that this family defines a single distribution $K$ on $\bR \times G$, conormal to $\{0\}\times M$, and with support contained in $[0,\infty)\times G$. For example, in the case of the Euclidean heat kernel in 1 dimension, this distribution is the locally integrable function 
\[ h(t)(4\pi t)^{-1/2}\exp(-(x-y)^2/4t) \] 
on $\bR^3=\bR\times \tn{Pair}(\bR)$ with $h(t)$ the Heaviside function. 

Regarding $\bR$ as a Lie group under addition, $\bR\times G$ is a Lie groupoid. To treat the heat operator $\partial_t+\Delta$, one wants a version of pseudodifferential calculus for the groupoid $\bR\times G$ in which the operator $\partial_t$ is considered to have order $2$. This is a relatively straight-forward change. In particular one adjusts the usual definition of the space of order $q$ symbols by replacing ordinary homogeneity with weighted homogeneity for an $\bR_{>0}$-action, where the coordinate $t$ is assigned weight $2$ instead of weight $1$, see for example \cite{melrose1993atiyah, beals1984heat, van2019groupoid} for further details (in the setup of \cite{van2019groupoid}, we use the filtration of the Lie algebroid $\bR_M\oplus A\rightarrow M=\{0\}\times M$ of $\bR\times G$ where the subbundle $\bR_M=M\times \bR$ is assigned degree $2$). The result is a filtered algebra of pseudodifferential operators $\Psi^\infty_{\tn{wt}}(\bR\times G)$ associated to the `weighted' Lie groupoid $\bR\times G$, with composition given by convolution of compactly supported generalized sections on $\bR\times G$ that are conormal to the unit space $\{0\}\times M$.

In this context the heat equation together with the initial condition become the equation
\begin{equation} 
\label{e:heatreformulated}
(\partial_t+\Delta^R)H=\delta_t\delta_M \qquad \text{(in } \Psi^\infty_{\tn{wt}}(\bR\times G)\text{)},
\end{equation}
where the RHS is the identity element of $\Psi^\infty_{\tn{wt}}(\bR\times G)$. To see this, convolve both sides of \eqref{e:heatreformulated} with a smooth compactly supported half-density $s \in \Psi^{-\infty}(\bR\times G)$ and evaluate at time $t$. On the LHS we obtain (omitting `$\star$' from the notation for brevity),
\[ ((\partial_t+\Delta^R)Hs)_t=\int_0^\infty \big(-H_u\partial_u s_{t-u}+\Delta^RH_us_{t-u}\big)\d u,\]
where we integrated by parts in the first term. Using $\Delta^R H_us_{t-u}=-\partial_u(H_us_v)|_{v=t-u}$ (since $u\mapsto H_us_v$ is the solution of the heat equation with initial condition $s_v$), one finds
\[ ((\partial_t+\Delta^R)Hs)_t=-\int_0^\infty \partial_u(H_us_{t-u})\d u=H_0s_t=s_t, \]
which agrees with the result of convolving the RHS of \eqref{e:heatreformulated} with $s$ and evaluating at time $t$.

The operator $\partial_t+\Delta \in \Psi^2_{\tn{wt}}(\bR\times G)$ and is elliptic in the weighted sense by construction, hence possesses a parametrix in $\Psi^{-2}_{\tn{wt}}(\bR\times G)$. Comparing \eqref{e:heatreformulated} with Corollary \ref{c:AHKeqn}, we see that $K$ is, in particular, an inverse of $\partial_t+\Delta$ in the algebra of compactly supported $\{0\}\times M$-conormal distributions on $\bR\times G$ modulo smoothing kernels, and hence is such a parametrix. That $K$ has weighted order $-2$ may also be checked more directly. Consider the Fourier transform of $K$ in the source fiber $\bR\times s^{-1}(m)$ of a point $(0,m)$ in the unit space $\{0\}\times M$. The Fourier transform of the Euclidean heat kernel $q(t,X)=(4\pi t)^{-n/2} \exp(-|X|^2/4t)h(t)$ in $\bR^n$ is the generalized function
\begin{equation} 
\label{e:FourierEuclHeat}
\scr{F}(q)(\varsigma,\xi)=\lim_{\epsilon\rightarrow 0^+} \frac{1}{2\pi\i \varsigma+\epsilon+4\pi^2|\xi|^2}. 
\end{equation}
(The limit $\epsilon\rightarrow 0^+$ comes from the Fourier transform of the Heaviside function $h(t)$.) Although not smooth at the origin $(0,0)\in \bR^{n+1}$, \eqref{e:FourierEuclHeat} is weighted homogeneous of degree $-2$ (i.e. under $(\varsigma,\xi)\mapsto (\lambda^2\varsigma,\lambda \xi)$), and satisfies the weighted symbol estimates outside any neighborhood of the origin $(0,0) \in \bR^{n+1}$. The Fourier transform of $K$ in $\bR\times s^{-1}(m)$ is the convolution of \eqref{e:FourierEuclHeat} with the Fourier transform of the smooth compactly supported function
\[ \Theta(t,X)=\sum_{i=0}^\infty\beta(b_it)t^i\chi(X,m)\Theta_i(X,m).\]
Convolution smooths out the singularity of \eqref{e:FourierEuclHeat} at $(0,0)$ while preserving approximate degree $-2$ homogeneity at infinity, and the result satisfies the weighted symbol estimates.

In fact $\partial_t+\Delta$, $K$ belong to the subalgebra $\Psi^\infty_{\tn{wt}}([0,\infty)\times G)$ consisting of elements whose support is contained in $[0,\infty)\times G$ (the subalgebra property follows since $[0,\infty)\subset \bR$ is a submonoid in $(\bR,+)$). The intersection 
\[ \Psi^\infty_{\tn{wt}}([0,\infty)\times G)\cap \Psi^{-\infty}(\bR\times G)=\Psi^{-\infty}([0,\infty)\times G) \]
is the ideal consisting of compactly supported smooth $O(t^\infty)$ families  $R=(R_t)_{t\ge 0}$ of smoothing operators, so that the calculations we have done modulo $O(t^\infty)$ errors in the previous section can be reinterpreted as statements regarding elements of the quotient algebra. In particular Corollary \ref{c:AHKeqn} may be reformulated as follows.
\begin{corollary}
\label{c:heatparam}
The asymptotic heat kernel $K=(K_t)_{t\ge 0}$ descends to an inverse of $(\partial_t+\Delta)$ in the quotient algebra $\Psi^\infty_{\tn{wt}}([0,\infty)\times G)/\Psi^{-\infty}([0,\infty)\times G)$.
\end{corollary}
\begin{proof}
In light of equations \eqref{e:heatreformulated} and \eqref{e:convdiff}, Corollary \ref{c:AHKeqn} says that $K$ descends to a right inverse of $(\partial_t+\Delta)$ in the quotient. The family $(K_t^\star)_{t\ge 0}$ similarly descends to a left inverse. Thus $(\partial_t+\Delta)$ is invertible in the quotient algebra, and $K$ descends to a 2-sided inverse.
\end{proof}
In this sense the asymptotic heat kernel defines a parametrix for the heat operator $(\partial_t+\Delta)$. We remark that it is also possible to work with a larger weighted Lie groupoid $\tn{Pair}(\bR)\times G$, and above we have taken advantage of the translation-invariance of $\partial_t$ to replace $\tn{Pair}(\bR)$ with $\bR$. We refer the reader to \cite{melrose1993atiyah, ponge2003new, beals1984heat} for further discussion of heat parametrices.

\begin{corollary}
\label{c:AHKprop}
The asymptotic heat kernel $K_t$ has the following properties:
\begin{enumerate}
\item $K_t-K_t^*=O(t^\infty)$;
\item the commutator (in $\Psi^\infty(G)$) $[D,K_t]=O(t^\infty)$;
\item the convolution square (in $\Psi^{-\infty}(G)$) $K_t^2=K_t\star K_t=K_{2t}+O(t^\infty)$.
\end{enumerate}
\end{corollary}
\begin{proof}
Property (a) was explained in the proof of Corollary \ref{c:heatparam}. Properties (b), (c) follow upon passing to the quotient algebra $\Psi^\infty_{\tn{wt}}([0,\infty)\times G)/\Psi^{-\infty}([0,\infty)\times G)$. For (b) use the fact that if an element $d$ commutes with an invertible element $k$ then $[d,k^{-1}]=k^{-1}[k,d]k^{-1}=0$. For (c) use the fact that $K_t^2$, $K_{2t}$ both descend to the inverse of $(\partial_t+2\Delta)$. Indeed using \eqref{e:convdiff},
\begin{align*} 
-\partial_t(K_t\star K_t)&=\Delta^R K_t\star K_t+K_t\star (\Delta^RK_t)\\
&=\Delta^R K_t\star K_t+K_t\star \Delta \star K_t\\
&=\Delta^R K_t\star K_t+\Delta\star K_t \star K_t+O(t^\infty)\\
&=2\Delta^R K_t\star K_t+O(t^\infty),
\end{align*}
where in the third line we used part (b).
\end{proof}

\subsection{Parametrices and the index}
Let $K_t=\chi F_t^b$ be an asymptotic heat kernel for $\Delta=D^2$. For each $t>0$ define
\begin{equation} 
\label{e:Pt}
P_t=D\int_0^t K_u \d u. 
\end{equation}
If $s \in \Psi^{-\infty}(G)$ then by Corollary \ref{c:AHKeqn}, $u\mapsto K_us$ is continuous. From this it follows that the integral \eqref{e:Pt} converges in the space of compactly supported $M$-conormal generalized sections on $G$.

\begin{proposition}
\label{p:heatparametrix}
For all $t>0$, $P_t \in \Psi^{-1}(G,E)$ is a parametrix for $D$. Moreover
\[ DP_t=1-K_t+O(t^\infty), \qquad P_tD=1-K_t+O(t^\infty).\]
\end{proposition}
\begin{proof}
Using Corollary \ref{c:AHKeqn} and the fundamental theorem of calculus,
\begin{align} 
\label{e:htparam1}
DP_t&=\int_0^t \Delta^R K_u \d u \nonumber \\
&=\int_0^t(-\partial_u K_u+O(u^\infty))\d u \nonumber \\
&=1-K_t+O(t^\infty).
\end{align}
Since
\[ P_tD=DP_t-[D,P_t], \]
we compute the commutator:
\[ [D,P_t]=D\int_0^t [D,K_u] \d u.\]
The integrand is $O(u^\infty)$ by Corollary \ref{c:AHKprop}. As in the first part of the proof we conclude $[D,P_t]$ is $O(t^\infty)$, and hence
\begin{equation} 
\label{e:htparam2}
P_tD=1-K_t+O(t^\infty).
\end{equation}
Equations \eqref{e:htparam1}, \eqref{e:htparam2} show in particular that $P_t$ descends to an inverse of $D$ in the algebra of $M$-conormal distributions modulo smoothing kernels, and hence by uniqueness of inverses, $P_t$ is pseudodifferential of order $-1$.
\end{proof}
Define $Q_t,R_t \in \Psi^{-\infty}(G,E)$ by
\[ P_tD=1-Q_t, \qquad DP_t=1-R_t.\]
Then according to Proposition \ref{p:heatparametrix},
\begin{equation} 
\label{e:QReqK}
Q_t=K_t+O(t^\infty), \qquad R_t=K_t+O(t^\infty).
\end{equation}
The next result is the main result of this section; it is a McKean-Singer-type formula for the equivariant trace pairing.
\begin{theorem}
\label{t:t0term}
The equivariant trace pairing $\pair{\tau}{\tn{ind}(D)}(\gamma)$ equals the constant term in the asymptotic expansion of $\tau^\gamma_s(K_t)$ as $t\rightarrow 0^+$.
\end{theorem}
\begin{proof}
Using equation \eqref{e:gradedtauind}, and then the $O(t^\infty)$ computations from equation \eqref{e:QReqK} and Corollary \ref{c:AHKprop}, we have
\begin{align*} 
\pair{\tau}{\tn{ind}(D)}(\gamma)&=\tau^\gamma(Q_t^{+,2})-\tau^\gamma(R_t^{-,2})\\
&=\tau^\gamma(K_t^{+,2})-\tau^\gamma(K_t^{-,2})+O(t^\infty)\\
&=\tau^\gamma(K_{2t}^+)-\tau^\gamma(K_{2t}^-)+O(t^\infty)\\
&=\tau^\gamma_s(K_{2t})+O(t^\infty).
\end{align*}
The LHS of this equation is independent of $t$. By uniqueness of asymptotic expansions the LHS equals the constant term in the asymptotic expansion of the RHS.
\end{proof}

Note that $\tau^\gamma_s(K_t)=\tau_s(K_t^\gamma)$ where we define $K_t^\gamma \in \Psi^{-\infty}(G,E)$ by
\begin{equation}
\label{e:Ktgamma}
K_t^\gamma(g)=\gamma^E K_t(\gamma^{-1}g).
\end{equation}
\begin{remark}
\label{r:suppKtgamma}
Since $K_t$ has support contained in a small neighborhood of $M$ in $G$, $K_t^\gamma$ has support contained in a small neighborhood of the submanifold $\gamma$ of $G$. In particular if $g=m \in M$ then $K_t^\gamma(m)$ vanishes unless $\gamma^{-1}m \in \supp(\chi)$. Thus by choosing $\chi$ with support sufficiently close to $M$, one can arrange that $K_t^\gamma|_M$ vanishes outside a small tubular neighborhood of $M^\gamma$.
\end{remark}

\section{Getzler rescaling and deformation to the normal cone}\label{s:rescaling}
In this section we describe a version of Getzler rescaling suited to the fixed point calculations in the next section. The discussion takes place in a more general setting, involving a metrised Lie algebroid $(B,\rho \colon B \rightarrow TP,h)$, a $\Cl(B^*)$-module $W$, and an embedded submanifold $j\colon Q \hookrightarrow P$ such that $TQ$ is transverse to $\rho$. To orient the reader, the relevant special case for the $\gamma$-fixed point calculations in the next section is
\[ P=G, \quad Q=M^\gamma, \quad B=\delta A, \quad W=\delta E.\]
We formulate Getzler rescaling in terms of a vector bundle $\sbW\rightarrow \sbN_PQ$ over the deformation to the normal cone construction for $j\colon Q\hookrightarrow P$, analogous to constructions in \cite{higsonyi, sadeghbraverman2022}. A variant also appears in \cite{melrose1993atiyah}.

\subsection{Deformation to the normal cone}
Let $j\colon Q\hookrightarrow P$ be an embedded submanifold. There is a smooth manifold $\sbN_PQ$ equipped with a submersion $u \colon \sbN_PQ\rightarrow \bR$ that interpolates between $P$ and the normal bundle $\dN_PQ$ to $Q$, in the sense that
\[ u^{-1}(c)=\begin{cases} P & \text{ if } c \ne 0,\\ \dN_PQ & \text{ if }c=0.\end{cases}\]
This is the \emph{deformation to the normal cone} construction. For $f \in C^\infty(P)$ let $o^{van}(f)$ be the order of vanishing of $f$ along $Q$. We shall make use of the description of $\sbN_PQ$ as the character spectrum of the \emph{Rees algebra}:
\[ \scr{A}(\sbN_PQ)=\bigg\{\sum u^{-k}f_k \in C^\infty(P)[u,u^{-1}]\, \bigg| \, o^{van}(f_k)\ge k\bigg\}.\]
See for example \cite{HigsonEulerLike, higsonyi} for a treatment in the smooth category adapted to our purposes. The special fiber $\dN_PQ=u^{-1}(0)$ is recovered algebraically as the character spectrum of the quotient
\[ \scr{A}(\dN_PQ)=\scr{A}(\sbN_PQ)/u\cdot\scr{A}(\sbN_PQ).\]
By construction $\scr{A}(\dN_PQ)$ is the associated graded algebra for the filtration of $C^\infty(P)$ by vanishing order $o^{van}$, and $\scr{A}(\dN_PQ)$ may be identified with the subalgebra of $C^\infty(\dN_PQ)$ consisting of functions that are polynomial along the fibers of the projection $p\colon \dN_PQ\rightarrow Q$.

Following \cite{higsonyi}, for later comparison, it is convenient to reformulate vanishing order $o^Q$ in terms of differential operators. Define a valuation
\[ o^{Q,0}(f)=\begin{cases} -\infty & \text{ if } f|_Q=0\\ 0 & \text{ else.} \end{cases}\]
For any differential operator $T$ on $P$, let $o(T)$ be the order of $T$. Then
\begin{equation} 
\label{e:vanishingorder}
o^{van}(f)=\inf_T\big\{o(T)-o^{Q,0}(Tf)\big\},
\end{equation}
where the infimum is taken over all differential operators on $P$. One obtains the same result in \eqref{e:vanishingorder} if the infimum is instead taken over the algebra of differential operators generated by any collection of vector fields whose images span the normal bundle $\dN_PQ$. 

\subsection{Schwartz functions}\label{s:schwartz}
We give a brief introduction to Schwartz functions on the deformation to the normal cone $\sbN_PQ$. Other approaches to Schwartz functions can be found in \cite{rouse2008schwartz, ewert2021, mohsen2021derivations, debord2014adiabatic}. This will not become relevant until Section \ref{s:rescaledtrace}, although we include it here because it is part of the general theory of the deformation to the normal cone construction. For simplicity we consider the case of an embedded submanifold $j\colon Q\hookrightarrow P$ of a compact manifold (the relevant case in Section \ref{s:rescaledtrace} will be $P=M$, $Q=M^\gamma$).

Our definition will be modelled on the following description of Schwartz functions on the total space of a vector bundle $V\rightarrow Q$ over a compact base. Let $\scr{A}(V)$ be the algebra generated by smooth functions on $V$ that are homogeneous of some degree for the action of $\bR_+$ by scalar multiplication. Let $\scr{D}(V)$ be the algebra generated by $\scr{A}(V)$ together with all smooth vector fields that are homogeneous of some degree for the $\bR_+$ action. A smooth function $f \in C^\infty(V)$ is of rapid decay (resp. Schwartz) if $af$ is bounded for all $a \in \scr{A}(V)$ (resp. $\D f$ is bounded for all $\D \in \scr{D}(V)$).

The deformation space $\sbN_PQ$ carries an action of $\bR_+$ called the \emph{zoom action}. It corresponds to the automorphism of the Rees algebra induced by $u\mapsto \lambda^{-1}u$, $\lambda \in \bR_+$. On the subset $P\times \bR^\times \subset \sbN_PQ$, the zoom action is $\lambda\cdot (p,u)=(p,\lambda^{-1}u)$, $\lambda \in \bR_+$. On the complementary subset $\dN_PQ\times \{0\}\subset \sbN_PQ$, the zoom action is scalar multiplication by $\lambda \in \bR_+$.

The Rees algebra $\scr{A}(\sbN_PQ)$ is the algebra of smooth functions generated by smooth functions that are homogeneous of some degree for the $\bR_+$ action; indeed the restriction of a homogeneous function to $P\times \bR^\times$ must be of the form $u^{-k}f_k$ for some $f_k \in C^\infty(P)$ and $k \in \bZ$, and then the fact that $u^{-k}f_k$ extends smoothly to $u^{-1}(0)$ implies that $o^{van}(f_k)\ge k$. Define $\scr{D}(\sbN_PQ)$ to be the algebra generated by $\scr{A}(\sbN_PQ)$ together with all smooth vector fields that are homogeneous of some degree for the zoom action. 

A function $f \in C^\infty(\sbN_PQ)$ is of \emph{rapid decay} (resp. \emph{Schwartz}) if $af$ is bounded for all $a \in \scr{A}(\sbN_PQ)$ (resp. $\D f$ is bounded for all $\D \in \scr{D}(\sbN_PQ)$). We will only be interested in the behavior of functions near $u=0$, so it is convenient to have the following variation: a function $f \in C^\infty(\sbN_PQ)$ is of \emph{rapid decay near} $u=0$ (resp. \emph{Schwartz near} $u=0$) if $af\upharpoonright u^{-1}([-1,1])$ is bounded for all $a \in \scr{A}(\sbN_PQ)$ (resp. $\D f\upharpoonright u^{-1}([-1,1])$ is bounded for all $\D \in \scr{D}(\sbN_PQ)$). Since $P$ is compact, the interval $[-1,1]$ could be replaced with any $[-r,r]$, $r>0$ without changing the definitions.

To make the Schwartz condition more transparent, we describe $\scr{A}(\sbN_PQ)$-module generators of $\scr{D}(\sbN_PQ)$. First consider the simpler case of a vector bundle $V \rightarrow Q$. Using a local trivialization of $V$, one can construct a local frame of $TV$ consisting of vector fields of degrees $-1$, $0$ (`vertical', `horizontal' vector fields respectively). Any vector field $X\in \mf{X}(V)$ can be expressed locally as a linear combination of the elements of the frame, and if $X$ is homogeneous, then the coefficients must be homogeneous. It follows that $\scr{D}(V)$ is generated as a left $\scr{A}(V)$-module by monomials in the vector fields of degrees $-1$, $0$. Since the commutator of vector fields of degrees $i$, $j$ has degree $i+j$, monomials with all degree $-1$ vector fields to the left of all degree $0$ vector fields (or vice versa) still generate. As a corollary, $f \in C^\infty(V)$ is Schwartz if and only if $X_1\cdots X_l Y_1\cdots Y_k f$ is of rapid decay for all $l,k\ge 0$ and $X_1,...,X_l \in \mf{X}(V)$ degree $-1$, $Y_1,...,Y_k \in \mf{X}(V)$ degree $0$.

Let $\mf{X}(P,Q)\subset \mf{X}(P)$ denote the Lie subalgebra of smooth vector fields tangent to $Q$. If $Y \in \mf{X}(P,Q)$ then $Y$ preserves order of vanishing of functions on $Q$ and hence acts naturally on the Rees algebra $\scr{A}(\sbN_PQ)$. The corresponding vector field on $\sbN_PQ$, also denoted $Y$, has degree $0$ for the zoom action. If $X\in \mf{X}(P)$ is any vector field, then $X$ decreases vanishing order of functions by at most $1$, and hence $uX$ acts naturally on $\scr{A}(\sbN_PQ)$. The corresponding vector field on $\sbN_PQ$, also denoted $uX$, has degree $-1$ for the zoom action.

More non-trivially, there is also a generator of degree $+1$. Let $\R \in \mf{X}(P,Q)$ be any Euler-like vector field along $Q$. Let $\C$ be the vector field on $\sbN_PQ$ induced by the derivation $u\partial_u$ of the Rees algebra. By either calculating in local coordinates or computing its action in the Rees algebra, one checks that
\begin{equation} 
\label{e:horizontalvf}
\T=\frac{1}{u}(\C+\R) 
\end{equation}
is a smooth vector field on $\sbN_PQ$, see \cite{HigsonEulerLike} for further discussion. The vector field $\T$ projects to the vector field $\partial_u$ on $\bR$, and is homogeneous of degree $+1$ for the zoom action.

Every point in $\sbN_PQ$ has an $\bR_+$-invariant neighborhood on which there is a local frame of $T\sbN_PQ$ whose elements are of the three types $Y \in \mf{X}(P,Q)$, $uX\in u\mf{X}(P)$, and $\T$. By the same argument as the vector bundle case, $\scr{D}(\sbN_PQ)$ is generated as a left $\scr{A}(\sbN_PQ)$-module by monomials in these vector fields. Moreover since the commutator of vector fields of degree $i$, $j$ has degree $i+j$, monomials with all degree $-1$ vector fields to the left of all degree $0$ vector fields to the left of powers of $\T$ still generate. As a corollary, $f \in C^\infty(\sbN_PQ)$ is Schwartz if and only if $(uX_1)\cdots (uX_l) Y_1\cdots Y_k \T^j f$ is of rapid decay for all $j,k,l\ge 0$ and $X_1,...,X_l \in \mf{X}(P)$, $Y_1,...,Y_k \in \mf{X}(P,Q)$.

The Euler-like vector field $\R$ determines a tubular neighborhood embedding $e\colon N \hookrightarrow \dN_PQ$, where $N$ is an open neighborhood of $Q$ in $P$ (in fact, the flow of $\T$ on $\sbN_PQ$ implements the embedding, see \cite{HigsonEulerLike}). Let
\begin{equation} 
\label{e:iota}
\iota \colon \sbN_NQ\hookrightarrow \bR\times \dN_NQ=\bR\times \dN_PQ 
\end{equation}
be the smooth embedding extending the map
\[ (u,x)\in (\bR\backslash \{0\})\times N \mapsto (u,u^{-1}e(x)) \in (\bR\backslash \{0\})\times \dN_PQ.\]
$\iota$ is a diffeomorphism onto the open subset
\begin{equation} 
\label{e:imageiota}
\{(u,v) \in \bR\times \dN_PQ\mid uv \in e(N)\}.
\end{equation}
Under the identification of $\sbN_PQ$ with \eqref{e:imageiota}, $\T$ becomes the vector field $\partial_u$. 

Recall $\pi \colon \sbN_PQ\rightarrow P$ is the canonical smooth map. For any function $f \in C^\infty(\sbN_PQ)$ whose support is contained in a set $\pi^{-1}(K) \subset \sbN_NQ$ with $K\subset N$ compact, the pushforward $\iota_* f$ of $f$ to \eqref{e:imageiota} extends smoothly by $0$ to all of $\bR \times \dN_PQ$.

\begin{lemma}
\label{l:SchwartzProperties}
Let $f \in C^\infty(\sbN_PQ)$ have support contained in $\pi^{-1}(K)\subset \sbN_NQ$ for some compact subset $K\subset N$. Then $f$ is of rapid decay (resp. Schwartz) near $u=0$ if and only if $\iota_*f\upharpoonright [-1,1]\times \dN_PQ$ is of rapid decay (resp. Schwartz).
\end{lemma}
\begin{proof}
Throughout the argument, $u$ is constrained to lie in the interval $[-1,1]$, and we omit this from the notation. Let $f \in C^\infty(\sbN_PQ)$ have support in $\pi^{-1}(K)$. Let $\varrho \in C^\infty(P)$ be a smooth function equal to $1$ on $K$ and with support contained in $N$. Suppose first that $f$ is of rapid decay near $u=0$. Let $p \in C^\infty(\dN_PQ)$ be a function that is homogeneous of degree $k$ along the fibers of $\dN_PQ \rightarrow Q$, which we view as a function on $\bR\times \dN_PQ$ independent of $u \in \bR$. Then $(\iota^*p)(\pi^*\varrho)\in \scr{A}(\sbN_PQ)$ and $(\iota^*p)f=(\iota^*p)(\pi^*\varrho)f$ is bounded. Thus $p(\iota_*f)$ is bounded, and since this holds for any $k$ and any $p$, $\iota_*f$ is of rapid decay.

Conversely suppose $\iota_*f$ is of rapid decay. Let $g \in C^\infty(P)$ vanish to order $k$ on $Q$ so that $u^{-k}g \in \scr{A}(\sbN_PQ)$. The pushforward $\iota_*(u^{-k}g\cdot \pi^*\varrho)$ is a smooth function, which, by Taylor's theorem, grows no faster than a polynomial of degree $k$. Since $\iota_*f$ is of rapid decay, the product $\iota_*(u^{-k}g\cdot \pi^*\varrho \cdot f)=\iota_*(u^{-k}g f)$ is bounded, and hence $u^{-k}gf$ is bounded. As this holds for any $k$ and any $g$ vanishing to order $k$ on $Q$, it follows that $f$ is of rapid decay near $u=0$.

The proof of the corresponding statements in the Schwartz case is similar. Note that a vector field $X$ of degree $-1$ (resp. $Y$ of degree $0$, resp. $\partial_u$) on $\dN_PQ$ pushes forward under $\iota$ to $u(e^{-1}_*X)$ (resp. $e^{-1}_*Y$, resp. $\T$) on $N$. Multiplying by $\pi^*\varrho$ and extending by $0$ then produces suitable homogeneous vector fields on $\sbN_PQ$. For the converse direction, it suffices to note that for $X \in \mf{X}(P)$ (resp. $Y \in \mf{X}(P,Q)$) the vector field $(\pi^*\varrho)uX$ (resp. $(\pi^*\varrho)Y$) pushes forward under $\iota$ to a vector field whose coefficients grow no faster than a polynomial along the fibers of $\dN_PQ$, by an argument with Taylor's theorem similar to above.
\end{proof}

\begin{lemma}
Let $f \in C^\infty(\sbN_PQ)$ have support contained in $\pi^{-1}(V)$, where $V$ is the complement of a neighborhood of $Q$. Then $f$ is of rapid decay near $u=0$ if and only if $f(u)\rightarrow 0$ uniformly as $u\rightarrow 0$ faster than any power of $u$. Likewise $f$ is Schwartz near $u=0$ if and only if for all $k,l\ge 0$ and $Z_1,...,Z_k \in \mf{X}(P)$, $Z_1\cdots Z_l\T^kf(u)\rightarrow 0$ uniformly as $u\rightarrow 0$ faster than any power of $u$.
\end{lemma}
\begin{proof}
Throughout the argument, $u$ is constrained to lie in the interval $[-1,1]$, and we omit this from the notation. Let $\chi$ be a smooth function equal to $1$ on $V$ and vanishing identically on a neighborhood of $Q$, so that $u^{-k}\pi^*\chi \in \scr{A}(\sbN_PQ)$ for any $k\ge 0$. If $f$ is of rapid decay near $u=0$, then by assumption $u^{-k}(\pi^*\chi)f=u^{-k}f$ is bounded for any $k\ge 0$, and consequently $f(u)\rightarrow 0$ uniformly as $u\rightarrow 0$ faster than any power of $u$. Conversely let $g \in C^\infty(P)$ vanish to order $k$ on $Q$. Then $u^{-k}g \in \scr{A}(\sbN_PQ)$ and $u^{-k}gf=g(u^{-k}f)$. Since $f(u)\rightarrow 0$ uniformly faster than any power of $u$, $u^{-k}f(u)\rightarrow 0$ as $u\rightarrow 0$, and $u^{-k}gf$ is bounded. The proof of the corresponding statements in the Schwartz case is similar.
\end{proof}
Combining the lemmas yields the following.
\begin{corollary}
Let $f \in C^\infty(\sbN_PQ)$. Let $\varrho \in C^\infty(P)$ be equal to $1$ on a neighborhood of $Q$ and have support contained in $N$. Then $f$ is of rapid decay (resp. Schwartz) near $u=0$ if and only if both of the following two conditions hold:
\begin{enumerate}[(a)]
\item $\iota_*((\pi^*\varrho)\cdot f)\upharpoonright [-1,1]\times \dN_PQ$ is of rapid decay (resp. Schwartz).
\item $(1-\pi^*\varrho)f(u)\rightarrow 0$ (resp. for all $k,l\ge 0$ and  $Z_1,...,Z_k \in \mf{X}(P)$, $Z_1\cdots Z_l \T^k(1-\pi^*\varrho)f(u)\rightarrow 0$) uniformly as $u\rightarrow 0$ faster than any power of $u$.
\end{enumerate}
\end{corollary}
\begin{example}
\label{ex:gaussian}
With notation as in Lemma \ref{l:SchwartzProperties}, suppose $|\cdot|$ is a norm on the fibers of $\dN_PQ$, and let $\chi \in C^\infty_c(N)$ be equal to $1$ on $Q$. Identify $N$ with a disk subbundle in $\dN_PQ$ via the tubular neighborhood embedding. Then
\[ f(u,v)=\begin{cases} \chi(v) e^{-|v|^2/u^2},  &\text{ if } u \ne 0, v \in N,\\ e^{-|v|^2}, &\text{ if } u=0, v \in \dN_PN, \end{cases}\] 
is a smooth function on $\sbN_NQ$ that is Schwartz near $u=0$. Indeed $(\iota_*f)(u,v)=\chi(uv)e^{-|v|^2}$ is Schwartz and apply Lemma \ref{l:SchwartzProperties}. Extending by $0$ we obtain a smooth function on $\sbN_PQ$ that is Schwartz near $u=0$. One thus obtains a collection of examples $\D f$ where $\D \in \scr{D}(\sbN_PQ)$.
\end{example}

\subsection{Getzler order}
Let $(B,\rho\colon B \rightarrow TP,h)$ be a metrised Lie algebroid of rank $r$ over $P$ such that $\rho$ is transverse to $TQ$. Let $R$ denote the curvature of the Levi-Civita $B$-connection $\nabla$ for $(B,h)$. Let $(W,c\colon \Cl(B^*)\rightarrow \End(W),\nabla)$ be a $\Cl(B^*)$-module equipped with a Clifford $B$-connection (also denoted $\nabla$) having curvature $F$. Recall that $\Cl(B^*)$ has a natural filtration 
\[ \Cl^0(B^*)\subset \Cl^1(B^*)\subset \cdots \subset \Cl^r(B^*)=\Cl(B^*).\] 
This induces a filtration of the endomorphism bundle 
\[ \End^k(W)=\bCl^k(B^*)\otimes \End_{\Cl}(W).\] 
We assume that $j^*W$ is equipped with a filtration 
\[ 0=(j^*W)^{-1}\subset (j^*W)^0\subset (j^*W)^1\subset \cdots \subset (j^*W)^n=j^*W \] 
satisfying:
\begin{enumerate}[(i)]
\item $j^*\End^k(W)\cdot (j^*W)^l\subset (j^*W)^{k+l}$;
\item $(j^*W)^k$ is preserved by the $j^!B$-connection $j^!\nabla$.
\end{enumerate}
Let
\[ \gr(j^*W)=\bigoplus_{k=0}^n (j^*W)^k/(j^*W)^{k-1} \rightarrow Q \]
be the associated graded vector bundle. For a section $\sigma \in \Gamma(W)$, let
\begin{equation} 
\label{e:WQorder}
o^{W,Q}(\sigma)=\tn{min}\big\{k \, \big| \, j^*\sigma \in \Gamma((j^*W)^k)\big\}.
\end{equation}
The relevant example for the $\gamma$-fixed point calculations in the next section is the following.
\begin{example}
\label{ex:fixedptexample1}
With notation as in the previous section, let $P=G$, $Q=M^\gamma$, $B=\delta A$, and $W=\delta E$ with the induced connection \eqref{e:deltaAconn}. Then $j^*W=(E\otimes E^*)|_{M^\gamma}\simeq \End(E)|_{M^\gamma}$ has filtration
\[ (j^*W)^k=\End^k(E)|_{M^\gamma}.\]
Property (i) is immediate, while (ii) holds because the Clifford $B$-connection on $W$ is induced by a Clifford $A$-connection on $E$.
\end{example}

Let $\U(B,W)$ be the algebra of differential operators acting on $W$ generated by bundle endomorphisms together with the operators $\nabla_X$ for $X \in \Gamma(B)$. Any element of $\U(B,W)$ is a linear combination of operators of the form
\begin{equation} 
\label{e:ordering}
L\nabla_{X_1}\cdots \nabla_{X_l},
\end{equation}
where $L\in \Gamma(\End^k(W))$, $X_1,...,X_l \in \Gamma(B)$.
\begin{definition}
The operator in equation \eqref{e:ordering} will be said to have \emph{Getzler order at most} $k+l$. The \emph{Getzler order} $o^g(T)$ of an operator $T \in \U(B,W)$ is 
\[ o^g(T)=\inf \big\{\tn{max}(k_1+l_1,...,k_m+l_m)\mid T=T_1+\cdots+T_m\big\} \]
where the infimum is taken over all finite decompositions of $T$ into operators $T_j$ of the form \eqref{e:ordering}, where $T_j=R_j\nabla_{X_{1,j}}\cdots \nabla_{X_{l_j,j}}$ and $R_j \in \Gamma(\End^{k_j}(W))$.
\end{definition}

\subsection{Scaling order}
Define a valuation $o^{Q,W,0}$ on $\Gamma(W)$ by
\[ o^{Q,W,0}(\sigma)=\begin{cases} -\infty &\text{ if } j^*\sigma=0\\
o^{Q,W}(\sigma) & \text{ else,} \end{cases}\]
where $o^{Q,W}(\sigma)$ is as in \eqref{e:WQorder}. By analogy with \eqref{e:vanishingorder}, we introduce the following.
\begin{definition}
The \emph{scaling order} $o^{sc}(\sigma) \in \{-n,-n+1,...,\infty\}$ of $\sigma \in \Gamma(W)$ is
\begin{equation} 
\label{e:scalingdegree}
o^{sc}(\sigma)=\inf_{T\in \U(B,W)}\big\{o^g(T)-o^{Q,W,0}(T\sigma)\big\}.
\end{equation}
\end{definition}
There is an equivalent and arguably more intuitive description of scaling order. Choose an Euler-like section $\E \in \Gamma(B)$ for $Q\hookrightarrow P$, i.e. a section $\E$ defined on a neighborhood $U$ of $Q$, vanishing along $Q$, and such that the linearization of the vector field $\rho(\E)$ along $Q$ is the Euler vector field of the normal bundle $\dN_PQ$. Existence of such a section is ensured by the transversality assumption $\rho \pitchfork TQ$, cf. \cite{bischoff2020deformation}. We say $\sigma \in \Gamma(W)$ is \emph{synchronous} if $\nabla_\E \sigma=0$. For $\sigma$ synchronous and $f \in C^\infty(P)$, we say that $f\sigma$ has \emph{Taylor order at most} $o^{van}(f)-o^{Q,W,0}(\sigma)$. The sections $f\sigma$ of this form and with Taylor order at most $d$ generate a $C^\infty(P)$-submodule $\W_d \subset \Gamma(W)$, fitting into a decreasing filtration
\[ \Gamma(W)=\W_{-n}\supset \W_{-n+1}\supset \cdots \supset \W_d\supset \cdots, \qquad \W_\infty=\bigcap_{d=-n}^\infty \W_d, \]
and we define the \emph{Taylor order} of a general section $\sigma$ to be the maximal $d \in \{-n,-n+1,...,\infty\}$ such that $\sigma \in \W_d$.

\begin{proposition}
\label{p:taylororder}
The scaling order of a section equals its Taylor order.
\end{proposition}
In particular Taylor order is independent of the choice of Euler-like section $\E$ as \eqref{e:scalingdegree} is manifestly so. The proof of this result is similar to \cite[Section 6]{higsonyi} and \cite[Section 4]{sadeghbraverman2022}. A new feature in Proposition \ref{p:taylororder} is that we have defined Getzler order, scaling order, and Taylor order using a Clifford $B$-connection $\nabla$ for a general Lie algebroid $B$, whereas \emph{loc. cit.} work in the setting $B=TP$. However by the transversality assumption $\rho \pitchfork TQ$, the splitting theorem for Lie algebroids (see Section \ref{s:splitting}) applies, yielding the local model $B\simeq p^!j^!B$ where $p\colon U\rightarrow Q$ is the projection for the tubular neighborhood embedding induced by $\E$. Making use of the local model, there is a relatively routine adaptation of the arguments in \cite[Section 6]{higsonyi} and \cite[Section 4]{sadeghbraverman2022} that proves Proposition \ref{p:taylororder}.

\subsection{Rescaled bundle}
Define the following $\scr{A}(\sbN_PQ)$-module:
\[ \scr{M}(\sbW)=\bigg\{\sum u^{-k}\sigma_k \in \Gamma(W)[u,u^{-1}] \, \bigg| \,o^{sc}(\sigma_k)=k\bigg\}.\]
Using the description of scaling order from Proposition \ref{p:taylororder}, one can show (cf. \cite[Sections 3.4--3.6]{higsonyi} for details) that $\scr{M}(\sbW)$ induces a locally free and finitely generated sheaf of modules over $\sbN_PQ$, and hence there is an associated vector bundle
\[ \sbW\rightarrow \sbN_PQ \]
as the notation suggests. We refer to $\sbW$ as the \emph{rescaled bundle}. The restriction of $\sbW$ to the special fiber $\dN_PQ=u^{-1}(0)$ is the vector bundle $\dW\rightarrow \dN_PQ$ associated to the $\scr{A}(\dN_PQ)$-module
\[ \scr{M}(\dW)=\scr{M}(\sbW)/u\cdot\scr{M}(\sbW).\]
By construction $\scr{M}(\dW)$ is the associated graded module for the filtration of $\Gamma(W)$ by scaling order.

Suitably rescaled versions of differential operators that act on sections of $W$ will induce operators on $\sbN_PQ$ and hence on $\dN_PQ$ by restriction. The simplest case is a bundle endomorphism $L \in \End^l(W)$. Then $u^lL$ acts naturally on $\scr{M}(\sbW)$ and induces a bundle endomorphism $\sbW\rightarrow \sbW$, and hence a bundle endomorphism $\dL\colon \dW \rightarrow \dW$ by restriction. The corresponding map on sections $\scr{M}(\dW)$ is simply the induced degree $l$ map on the associated graded. For example, if $\alpha \in \Gamma(B^*)$ then $L=c(\alpha)$ is of this type with $l=1$.

Let $X \in \Gamma(B)$. There is a derivation of the Rees algebra $\scr{A}(\sbN_PQ)$, denoted $uX$ which acts by
\[ uX\sum u^{-k}f_k=\sum u^{-k+1}Xf_k.\]
Let $\dX$ denote the induced derivation on $\scr{A}(\dN_PQ)$. One has similar module derivations $u\nabla_X=\nabla_{uX}$ on $\scr{M}(\sbW)$ and $\dnabla_{\dX}$ on $\scr{M}(\dW)$. Since elements of $\scr{A}(\sbN_PQ)$ are finite Laurent polynomials, the derivation $\dX$ induced by $uX$ on the associated graded $\scr{A}(\dN_PQ)$ is locally nilpotent. In particular the exponential 
\[ \exp(\dX) \in \tn{Aut}(\scr{A}(\dN_PQ)) \]
may be defined algebraically by the usual power series formula. Similarly $\dnabla_{\dX}$ is locally nilpotent, and hence the exponential
\[ \exp(\dnabla_{\dX})\in \tn{Aut}(\scr{M}(\dW)) \]
is well-defined. 

The next proposition describes commutation properties of the operators $\dnabla_{\dX}$.
\begin{proposition}
The rescaled curvatures $R(uX_1,uX_2)$, $F(uX_1,uX_2)$ of $B$, $E$ respectively, induce bundle endomorphisms $\dR(\dX_1,\dX_2)=\dF(\dX_1,\dX_2)$ of $\dW$, and
\[ \dR(\dX_1,\dX_2)=[\dnabla_{\dX_1},\dnabla_{\dX_2}]. \]
The operators $\dnabla_{\dX}$ and $\dR(\dX_1,\dX_2)$ commute. On exponentiating one has, 
\begin{equation}
\label{e:bch}
\exp(\dnabla_{\dX_1})\exp(\dnabla_{\dX_2})=\exp(\dR(\dX_1,\dX_2)/2)\exp(\dnabla_{\dX_1+\dX_2}).
\end{equation}
\end{proposition}
\begin{proof}
Recall that 
\[ F=R\otimes 1+1\otimes F^{E/S} \in \Gamma(\bCl(B^*)\otimes \End_{\Cl}(W)). \]
As $o^g(F^{E/S}(X_1,X_2))$, $o^g(\nabla_{[X_1,X_2]})$ are both less than $2$, the rescaled operators $F^{E/S}(uX_1,uX_2)$, $\nabla_{[uX_1,uX_2]}$ on $\scr{M}(\sbW)$ restrict to the zero operator over the $u=0$ fiber $\scr{M}(\dW)$, hence 
\[ \dF(\dX_1,\dX_2)=\dR(\dX_1,\dX_2)=[\dnabla_{\dX_1},\dnabla_{\dX_2}]. \]
Similarly $o^g([\nabla_X,R(X_1,X_2)])<3$ by the Clifford connection property, hence $[\dnabla_{\dX},\dR(\dX_1,\dX_2)]=0$. The last claim follows from the others and the initial terms of the Baker-Campbell-Hausdorff formula.
\end{proof}

The next goal is to describe an embedding
\[ \vep\colon \scr{M}(\dW)\hookrightarrow \Gamma(p^*\gr(j^*W)). \]
For $q\in Q$ define an evaluation map
\[ \vep_q=\vep_{0_q} \colon \scr{M}(\dW) \rightarrow \gr(j^*W)_q,\]
by
\begin{equation} 
\label{e:vepm}
\vep_q\Big(\sum u^{-k}\sigma_k\Big)=\sum_{k\le 0} \sigma_k(q)_{[-k]} 
\end{equation}
where $\sigma_k(q)_{[-k]}$ denotes the image of $\sigma_k(q)\in (j^*W)^{-k}$ in the quotient $(j^*W)^k/(j^*W)^{k-1}$. More generally, for $X_q \in B_q$ with extension $X \in \Gamma(B)$, define
\[ \vep_{X_q}\colon \scr{M}(\dW) \rightarrow \gr(j^*W)_q, \]
by
\[ \vep_{X_q}=\vep_q \circ \exp(\dnabla_{\dX}). \]
One checks that this is independent of the choice of extension $X$ using \eqref{e:vepm}.

Choose a splitting 
\begin{equation} 
\label{e:orthogsplitting}
\dN_PQ\hookrightarrow j^*B 
\end{equation}
of the short exact sequence of vector bundles
\begin{equation}
\label{e:normalbundlesequence}
0 \rightarrow j^!B \rightarrow j^*B \rightarrow \dN_PQ\rightarrow 0 
\end{equation}
where the map $j^*B\rightarrow \dN_PQ$ is the composition of the anchor map with the quotient map to the normal bundle (this composition is surjective since $\rho \pitchfork TQ$). For example, the metric $h$ on $B$ determines such a splitting.

Below we shall identify $\dN_PQ$ with the image of \eqref{e:orthogsplitting}. For $Y_q\in \dN_PQ_q$ and $\dsigma \in \scr{M}(\dW)$ define
\[ \dsigma(Y_q)=\vep_{Y_q}\dsigma \in \gr(j^*W)_q=(p^*\gr(j^*W))_{Y_q}.\]
\begin{proposition}
\label{p:iso}
The map
\[ \vep \colon \dsigma \mapsto \Big(Y_q\in \dN_PQ_q\mapsto \dsigma(Y_q)\in p^*\gr(j^*W)_{Y_q}\Big) \]
is an isomorphism from $\scr{M}(\dW)$ onto the space of smooth sections of the bundle $p^*\gr(j^*W)\rightarrow \dN_PQ$ that are polynomial along the fibers of the projection $p$. Under $\vep$, a bundle endomorphism $\dL$ induced by $L \in \Gamma(\End^k(W))$ goes to the induced degree $k$ endomorphism $p^*(j^*L)_{[k]}$ on the associated graded.
\end{proposition}
We omit the proof which is similar to the discussion in \cite[Sections 3.4--3.6]{higsonyi}. The next result is similar to \cite[Lemma 3.30]{higsonyi}; we include the calculation this time. If $X \in \Gamma(B)$ is such that $j^*X \in \Gamma(j^!B)$ then $o^g(\nabla_X)<1$ hence $\dnabla_{\dX}=0$. Hence the only case of interest is when $j^*X \in \Gamma(\dN_PQ)$ (recall $\dN_PQ\subset j^*B$ via \eqref{e:orthogsplitting}).
\begin{proposition}
\label{p:rescaledcovder}
Let $X \in \Gamma(B)$ such that $j^*X\in \Gamma(\dN_PQ)$. Under the isomorphism given in Proposition \ref{p:iso}, the operator $\dnabla_{\dX}$ goes to
\[ \partial_X+\frac{1}{2}R(-,X)_{[2]} \]
where $\partial_X$, resp. $R(-,X)_{[2]}$, are given at the point $Y_q \in \dN_PQ_q$ by vertical directional derivative in the direction $X_q$, resp. $R(Y_q,X_q)_{[2]}$.
\end{proposition}
\begin{proof}
Let $\dsigma \in \scr{M}(\dW)$. Using \eqref{e:bch},
\begin{align*}
(\dnabla_{\dX}\dsigma)(Y_q)&=\vep_{Y_q}\dnabla_{\dX}\dsigma\\
&=\vep_q \exp(\dnabla_{\dY})\dnabla_{\dX}\dsigma\\
&=\vep_q \frac{\d}{\d s}\bigg|_{s=0} \exp(\dnabla_{\dY})\exp(\dnabla_{s\dX})\dsigma\\
&=\vep_q \frac{\d}{\d s}\bigg|_{s=0} \exp(s\dR(\dY,\dX)/2)\exp(\dnabla_{\dY+s\dX})\dsigma\\
&=\frac{1}{2}R(Y_q,X_q)_{[2]}\dsigma(Y_q)+\frac{\d}{\d s}\bigg|_{s=0} \dsigma(Y_q+sX_q)
\end{align*}
where in the last line we used the Leibniz rule and Proposition \ref{p:iso}.
\end{proof}

In one instance we will need a rescaled version of a bundle automorphism that covers a non-identity diffeomorphism of $P$. Let $\gamma \in \tn{Diff}(P)$ be a diffeomorphism fixing $Q$. In particular $\gamma$ preserves vanishing order of smooth functions on $Q$ and hence induces an algebra automorphism of $\scr{A}(\sbN_PQ)$. Assume there exist compatible lifts $\gamma^B \in \tn{Aut}(B,h)$ and $\gamma^W \in \tn{Aut}(W,c,\nabla)$. Suppose the restriction $\gamma^W|_Q \in \Gamma(\End^l(W)|_Q)$ for some $l$. Consider the map
\begin{equation}
\label{e:gammasbW}
\sigma \in \Gamma(W)[u,u^{-1}]\mapsto u^l \sigma^\gamma \in \Gamma(W)[u,u^{-1}],
\end{equation}
where $\sigma^\gamma(g)=\gamma^W \sigma(\gamma^{-1}g)$.
\begin{proposition}
The map \eqref{e:gammasbW} extends to a map $\scr{M}(\sbW)\rightarrow \scr{M}(\sbW)$ intertwining the automorphism of $\scr{A}(\sbN_PQ)$ induced by $\gamma$.
\end{proposition}
\begin{proof}
It suffices to show that if $\sigma$ has scaling order $k$ then $\sigma^\gamma$ has scaling order at most $k+l$. By Proposition \ref{p:taylororder} we may use Taylor order. Since $\gamma$ is a diffeomorphism fixing $Q$, it preserves vanishing order of functions along $Q$ and maps any chosen Euler-like section $\E$ to some other Euler-like section $\E'$. Since the lifts $\gamma^B$, $\gamma^W$ preserve $\nabla$, \eqref{e:gammasbW} sends $\E$-synchronous sections to $\E'$-synchronous sections. If $\sigma$ is $\E$-synchronous with $j^*\sigma \in \Gamma((j^*W)^k)$, then the $\E'$-synchronous section $\sigma^\gamma$ is uniquely determined near $Q$ by its restriction to $Q$, 
\[ j^*(\sigma^\gamma)=(\gamma^W|_Q)(j^*\sigma) \in \Gamma(\End^l(j^*W))\cdot \Gamma((j^*W)^k)\subset \Gamma((j^*W)^{k+l}).\]
Thus Taylor order increases by at most $l$ as required. 
\end{proof}

Restricting to $u=0$ we obtain the map
\begin{equation}
\label{e:dgamma}
\scr{M}(\dW)\rightarrow \scr{M}(\dW), \qquad \dsigma \mapsto \dsigma^{\dgamma},
\end{equation}
where $\dsigma^{\dgamma}(X)=\dgamma^W\dsigma(\dgamma^{-1}X)$, and $\dgamma^W$ is the induced degree $l$ map on the associated graded.

\section{Fixed point formula for the equivariant trace pairing}\label{s:fixedptcalc}
In this section we combine Theorem \ref{t:t0term} and the rescaled bundle construction from Section \ref{s:rescaling} to compute the equivariant trace pairing $\tau^\gamma(\tn{ind}(D))$. 

\subsection{Rescaled bundle for the fixed point formula}
We specialize the results of Section \ref{s:rescaling} to the case
\[ P=G, \quad Q=M^\gamma, \quad B=\delta A, \quad W=\delta E,\]
where $(j^*W)^k=\End^k(E)$, see Example \ref{ex:fixedptexample1}. Here and below we abuse notation by using $j$ to denote both the inclusion $M^\gamma \hookrightarrow M$ and the inclusion $M^\gamma \hookrightarrow G$, when it will be clear from context which is meant. We shall use a simplified notation for the corresponding spaces and vector bundles:
\[ \sbG_\gamma=\sbN_GM^\gamma, \quad \dG_\gamma=\dN_GM^\gamma, \quad \sbE_\gamma=\sbW, \quad \dE_\gamma=\dW.\] 
The associated graded 
\begin{equation}
\label{e:assocgr}
\gr(j^*W)=j^*\gr(\End(E))\simeq j^*(\wedge A^*_\bC\otimes \End_{\Cl}(E)),
\end{equation}
with the standard grading on the exterior algebra.

Recall that if $X,Y \in \Gamma(A)$ then we have used the notation $X\oplus Y$ for the section $X^R\oplus Y^L$ of $B=\delta A$; we use similar simplifications elsewhere, for example if $m \in M^\gamma$ then $X_m\oplus Y_m:=(X^R\oplus Y^L)_m$. We shall need the following.
\begin{proposition}
\label{p:vepXY}
Let $X,Y \in \Gamma(A)$ and $m \in M^\gamma$. Then
\[ \vep_{X_m\oplus Y_m}=\exp(-\tfrac{1}{4}\la X_m|R|Y_m\ra)\vep_m\exp(\dnabla_{\dY\oplus \dY})\exp(\dnabla_{\dX-\dY\oplus \dz}).\]
%as operators, where $R(X_m,Y_m)\in \mf{so}(A_m)\simeq \wedge^2A_m^*$ acts on $\wedge A_\bC^*|_m\otimes \End_\Cl(E)|_m$ by exterior multiplication on the left.
\end{proposition}
\begin{proof}
Using the properties developed in Section \ref{s:rescaling}:
\begin{align*}
\vep_{X_m\oplus Y_m}&=\vep_m\exp(\dnabla_{\dz\oplus \dY}+\dnabla_{\dX\oplus \dz})\\
&=\vep_m\exp(\dnabla_{\dz\oplus \dY})\exp(\dnabla_{\dX\oplus \dz}) \\
&=\vep_m\exp(\dnabla_{\dY\oplus \dY}+\dnabla_{-\dY\oplus \dz})\exp({\dnabla_{\dX\oplus \dz}}) \\
&=\vep_m\exp(\dnabla_{\dY\oplus \dY})\exp(\dnabla_{-\dY\oplus \dz})\exp({\dnabla_{\dX\oplus \dz}}) \\
&=\vep_m\exp(\dnabla_{\dY\oplus \dY})\exp(\dR(\dX,\dY)/2)\exp(\dnabla_{\dX-\dY\oplus \dz})\\
&=\vep_m\exp(\dR(\dX,\dY)/2)\exp(\dnabla_{\dY\oplus \dY})\exp(\dnabla_{\dX-\dY\oplus \dz}),
\end{align*}
where in lines 5, 6 we used \eqref{e:bch}. After evaluation at $m$, $\dR(\dX,\dY)$ becomes $R(X_m,Y_m)_{[2]}$, the induced degree $2$ operator on the associated graded, which in turn is the image of $R(X_m,Y_m)$ under the composition $\mf{so}(A_m)\rightarrow \Cl^{[2]}(A_m)\rightarrow \wedge^2 A_m$. This map differs from the standard map $\mf{so}(A_m)\rightarrow \wedge^2A_m$ induced by the metric by a factor of $2$ (cf. \cite[Chapter 3]{BerlineGetzlerVergne}). Using the symmetry \eqref{e:swapsymmetry}, \eqref{e:swapsymmetry2}, we have $\vep_m\exp(\dR(\dX,\dY)/2)=\exp(-\tfrac{1}{4}\la X_m|R|Y_m\ra)\vep_m$.
%&=\exp(R(X_m,Y_m)/2)\vep_m\exp(\dnabla_{\dY\oplus \dY})\exp(\dnabla_{\dX-\dY\oplus \dz})
\end{proof}

\subsection{Rescaled generalized Laplacian}
The short exact sequence
\[ 0 \rightarrow j^*A \rightarrow \dN_GM^\gamma \rightarrow \dN_MM^\gamma \rightarrow 0 \]
of vector bundles over $M^\gamma$ has a canonical splitting $\dN_MM^\gamma \rightarrow \dN_GM^\gamma$ induced by the inclusion $M\hookrightarrow G$. Thus elements of $\dN_GM^\gamma$ can be represented as pairs
\[ (X,V), \quad \text{where} \quad X \in (j^*A)_m,\, V \in \dN_MM^\gamma_m, \, m \in M^\gamma.\]
In our situation the short exact sequence \eqref{e:normalbundlesequence} specializes to
\[ 0 \rightarrow j^!(\delta A)\rightarrow j^*\delta A \rightarrow \dN_GM^\gamma \rightarrow 0.\]
Choose the splitting $\dN_GM^\gamma \rightarrow \delta A$ (see \eqref{e:orthogsplitting}) given by
\begin{equation} 
\label{e:splitting}
(X,V)\in \dN_GM^\gamma \mapsto (X+V)\oplus V \in j^*\delta A,
\end{equation}
where $V \in \dN_MM^\gamma$ is identified with an element of $j^*A$ via the isomorphism $\dN_MM^\gamma\simeq (j^!A)^\perp$ determined by the metric. The composition $j^*A\rightarrow \dN_GM^\gamma \rightarrow j^*\delta A$ sends $X$ to $X\oplus 0$, tangent to the $s$ fibers. On the other hand \eqref{e:splitting} sends $(0,V)\in \dN_MM^\gamma \mapsto V\oplus V \in j^*\delta A$, and $\delta \rho(V\oplus V)=V^R-V^L$ is tangent to the unit space $M$.

By Proposition \ref{p:iso} (and equation \eqref{e:assocgr}), there is an isomorphism $\vep$ from $\scr{M}(\dE_\gamma)$ to the space of sections of
$j^*(\wedge A^*_\bC\otimes \End_{\Cl}(E))$ that are polynomial along the fibers of the projection map $p\colon \dN_GM^\gamma \rightarrow M^\gamma$. Specializing Proposition \ref{p:rescaledcovder} and using \eqref{e:swapsymmetry2}, we have that under $\vep$ the operator $\dnabla_{\dX\oplus \dz}$ is sent to
\begin{equation} 
\label{e:rescaledcovder}
\partial_X-\frac{1}{4}R|X\ra,
\end{equation}
where $R|X\ra$ denotes the bundle endomorphism of $p^*j^*(\wedge A_\bC^*\otimes \End_{\Cl}(E))$ which, at the point $(Y,V) \in j^*A\oplus \dN_MM^\gamma$, is left exterior multiplication by $\la Y+V|R|X\ra$.

By the Lichnerowicz formula \eqref{e:Lichnerowicz}, the generalized $A$-Laplacian $\Delta=D^2$ has Getzler order $2$. Hence the operator $u^2\Delta$ is well-defined on $\scr{M}(\sbE_\gamma)$, and induces an operator $\dDelta$ on $\scr{M}(\dE_\gamma)$. Equations \eqref{e:rescaledcovder}, \eqref{e:Lichnerowicz}, \eqref{e:Bochner} yield the following. Let $R_{ij}=\la \partial_{X^i}|R|\partial_{X^j}\ra$ be the matrix entries of $R$.
\begin{corollary}
Under the isomorphism given in Proposition \ref{p:iso}, the operator $\dDelta$ is
\[ \dDelta=-\sum_{i=1}^n \Big(\frac{\partial}{\partial X^i}+\frac{1}{4}\sum_{j=1}^n R_{ij}(X^j+V^j)\Big)^2+F^{E/S} \]
in local linear orthonormal coordinates $(X^i,V^j)$ along the fibers of $\dG_\gamma\simeq j^*A\oplus \dN_MM^\gamma$.
\end{corollary}

\subsection{Rescaled trace}\label{s:rescaledtrace}
Recall that the trace $\tau$ is given by restriction to the unit space $M\subset G$ followed by pairing with a generalized section $\tau \in \Gamma^{-\infty}(\Lambda^{-1}\otimes \Lambda_M)$ which is $\Gamma(A)$-invariant \eqref{e:PsiInv}. By transversality (Proposition \ref{p:transversefixedpt}), the image of $\rho(\Gamma(A|_{M^\gamma}))$ under the quotient map $TM|_{M^\gamma}\rightarrow \dN_MM^\gamma$ generates the normal bundle to $M^\gamma$. Thus equation \eqref{e:PsiInv} implies that $\tau$ is smooth in directions co-normal to $M^\gamma$. In particular by the wavefront set condition, the pullback $j^*\tau\in \Gamma^{-\infty}(j^*(\Lambda^{-1}\otimes \Lambda_M))$ is well defined. Let $p \colon \dN_MM^\gamma\rightarrow M^\gamma$ denote the projection, and let
\[ \bm{\tau}=p^*j^*\tau\in \Gamma^{-\infty}(p^*j^*(\Lambda^{-1}\otimes \Lambda_M)).\]
Recall that a short exact sequence of vector bundles
\[ 0 \rightarrow V_1\rightarrow V \rightarrow V_2 \rightarrow 0 \]
induces a canonical isomorphism $|\det(V)|\simeq |\det(V_1)|\otimes |\det(V_2)|$, independent of a choice of splitting. Thus $p^*j^*\Lambda_M$ is canonically identified with $\bm{\Lambda}_M=|\det(T\dN_MM^\gamma)|$, the density bundle for the total space of the normal bundle. Similarly $p^*j^*\Lambda$ is canonically identified with $\bm{\Lambda}=|\det(p^*j^*A)|$. It follows that $\bm{\tau}$ may be regarded as a generalized section of $\bm{\Lambda}^{-1}\otimes \bm{\Lambda}_M$, and so determines a linear functional $\bm{\tau}$ on the space of Schwartz sections of $\bm{\Lambda}$ over $\dN_MM^\gamma$.

In Section \ref{s:schwartz} we introduced the space of Schwartz functions near $u=0$ on the deformation to the normal cone space $\sbN_PQ$, and we now apply this in the case $P=M$, $Q=M^\gamma$. The metric $h$ on $A$ determines a trivialization of $\Lambda$, hence also of its pullback $\pi^*\Lambda$ to $\sbN_MM^\gamma$, and so we may speak of Schwartz sections of $\pi^*\Lambda$.
\begin{lemma}
\label{l:n1vanishing}
Let $s \in \Gamma(\pi^*\Lambda)$ be Schwartz near $u=0$ with support contained in $\sbN_NM^\gamma \subset \sbN_MM^\gamma$, where $N$ is a tubular neighborhood of $M^\gamma$ such that the connected components of $N$ have disjoint closures. Let $n_1\colon N\rightarrow \bZ$ be the (locally constant) codimension of $N$. Then
\[ \lim_{u\rightarrow 0} \tau(u^{-n_1}s(u))=\bm{\tau}(s(0)).\]
\end{lemma}
\begin{proof}
%By Lemma \ref{l:SchwarzProperties} (b), we may as well assume $s$ has support contained in $\sbN_NM^\gamma \subset \sbN_MM^\gamma$ for some tubular neighborhood $N$ of $M^\gamma$. 
Throughout the argument we identify $N$ with the normal bundle to $M^\gamma$ in $M$, and write $p \colon N \rightarrow M^\gamma$ for the vector bundle projection. Using a finite partition of unity on $M^\gamma$ to localize further, we may assume without loss of generality that $M^\gamma$ has a single connected component and $N$ is a trivial vector bundle, $N\simeq M^\gamma \times V$ for a vector space $V$ of dimension $n_1$. Since $\rho \pitchfork TM^\gamma$ (Proposition \ref{p:transversefixedpt}), the splitting theorem for Lie algebroids (see Section \ref{s:splitting}) gives a local model
\begin{equation} 
\label{e:locmod}
A|_N\simeq j^!A\times TV.
\end{equation}
Let $\R$ be the Euler vector field for $N=M^\gamma\times V$, which we promote to an Euler-like section $\E$ of $A|_N$ using \eqref{e:locmod}; trivially $\rho(\E)=\R$.

By \eqref{e:locmod}, $\Lambda|_N=|\det(A^*|_N)|\simeq |\det((j^!A)^*)|\boxtimes (V\times |\det(V^*)|)$. All three line bundles are trivial. Choose a trivialization $\Xi_{j^!A}$ of $|\det((j^!A)^*)|$ and a volume form $\d^{n_1}v$ on $V$, and let $\Xi=\Xi_{j^!A}\boxtimes|\d^{n_1}v|$ be the corresponding trivialization of $\Lambda|_N$. Since $\L_\R(\d^{n_1}v)=n_1\d^{n_1}v$ we have $\L_\E(\Xi)=n_1\Xi$.

Let
\[ f=\Xi^{-1}\otimes s \in C^\infty(\sbN_NM^\gamma) \]
be the corresponding smooth function, which is Schwartz near $u=0$. Since $\Xi$ is non-vanishing, we may factor $\tau|_N$ as
\[ \tau|_N=\Xi^{-1}\otimes \mu \] 
where $\mu \in \Gamma^{-\infty}(\Lambda_N)$. By equation \eqref{e:PsiInv},
\begin{equation} 
\label{e:homogeneity}
\L_{\rho(\E)}\mu=n_1\mu,
\end{equation}
hence $\mu$ is homogeneous of degree $n_1$ for scalar multiplication on the fibers of $N\rightarrow F$. In particular $\mu$ is a tempered generalized section.

The map $\iota$ from \eqref{e:iota} is a diffeomorphism from $\sbN_NF$ to $\bR\times N$. By Lemma \ref{l:SchwartzProperties} (a), $\iota_*(f)|_{[-1,1]\times N}$ is Schwartz. By \eqref{e:homogeneity}, $\iota_\ast(u^{-n_1}\mu)=\mu$ for $u \ne 0$. Thus
\[ \lim_{u\rightarrow 0} u^{-n_1}\tau(s(u))=\lim_{u\rightarrow 0} \mu(\iota_*f(u))=\mu(\lim_{u\rightarrow 0}\iota_*f(u))=\mu(\iota_*f(0))=\bm{\tau}(s(0)),\]
where the second equality follows because $u \in [-1,1]\mapsto \iota_*f(u)$ is a continuous family of Schwartz functions on $N$ while $\mu$ belongs to the topological dual, and the third equality follows because under the identifications above $\bm{\tau}=\Xi^{-1}\otimes \mu$.
\end{proof}

Let $\tr_s^M \colon \Gamma(\delta E)\rightarrow C^\infty(M)$ denote the composition of restriction to the unit space $M\subset G$ with the fiberwise supertrace of $\End(E)$. Recall that the latter vanishes identically on $\End^{n-1}(E)$. It follows from Proposition \ref{p:taylororder} that $\tr_s^M$ has filtration order $-n$, where $\Gamma(\delta E)$ is filtered by scaling order $o^{sc}$ and $C^\infty(M)$ is filtered by vanishing order $o^{van}$. Therefore there is an induced map of the Rees constructions (and of $\scr{A}(\sbG_\gamma)$-modules),
\begin{equation} 
\label{e:unvanishing1}
\tr_s^{\sbN_MM^\gamma}:=u^{-n}\tr_s^M \colon \scr{M}(\sbG_\gamma)\rightarrow \scr{A}(\sbN_MM^\gamma).
\end{equation}
Restricting to the fiber $\dG_\gamma=u^{-1}(0)$, we obtain a map
\begin{equation} 
\label{e:unvanishing2}
\tn{\textbf{tr}}^{\dN_MM^\gamma}_s\colon \scr{M}(\dG_\gamma)\rightarrow \scr{A}(\dN_MM^\gamma),
\end{equation}
which is simply the induced degree $-n$ map between the associated graded modules. We extend these maps $C^\infty$-linearly to smooth sections.

The bisection $\gamma$ of $G$ determines a bisection $\sbN_\gamma M^\gamma$ (the deformation to the normal cone for $M^\gamma$ inside the submanifold $\gamma \subset G$) of the Lie groupoid $\sbN_GM^\gamma=\sbG_\gamma$. Its intersection with $\dN_GM^\gamma=u^{-1}(0)$ is $\dgamma=\dN_\gamma M^\gamma$, a bisection of the Lie groupoid $\dN_GM^\gamma \rightrightarrows \dN_MM^\gamma$. Combining $\dgamma$ with the trace $\bm{\tau}$ yields the $\dgamma$-twisted trace $\bm{\tau}^{\dgamma}$, a continuous linear functional on the convolution algebra of the Lie groupoid $\dN_GM^\gamma$ built using Schwartz half-densities.

The generalization of the $\dgamma$-twisted trace to the situation with coefficients in $E$ involves one additional complication. Recall that by assumption there is an auxiliary invertible section $\gamma^E \in \Gamma(\delta E|_\gamma)$. At first glance one wants to extend the action of $(\gamma,\gamma^E)$ to the rescaled bundle $\sbE_\gamma\rightarrow \sbG_\gamma$ as in equation \eqref{e:gammasbW}, for a suitable scaling exponent $l$ (see equation \eqref{e:gammasbW}). By Propositions \ref{p:transversefixedpt}, \ref{p:gammaAaction}, $\gamma^E|_{M^\gamma} \in \End^{n_1}(E|_{M^\gamma})$, since its $\bC l(A^*|_{M^\gamma})$-component is the quantization (in the sense of \cite[Chapter 2]{BerlineGetzlerVergne}) of $\gamma^A|_{M^\gamma}$, and the latter acts trivially on $j^!A$ and with no non-zero fixed vectors on $\dN_MM^\gamma\simeq (j^!A)^\perp$. This suggests that the correct scaling exponent is $l=n_1$. But $M^\gamma$ may have multiple components of different dimensions in which case the codimension $n_1$ is only a locally constant function $M^\gamma \rightarrow \bZ$, and setting $l=n_1$ in \eqref{e:gammasbW} does not make sense.

To handle this issue one can proceed as follows. Choose a $\tilde{\gamma}$-invariant tubular neighborhood $N$ of $M^\gamma$ such that the connected components of $N$ have disjoint closures. Let $Q=M^\gamma\subset P=s^{-1}(N)$ and note that the diffeomorphism of $G$ given by left multiplication by $\gamma$ preserves the open subset $P$ and fixes $Q$. The pair $(\gamma,\gamma^E)$ determine an automorphism of the vector bundle $W=\delta E|_P\rightarrow P$, compatible with the connection and Clifford module structure. Extend $n_1$ in the obvious way to a function $N\rightarrow \bZ$ which is constant on each connected component. Then we may apply \eqref{e:gammasbW} instead to $P$, $Q$, $W$, $(\gamma,\gamma^E)|_P$ as above. Since the $u^{-1}(0)$ fiber of $\sbN_PM^\gamma$ is canonically identified with that of $\sbN_GM^\gamma$, this suffices to define the induced pair $(\dgamma,\dgamma^E)$ for $\dE_\gamma \rightarrow \dG_\gamma$. If $\sigma \in \Gamma(\dE_\gamma\otimes \bm{\rho}^*\bm{\Lambda})$ is such that $\tn{\textbf{tr}}_s^{\dN_MM^\gamma}(\sigma^{\dgamma})$ is a Schwartz section of $\bm{\Lambda}$, then we define
\[ \bm{\tau}^{\dgamma}_s(\sigma)=\bm{\tau}(\tn{\textbf{tr}}^{\dN_MM^\gamma}_s(\sigma^{\dgamma})), \qquad \sigma^{\dgamma}(g)=\dgamma^E\sigma(\gamma^{-1}g).\]
\begin{corollary}
\label{c:traceextends}
Let $\sigma$ be a smooth section of $\sbE_\gamma \otimes \pi_G^*\delta\Lambda^{1/2}$ such that the support of $\sigma(u)$ is sufficiently close to the unit space for all $u\ne 0$, and such that $\tr_s^{\sbN_MM^\gamma}(\sigma^\gamma)$ is Schwartz near $u=0$. Then
\begin{equation} 
\label{e:traceextends}
\lim_{u\rightarrow 0}u^{-n}\tau_s^\gamma(\sigma(u))=\bm{\tau}_s^{\dgamma}(\sigma(0)).
\end{equation}
\end{corollary}
\begin{proof}
By sufficiently close we mean more precisely that $\tn{support}(\sigma(u)^\gamma)\cap M\subset N$ for all $u \ne 0$, where $N$ is the tubular neighborhood of $M^\gamma$ chosen above. In this case $\sigma(u)^\gamma|_M$ equals the extension-by-zero of $\sigma(u)^\gamma|_N$ from $N$ to $M$; below we abuse notation slightly by using the same expression $\sigma(u)^\gamma|_N$ for this extension-by-zero. Recall $n_1$ is constant on each connected component of $N$. By equation \eqref{e:gammasbW}, Lemma \ref{l:n1vanishing}, and equations \eqref{e:unvanishing1}, \eqref{e:unvanishing2}, the left hand side of \eqref{e:traceextends} is 
\[ \lim_{u\rightarrow 0} u^{-n}\tau(\tr_s(\sigma(u)^\gamma|_M))=\lim_{u\rightarrow 0} u^{-n}\tau(u^{-n_1}\tr_s(u^{n_1}\sigma(u)^\gamma|_N))=\bm{\tau}_s^{\dgamma}(\sigma(0)),\]
as claimed.
\end{proof}

\subsection{Rescaled asymptotic heat kernel}
Recall the asymptotic heat kernel $K_t$ introduced in Section \ref{s:Borel}, which served as an approximate solution of the heat equation for the generalized Laplacian $\Delta$. For $u \ne 0$, the family of operators $K_{tu^2}$, $t>0$ plays the analogous role for the rescaled operator $u^2\Delta$ (see also Remark \ref{r:weightedGetzler} below). Note that $K_t$ was well-defined for all $t \in (0,\infty)$, and the substitution $t\mapsto tu^2$ is a reparametrization of $(0,\infty)$. One has
\begin{equation} 
\label{e:Ku}
K_{tu^2}(X,m)=\chi q_{tu^2}(X)\sum_{i=0}^\infty\beta(b_itu^2)t^iu^{2i}\Theta_i(X,m)\nu_m^{1/2},
\end{equation}
\begin{equation} 
\label{e:qu}
q_{tu^2}(X)=(4\pi tu^2)^{-n/2}e^{-|X|^2/4tu^2}\d X^{1/2}, 
\end{equation}
and the coefficients $\Theta_i$ are the unique smooth solutions of the recursive system
\begin{equation} 
\label{e:recursive}
(\nabla_\R+i)\Theta_i(-,m)=-B\Theta_{i-1}(-,m), \qquad \Theta_0(X,m)=\id,
\end{equation}
with $\R$ the Euler vector field of $A$ and $B=\d X^{-1/2}\Delta \d X^{1/2}$.%\footnote{Also recall that in this notation, $y \in M$, $x \in s^{-1}(y)$, and we identify a neighborhood of $y$ in $s^{-1}(y)$ with a neighborhood of $0$ in $A_y$, and $|x|$ denotes the length of the corresponding vector in $A_y$. Multiplication on the left with $\gamma^{-1}$ preserves source fibers, so $s(\gamma^{-1}x)=y$. We may want to re-work the notation. We did this to try to make the notation look like the pair groupoid case. But it seems a bit confusing. Perhaps better to replace $(x,y)$ with $g$ and introduce $\d(g)$ as the Riemannian distance along the $s$ fiber from $g$ to $s(g)$.}

As in the classical setting, it follows from \eqref{e:recursive} that $\Theta_i$ has scaling order $-2i$, hence $u^{2i}\Theta_i$ extends smoothly to the $u=0$ fiber and we obtain a section $\dTheta_i$ at $u=0$. As the smallest possible scaling order is $-n$, one has $\dTheta_i=0$ for $2i>n$. 
\begin{proposition}
\label{p:dTheta0}
$\dTheta_0(X,V)=\exp\big(-\tfrac{1}{4}\langle X|R|V\rangle\big)$.
\end{proposition}
\begin{proof}
By Proposition \ref{p:vepXY} and noting that $\la V|R|V\ra=0$ by antisymmetry of $R$, we find
\begin{align*}
\dTheta_0(X,V)&=\vep_{X+V\oplus V}\dTheta_0\\
&=\exp(-\tfrac{1}{4}\la X+V|R|V\ra)\vep_m \exp(\dnabla_{\dV\oplus \dV})\exp(\dnabla_{\dX\oplus \dz)})\dTheta_0\\
&=\exp(-\tfrac{1}{4}\la X|R|V\ra)\vep_m\exp(\dnabla_{\dV\oplus \dV})\dTheta_0\\
&=\exp(-\tfrac{1}{4}\la X|R|V\ra)\vep_m\dTheta_0\\
&=\exp(-\tfrac{1}{4}\la X|R|V\ra)\id
\end{align*}
where to obtain the third line we used that $\nabla_{\dX\oplus \dz}\Theta_0=0$, and to obtain the fourth line we used that $\Theta_0|_M$ is constant equal to $\id$.
\end{proof}

As $u \rightarrow 0$, $\beta(b_itu^2)\rightarrow 1$. The function $|X|^2$ has scaling order $2$ hence $|X|^2/u^2$ extends smoothly to the function $|X|^2$ on the $u=0$ fiber, where $(X,V)\in A|_{M^\gamma}\oplus \dN_MM^\gamma$. To handle the factor $(4\pi t)^{-n/2}$ in \eqref{e:qu}, we multiply by an additional factor of $u^n$. The product $u^nK_{tu^2}$ then extends smoothly to a section $\dK_t$ over the $u=0$ fiber.

\begin{proposition}
\label{p:dKt}
$\dK_t=\dK_t(X,V)$ is the solution to the heat equation for $\dDelta$ given by
\[ (4\pi t)^{-n/2}\tn{det}^{1/2}\bigg(\frac{tR/2}{\sinh(tR/2)}\bigg)\exp\Big(-\tfrac{1}{4t}\big\langle X\big|\tfrac{tR}{2}\coth(\tfrac{tR}{2})\big|X\big\rangle-tF^{E/S}-\tfrac{1}{4t}\big\langle X\big|tR\big|V\big\rangle\Big).\]
\end{proposition}
\begin{remark}The formula in Proposition \ref{p:dKt} differs by the factor $\exp(-\tfrac{1}{4}\langle X|R|V\rangle)$ from the standard expression for Mehler's kernel. $F^{E/S} \in \Gamma(\wedge^2A^*\otimes \End_\Cl(E))$ was defined in \eqref{e:barF}.\end{remark}
\begin{proof}
Recall that $K_t$ satisfies the heat equation for $\Delta$ modulo an $O(t^\infty)$ error. Making the substitution $t\leadsto tu^2$, multiplying by $u^n$ and taking $u\rightarrow 0$, we deduce that $\dK_t$ becomes an exact solution of the heat equation for $\dDelta$.

Let
\[ \tilde{\delta}_i=\frac{\partial}{\partial X^i}+\frac{1}{4}\sum_{j=1}^n R_{ij}X^j, \quad \delta_i=\tilde{\delta}_i+\frac{1}{4}\sum_{j=1}^nR_{ij}V^j,\]
so that
\[ \dDelta=-\sum_{i=1}^n \delta_i^2+F^{E/S},\]
and let
\[ \tilde{\dDelta}=-\sum_{i=1}^n \tilde{\delta}_i^2+F^{E/S}. \]
Let $\Phi_t(X,V)$ denote the right hand side of the expression in Proposition \ref{p:dKt}, that is,
\[ \Phi_t(X,V)=\exp\big(-\tfrac{1}{4}\big\langle X\big|R\big|V\big\rangle\big)\tilde{\Phi}_t(X,V),\]
where $\tilde{\Phi}_t(X,V)$ is Mehler's kernel; then (cf. \cite[Chapter 4]{BerlineGetzlerVergne}),
\[ (\partial_t+\tilde{\dDelta})\tilde{\Phi}_t=0.\]
Since
\[ \delta_i \circ \exp\big(-\tfrac{1}{4}\big\langle X\big|R\big|V\big\rangle\big)=\exp\big(-\tfrac{1}{4}\big\langle X\big|R\big|V\big\rangle\big)\circ \tilde{\delta}_i,\]
one finds
\begin{align*}
(\partial_t+\dDelta)\Phi_t&=(\partial_t+\dDelta)\exp\big(-\tfrac{1}{4}\big\langle X\big|R\big|V\big\rangle\big)\tilde{\Phi}_t\\
&=\exp\big(-\tfrac{1}{4}\big\langle X\big|R\big|V\big\rangle\big)(\partial_t+\tilde{\dDelta})\tilde{\Phi}_t\\
&=0.
\end{align*}
Thus $\Phi_t$ is a solution of the heat equation for $\dDelta$.

The heat equation for $\dDelta$ may be solved recursively as in \eqref{e:recursive}, with the solution uniquely determined given $\dTheta_0$. Proposition \ref{p:dTheta0} shows that $\Phi_t$ has the correct initial term $\dTheta_0$ in the recursion, and this completes the proof.
\end{proof}

\begin{remark}
\label{r:weightedGetzler}
For a slightly more geometric of approach to the discussion above, note that one can pull back the rescaled bundle $\sbE_\gamma$ to the larger deformation space $\bR\times \sbN_GM^\gamma=\sbN_{\bR\times G}(\{0\}\times M^\gamma)$ and use the weighted structure on the Lie groupoid $\bR\times G$ where the coordinate $t \in \bR$ is assigned weight $2$ (as in Section \ref{s:heatparametrix}). The resulting vector bundle over the larger space $\bR\times \sbN_GM^\gamma$ provides a natural home for the Getzler rescalings of the heat operator $\partial_t+\Delta$ and its Getzler-rescaled limit as well as for its parametrix, whose Getzler-rescaled limit (when $\gamma=1$) is Mehler's kernel.
\end{remark}

\subsection{Calculation of the equivariant trace pairing}
Recall that by Theorem \ref{t:t0term}, the equivariant trace pairing $\pair{\tau}{\tn{ind}(D)}(\gamma)=\tau^\gamma(\tn{ind}(D))$ is the constant term in the asymptotic expansion of $\tau^\gamma_s(K_t)=\tau_s(K_t^\gamma)$ as $t\rightarrow 0^+$. The discussion in the previous section showed that $\sigma(u)=u^n K_{tu^2}$ extends smoothly to $\sbG_\gamma$, restricting to a smooth section $\dK_t$ over the fiber $u^{-1}(0)=\dG_\gamma$.

\begin{proposition}
\label{p:traceu0}
The equivariant trace pairing $\tau^\gamma(\tn{ind}(D))$ equals the constant term in the asymptotic expansion of $\bm{\tau}^\gamma_s(\dK_t)$ as $t\rightarrow 0^+$.
\end{proposition}
\begin{proof}
For any $t>0$, Remark \ref{r:suppKtgamma} and Example \ref{ex:gaussian} show that $\tr_s^{\sbN_MM^\gamma}(u^nK_{tu^2}^\gamma) \in C^\infty(\sbN_MM^\gamma)$ is Schwartz near $u=0$. By appropriate choice of cutoff function $\chi$ in the definition of $K_t$, the support of $K_t$ can be arranged to be arbitrarily close to the unit space. Thus by Corollary \ref{c:traceextends}, 
\begin{equation} 
\label{e:limu0}
\bm{\tau}^{\dgamma}_s(\dK_t)=\lim_{u\rightarrow 0} u^{-n}\tau^\gamma_s(u^nK_{tu^2})=\lim_{u\rightarrow 0}\tau^\gamma_s(K_{tu^2}).
\end{equation}
The substitution $t\mapsto tu^2$ preserves the $t^0$ term in the asymptotic expansion of $\tau^\gamma_s(K_t)$, hence taking asymptotic expansions of both sides of \eqref{e:limu0} as $t\rightarrow 0^+$ yields the result.
\end{proof}

The Levi-Civita $A$-connection is $\sf{K}$-invariant and hence preserves the orthogonal splitting
\[ A|_{M^\gamma}=j^!A\oplus (j^!A)^\perp\simeq j^!A \oplus \dN_MM^\gamma.\]
It follows that the pullback of the curvature $j^!R$ is block diagonal with blocks $R_0 \in \Gamma(\wedge^2 j^!A\otimes \mf{so}(j^!A))$ the curvature of $j^!A$, and $R_1 \in \Gamma(\wedge^2 j^!A \otimes \mf{so}(\dN_MM^\gamma))$ the curvature of the induced $j^!A$-connection on $\dN_MM^\gamma$. 

Let $|\nu| \in \Gamma(\Lambda_{\dN})$, $\Lambda_{\dN}=|\det(\dN_MM^\gamma)|$ be the density determined by the restriction of the metric $h$ to the subbundle $\dN_MM^\gamma\simeq (j^!A)^\perp$. Recall that the generalized section $\tau$ has a well-defined pullback $j^*\tau \in \Gamma^{-\infty}(j^*(\Lambda^{-1}\otimes \Lambda_M))$. Since
\[ j^*\Lambda_M\simeq \Lambda_{\dN}\otimes \Lambda_{M^\gamma}, \]
we may multiply $j^*\tau$ by $|\nu|^{-1}\in \Gamma(\Lambda_{\dN}^{-1})$ to obtain a generalized section
\[ \frac{j^*\tau}{|\nu|} \in \Gamma^{-\infty}(j^*\Lambda^{-1}\otimes \Lambda_{M^\gamma}), \]
for which there exists a natural pairing with smooth sections of $j^*\Lambda$. 

We come to our main theorem, which is a fixed point formula \eqref{e:fixedptformula} for the equivariant trace pairing. 
The notation in \eqref{e:fixedptformula} is similar to that used in \cite[Theorem 6.16]{BerlineGetzlerVergne} (to which it reduces in the case $A=TM$). The variables $n_0,n_1$ are the ranks of $j^!A,\dN_MM^\gamma$ respectively; they are locally constant, even integer-valued functions on $M^\gamma$ satisfying $n_0+n_1=n$, the rank of $A$. The characteristic forms $\Ahat(j^!A,j^!\nabla)$, $\tn{det}^{1/2}(1-\gamma_1 e^{-R_1})$ belong to $\Gamma(\wedge(j^!A)^*)$. Recall that $\gamma_1$ is an isometry with no non-zero fixed vector. Thus $\tn{det}(1-\gamma_1)>0$, and the analytic square-root is chosen such that $\tn{det}^{1/2}(1-\gamma_1)>0$. The same observation guarantees $\tn{det}^{1/2}(1-\gamma_1 e^{-R_1})\in \Gamma(\wedge(j^!A)^*)$ has non-vanishing degree $0$ component, hence is invertible. The characteristic form $\tn{\textbf{ch}}^\gamma(E/S,\nabla)\in \Gamma(\wedge (j^*A)^*)$ is the localized relative $A$-Chern character form (\cite[Definition 6.13]{BerlineGetzlerVergne} adapted to a general $A$); we will give its detailed definition in the course of proving the theorem, and for now only remark that in the case $E=S\otimes W$ where $S$ is the spinor bundle associated to a Spin double cover of the bundle of oriented orthonormal frames of $A$, then $\tn{\textbf{ch}}^{\gamma}(E/S,\nabla)=\tn{\textbf{ch}}^{\gamma}(j^*W,j^*\nabla)$ the $\gamma$-twisted Chern character of $j^*W$. Via the inclusion $j^!A\subset j^*A$, the quotient on the right hand side of the pairing in \eqref{e:fixedptformula} below may be regarded as a section of $\wedge(j^*A)^*\boxtimes \bC$ over $M^\gamma$. Taking its degree $n$ component results in a section of $\wedge^n (j^*A)^*\boxtimes \bC$, and using the orientation on $A$, the latter is isomorphic to $j^*\Lambda$. Thus there is a natural $\bC$-valued pairing between the objects on the right hand side of \eqref{e:fixedptformula} below.
\begin{theorem}
\label{t:fixedptformula}
Let $G$ be a Hausdorff Lie groupoid over a compact unit space $M$ with even rank oriented metrised Lie algebroid $(A,h)$. Let $\sf{K}$ be a compact Lie group that acts on $G$ by bisections preserving the metric and orientation on $A$. Let $\tau \colon \Psi^{-\infty}(G)\rightarrow \bC$ be a continuous trace that factors through restriction to the unit space. Let $(E,c,\nabla)$ be a $\sf{K}$-equivariant $\Cl(A^*)$-module with Clifford $A$-connection and let $D=c\circ \nabla$ be the corresponding $A$-Dirac operator. Let $\gamma \in \sf{K}$ and $j\colon M^\gamma \hookrightarrow M$ the inclusion. %Assume in addition that $\rho$ is transverse to $TM^\gamma$ and that $\gamma$ fixes $j^!A$. 
Then
\begin{equation} 
\label{e:fixedptformula}
\tau^\gamma(\tn{ind}(D))=\bigg\la \frac{j^*\tau}{|\nu|},\bigg(\frac{\Ahat(j^!A)\tn{\textbf{ch}}^{\gamma}(E/S)}{(2\pi \i)^{n_0/2}\i^{n_1/2}\tn{det}^{1/2}(1-\gamma_1 e^{-R_1})}\bigg)_{[n]}\bigg\ra.
\end{equation}
\end{theorem}
\begin{proof}
By Proposition \ref{p:traceu0}, we compute the constant term in the asymptotic expansion of $\bm{\tau}^\gamma_s(\dK_t)=\bm{\tau}_s(\dK_t^\gamma)$ as $t\rightarrow 0^+$. By equation \eqref{e:connvectgamma},
\[ \vep_{(X+V)\oplus V}\dK_t^\gamma=\dgamma^E\vep_{(\gamma^A)^{-1}(X+V)\oplus V}\dK_t.\]
Therefore
\[ \dK_t^\gamma(X,V)=\dgamma^E\vep_{(\gamma^A)^{-1}(X+V)\oplus V}\dK_t=\dgamma^E\dK_t((\gamma^A)^{-1}(X+V)-V,V).\]
To simplify notation, for the rest of the proof we will write $\gamma$ instead of $\gamma^A$. The action of $\gamma^A$ on $j^*A$ was described in Proposition \ref{p:gammaAaction}.

Using Proposition \ref{p:dKt}, we obtain a complicated expression for $\bm{\gamma}^E\dK_t(\gamma^{-1}(X+V)-V,V)$. To compute the supertrace $\bm{\tau}_s$ we need only the value at $X=0$ of this expression, which simplifies to
\begin{align} 
\label{e:Ktgamma0}
\nonumber \dK_t^\gamma(0,V)=(4\pi & t)^{-n/2}\bm{\gamma}^E\tn{det}^{1/2}\bigg(\frac{tR/2}{\sinh(tR/2)}\bigg)\exp(-tF^{E/S})\\
& \exp\Big(-\tfrac{1}{4t}\langle (\gamma^{-1}V-V)|\tfrac{tR}{2}\coth(\tfrac{tR}{2})|(\gamma^{-1}V-V)\rangle-\tfrac{1}{4t}\langle \gamma^{-1}V|tR|V\rangle\Big).
\end{align}
By Corollary \ref{c:traceextends}, $\bm{\tau}_s(\dK_t^\gamma)$ is given as the evaluation of the generalized section $j^*\tau$ on
\begin{equation} 
\label{e:fiberint}
|\nu|^{-1}\int_{\dN_MM^\gamma/M^\gamma} \tr_s(\dK_t^\gamma(0,V)) |\nu|(\d^{n_1}V).
\end{equation}
Note that $\dN_MM^\gamma$ need not be an orientable vector bundle, hence the integration over the fibers is in the sense of densities. The integral over $V$ just involves the $V$-dependent exponential in \eqref{e:Ktgamma0}, and so we begin by computing this factor.

Recall that $\dN_MM^\gamma$ has even rank $n_1$. Locally on $M^\gamma$, $\dN_MM^\gamma$ splits into a direct sum of rank $2$ subbundles preserved by the action of $\gamma_1$ (see Proposition \ref{p:gammaAaction} for the definition of $\gamma_1$). The integral over the fibers \eqref{e:fiberint} becomes a product of integrals; to evaluate these we need the following.
\begin{lemma}
Let $1\ne \gamma \in SO(2)$. Then for $A \in \mf{so}(2)$ sufficiently small,
\[ \int_{\bR^2}e^{-\frac{1}{4}\langle (\gamma^{-1}V-V)|\frac{A}{2}\coth(\frac{A}{2})|(\gamma^{-1}V-V)\rangle-\frac{1}{4}\langle \gamma^{-1}V|A|V\rangle}\d^2 V=\frac{4\pi\cdot \tn{det}^{1/2}\Big(\frac{\sinh(A/2)}{A/2}\Big)}{\tn{det}^{1/2}(1-\gamma)\tn{det}^{1/2}(1-\gamma e^{-A})}.\]
\end{lemma}
\begin{proof}
In the proof we repeatedly use $\gamma^T=\gamma^{-1}$, $A^T=-A$ and $\gamma A=A\gamma$. The integrand is $\exp(-\frac{1}{4}\la V|B|V\ra)$ where
\[ B=(\gamma^{-1}-1)^T\frac{A}{2}\coth(\frac{A}{2})(\gamma^{-1}-1)+\gamma A.\]
We may replace $B$ with its symmetric part $C=(B+B^T)/2$ without changing the integrand, and a short calculation reveals
\[ C=\frac{A/2}{\sinh(A/2)}(-e^{A/2}\gamma^{-1})(1-\gamma)(1-\gamma e^{-A}).\]
When $A=0$ (and as $\gamma \ne 1$), $C$ is positive definite, hence remains positive definite for $A$ sufficiently small. Thus
\[ \int_{\bR^2} e^{-\frac{1}{4}\la V|C|V\ra}\d^2V=\frac{4\pi}{\tn{det}^{1/2}(C)},\]
and as $\det(-e^{A/2}\gamma^{-1})=1$, the lemma follows.
\end{proof}
Applying the lemma to the integral over the fibers of the exponential factor in \eqref{e:Ktgamma0} yields
\begin{equation}
\label{e:factor1s}
\frac{(4\pi t)^{n_1/2}|\nu|^{-1}}{\tn{det}^{1/2}(1-\gamma_1)\tn{det}^{1/2}(1-\gamma_1 e^{-tR_1})}\tn{det}^{1/2}\bigg(\frac{\sinh(tR_1/2)}{tR_1/2}\bigg).
\end{equation}
The remaining factors in \eqref{e:Ktgamma0} are
\begin{equation} 
\label{e:factor2}
(4\pi t)^{-n/2}\bm{\gamma}^E\tn{det}^{1/2}\bigg(\frac{tR/2}{\sinh(tR/2)}\bigg)\exp(-tF^{E/S}).
\end{equation}
Multiplying equations \eqref{e:factor1s} and \eqref{e:factor2} gives
\begin{equation}
\label{e:product12}
|\nu|^{-1}\tn{det}^{1/2}\bigg(\frac{tR_0/2}{\sinh(tR_0/2)}\bigg)\frac{(4\pi t)^{-n_0/2}\bm{\gamma}^E\exp(-tF^{E/S})}{\tn{det}^{1/2}(1-\gamma_1)\tn{det}^{1/2}(1-\gamma_1 e^{-tR_1})}.
\end{equation}
Next recall that the relevant term in \eqref{e:product12} is the coefficient of $t^0$. Since each of the elements $F^{E/S}$, $R_0$, $R_1$ of exterior degree $2$ appears with a factor of $t$, the coefficient of $t^0$ is obtained from the exterior degree $n_0$ part of the corresponding product of forms. Recall that the exterior degree of $\bm{\gamma}^E$ is bounded by $n_1$. Applying the supertrace picks out the term of exterior degree $n=n_0+n_1$ and multiplies it by $(2/\i)^{n/2}$, while taking the relative supertrace $\tr_s^{E/S}$ (\cite[Definition 3.28]{BerlineGetzlerVergne}) of the $\End_\Cl(E)$-component (see \cite[pp.146, 195]{BerlineGetzlerVergne}). Thus the calculation of \eqref{e:fiberint} has lead to 
\begin{equation}
\label{e:productn0}
(2\pi \i)^{-n_0/2}|\nu|^{-1}\Bigg(\tn{det}^{1/2}\bigg(\frac{R_0/2}{\sinh(R_0/2)}\bigg)\frac{2^{n_1/2}\tr^{E/S}_s(\bm{\gamma}^E_{[n_1]}\exp(-F^{E/S}))}{\i^{n_1/2}\tn{det}^{1/2}(1-\gamma_1)\tn{det}^{1/2}(1-\gamma_1 e^{-R_1})}\Bigg)_{[n]}.
\end{equation}
The product
\[ \frac{2^{n_1/2}\tr^{E/S}_s(\bm{\gamma}^E_{[n_1]}\exp(-F^{E/S}))}{\tn{det}^{1/2}(1-\gamma_1)}=\tn{\textbf{ch}}^{\gamma}(E/S,\nabla) \]
is, by definition (compare \cite[Definition 6.13]{BerlineGetzlerVergne}), the localized relative $A$-Chern character form. If we also substitute the definition of $\Ahat(j^!A,j^!\nabla)$, equation \eqref{e:productn0} gives finally
\[ \bigg(\frac{\Ahat(j^!A,j^!\nabla)\tn{\textbf{ch}}^{\gamma}(E/S,\nabla)}{(2\pi \i)^{n_0/2}\i^{n_1/2}\tn{det}^{1/2}(1-\gamma_1 e^{-R_1})}\bigg)_{[n]}|\nu|^{-1},\]
which is the expression appearing in the statement of the theorem.
\end{proof}

\section{Some extensions of the main result}
In this brief section we mention two modest extensions of the main result: to non-Hausdorff Lie groupoids, and to proper actions of non-compact Lie groups.

\subsection{Non-Hausdorff Lie groupoids}
Let $G$ be a Lie groupoid which is not necessarily Hausdorff. As is common in the literature on Lie groupoids, we require that the base $M$ as well as the $r$ and $s$ fibers are Hausdorff. By a `possibly non-Hausdorff Lie groupoid' we will mean a Lie groupoid satisfying these conditions. In this section we shall briefly indicate how the main theorem may be extended to this setting under some additional hypotheses.

We begin by mentioning a simple example of a non-Hausdorff Lie groupoid that is helpful to keep in mind.
\begin{example}
\label{ex:pushout}
Consider the pushout $G=\bR \cup_J \bR$ where $J\subset \bR$ is any non-empty open subset. View $G$ as a bundle of discrete groups over $M=\bR$ with fibers $\{1\}$ for $x\in J$ and $\bZ/2\bZ$ otherwise. The unit space $M$ as well as the $r$ and $s$ fibers are Hausdorff, but note for example that $M$ is not a closed subset of $G$.
\end{example}

Returning to the general situation, an important observation is that there is always an open neighborhood $U$ of $M$ in $G$ that is Hausdorff. Indeed each $s$ fiber $s^{-1}(m)$ is Hausdorff, and thus the Riemannian exponential map $\exp^h|_m$ gives a diffeomorphism from a neighborhood of $0 \in A_m$ to a neighborhood of $m$ in $s^{-1}(m)$. Altogether these assemble to a diffeomorphism $\exp^h$ from a neighborhood of the $0$ section in $A$ to a neighborhood $U$ of $M$ in $G$. Many of the constructions in the article are carried out locally on a neighborhood of $M$ in $G$, and hence go through without change. Note however that even though $M$ has a Hausdorff neighborhood, there may still be points $m \in M$ for which there exists some point $g \in G$ such that $m,g$ do not have disjoint neighborhoods.

The fixed point subsets of the action of a compact Lie group $\sf{K}$ on a possibly non-Hausdorff manifold $G$ are smooth, as in the Hausdorff case. To see this let $\gamma \in \sf{K}$, $g \in G^\gamma$, and let $\sf{K}_g$ be the stabilizer of $g$ under the action. By compactness of $\sf{K}_g$ and smoothness of the action, one can find Hausdorff open neighborhoods $U,V_1,...,V_n$ of $g$ and an open cover $\sf{K}_{g,1},...,\sf{K}_{g,n}$ of $\sf{K}_g$ such that $\sf{K}_{g,i}\cdot U\subset V_i$. Then $V=V_1\cap \cdots \cap V_n$ is a Hausdorff open neighborhood of $g$, and its inverse image under the action map $\sf{K}_{g,i}\times U\rightarrow V_i$ contains a product neighborhood of the form $\sf{K}_{g,i}'\times U'$, where $U'\subset U$ is a Hausdorff open neighborhood of $g$ and $\sf{K}_{g,i}'\subset \sf{K}_{g,i}$ are open subsets such that $\sf{K}_{g,1}',...,\sf{K}_{g,n}'$ still cover $\sf{K}_g$. Then $U''=\sf{K}_g\cdot U'\subset V$ is a Hausdorff open neighborhood of $g$ that is preserved by $\sf{K}_g$. $G^\gamma$ is smooth at $g$ because $(U'')^\gamma$ is smooth, being a fixed point subset for the action of a compact Lie group $\sf{K}_g$ on the Hausdorff manifold $U''$.

The definition of an action of $\sf{K}$ on $G$ by bisections is as before. Note that bisections $\gamma\subset G$, being diffeomorphic to $M$, are Hausdorff submanifolds of $G$. If $U$ is a Hausdorff neighborhood of $M$ in $G$ then $\gamma U$ is a Hausdorff neighborhood of $\gamma$ in $G$. The group action properties developed in Section \ref{s:groupactions} only involve local considerations in charts (smoothness and transversality being local properties for example), and thus apply in the non-Hausdorff setting as well. One important difference is that the fixed-point locus $M^\gamma$, although smooth, need not be a closed submanifold of $M$. In Example \ref{ex:pushout} there is a non-trivial $\bZ_2$ action by bisections that swaps the two copies of $\bR$, and hence $M^\gamma=J$ can be any open subset of $\bR$. To rule out such behavior we shall assume that the union $U\cup \gamma U$ is Hausdorff. 

The definition of the convolution algebra for a non-Hausdorff Lie groupoid associated to a foliation is explained by Connes \cite{connes1982survey}, and can be adapted to other non-Hausdorff Lie groupoids. In brief one defines $C^\infty_c(G)$ to be the vector space spanned by functions obtained by pushing forward a smooth compactly supported function in a chart and extending by $0$. One uses a similar definition for sections of vector bundles, for example $\Psi^{-\infty}(G)$. Elements of $C^\infty_c(G)$ need not be continuous and the space of such functions is not closed under pointwise multiplication. Nevertheless one has a well-defined convolution algebra. The larger algebra of pseudodifferential operators is defined similarly. As explained above, there is a Hausdorff tubular neighborhood $U$ of $M$ in $G$, and hence quantization of complete symbols can be carried out as in the Hausdorff setting. The reduced kernel of a general $G$-equivariant pseudodifferential operator is then a sum of a reduced kernel supported in $U$ and in the image of the quantization map, and an element of $\Psi^{-\infty}(G)$.

Let $F\subset M$ be the closure in $M$ of the set points $m \in M$ for which there exists some point $g \in G$ such that $m,g$ do not have disjoint neighborhoods. The restriction of an element of $C^\infty_c(G)$ or $\Psi^{-\infty}(G)$ to $M$ is smooth on $M\backslash F$, but can fail to be continuous at points in $F$ (this occurs in Example \ref{ex:pushout}). To ensure a well-defined trace pairing, we assume that there is a $G$-invariant open neighborhood $R\subset M$ of $F$ in $M$ such that $\tau|_R$ is a Radon measure (after choosing any smooth trivialization of $\Lambda_M\otimes \Lambda^{-1}|_R$) and $F$ has $\tau|_R$-measure $0$. If $\tau$ satisfies this condition then we shall say that $\tau$ is admissible. Let $\{\rho,1-\rho\}$ be a partition of unity on $M$ subbordinate to $\{R,M\backslash F\}$. If $f$ is the restriction of an element of $\Psi^{-\infty}(G)$ to $M$, we define
\[ \pair{\tau}{f}:=\pair{\tau|_{M\backslash F}}{(1-\rho)f}+\pair{\tau|_R}{\rho f}. \]
The discussion in Section \ref{s:asymptoticheat} is as in the Hausdorff situation, for example, the construction of the asymptotic heat kernel is unchanged. The assumption that $U\cup \gamma U$ is Hausdorff also ensures that Remark \ref{r:suppKtgamma} remains valid.

By assumption $M^\gamma$ is a closed submanifold of the Hausdorff manifold $G'=U\cup \gamma U$. Hence the deformation space $\sbN_{G'}M^\gamma$ can be defined using the Rees construction as in Section \ref{s:rescaling}. We then define $\sbN_G M^\gamma$ to be the pushout of $\sbN_{G'}M^\gamma$ and $(G\backslash M^\gamma)\times \bR^\times$ along $(G'\backslash M^\gamma)\times \bR^\times$. The definition of the rescaled bundle is similar. The calculations in Section \ref{s:fixedptcalc} only involve local considerations on the Hausdorff manifold $U\cup \gamma U$, and so proceed as before, leading to the following extension of Theorem \ref{t:fixedptformula}.
\begin{theorem}
\label{t:nonHausdorfffixedpt}
Consider the situation of Theorem \ref{t:fixedptformula} except that $G$ is a possibly non-Hausdorff Lie groupoid. Assume that $\tau$ is admissible and that there is a neighborhood $U$ of $M$ in $G$ such that $U\cup \gamma U$ is Hausdorff. Then the fixed-point formula \eqref{e:fixedptformula} holds.
\end{theorem}

\subsection{Proper actions}
Let $G$ be a Lie groupoid and let $H$ be a Lie group that acts on $G$ by bisections. Since bisections act on $G$ by left multiplication, there is an induced action
\[ H\times G \rightarrow G, \qquad (\gamma,g)\mapsto \gamma g.\]
We say that the action of $H$ on $G$ by bisections is \emph{proper} if this induced action is proper in the usual sense. An action of a compact Lie group is automatically proper.

Let $H$ act properly on $G$ by bisections, and assume that $M^\gamma$ is non-empty. Then the results of Section \ref{s:groupactions} hold with $H$ in place of $\sf{K}$. Indeed the properties of actions of compact Lie groups that were used also hold for proper actions of Lie groups. We also had occasion to consider $M^{\tilde{\gamma}}$, a fixed-point set for a possibly non-proper action. But if $M^\gamma \subset G^\gamma$ is non-empty, then properness implies that the closure of the subgroup of $H$ generated by $\gamma$ is a compact Lie group $H_\gamma$, thus we are still able to deduce that $M^{\tilde{\gamma}}=M^{H_\gamma}$ is smooth.

In the case $H$ is non-compact, the K-theory groups appearing in Section \ref{s:equivtracepairing} should be further clarified. Instead, to avoid this, one can take equation \eqref{e:gradedtauind} as the definition of the `equivariant trace pairing' $\pair{\tau}{\tn{ind}(D)} \in C^\infty(H)^H$; the proof of Proposition \ref{p:tracepairingindependent} shows this is independent of the choice of $H$-equivariant parametrix. If $M^\gamma=\gamma \cap M$ is empty, then, by choosing a parametrix for $D$ supported sufficiently close to $M\subset G$, one can arrange that the supports of $\gamma (Q^+)^2$, $\gamma (R^-)^2$ do not intersect $M$, hence $\pair{\tau}{\tn{ind}(D)}(\gamma)=\tau^\gamma(\tn{ind}(D))=0$ by equation \eqref{e:gradedtauind}. The rest of the local calculation of the pairing in the following sections is unchanged. Thus, supposing $H$ acts properly on $G$ by bisections, Theorem \ref{t:fixedptformula} holds with $H$ in place of $\sf{K}$.

\section{Examples}\label{s:examples}
In the special case $G=\tn{Pair}(M)$ for a closed manifold $M$ with $\tau$ the trace given by integration over the diagonal in $M\times M$, Theorem \ref{t:fixedptformula} reduces to the Atiyah-Bott-Segal-Singer fixed point formula for the equivariant index of a Dirac operator. Indeed in this case $j^*\tau/|\nu|$ coincides with the functional $\Gamma(j^*\Lambda)\rightarrow \bC$ given by Berezin integration followed by integration over $M^\gamma$ with respect to the Riemannian volume density, and the formula is precisely that given in \cite[Theorem 6.16]{BerlineGetzlerVergne}. 

Another special case is for $G$ the holonomy groupoid of a foliation $\scr{F}$ of a closed manifold $M$, with trace $\tau$ given by a transverse measure. In this instance Theorem \ref{t:fixedptformula} recovers a result of Connes \cite{connes1979theorie} for the non-equivariant case and Heitsch-Lazarov \cite{heitsch1990lefschetz} in the equivariant case. (Note however that there are some differences in settings, for example, \cite{heitsch1990lefschetz} relax the condition that $\gamma$ belong to a compact Lie group.)

Below we explain the case where $A={}^bTM$ is the b-tangent bundle associated to a hypersurface, or more generally to a simple normal crossing divisor, in $M$. The corresponding $A$-Dirac operator is not elliptic in the usual sense because its symbol degenerates in a prescribed way along the hypersurfaces. In the final subsection we describe an analogue of the Atiyah-Hirzebruch vanishing theorem in our context.

\subsection{Normal crossing divisors}\label{s:normalcrossing}
Let $M^n$ be a closed connected manifold. A simple normal crossing divisor is a finite collection of hypersurfaces $\Z$ such that around each point $m \in M$ there is a coordinate chart $(U,\varphi)$ centred at $m$ such that for each $Z\in \Z$, $\varphi(U\cap Z)$ is an open subset of a coordinate hyperplane in $\bR^n$; the pair $(U,\varphi)$ will be called a normal crossing coordinate chart. In particular it follows that at most $n$ of the hypersurfaces can intersect at any point. %A normal crossing divisor determines a stratification of $M$, where the top dimensional strata are the connected components of $M\backslash \cup \Z$.

The sheaf of vector fields tangent to all hypersurfaces in $\Z$ is generated locally, in a normal crossing chart $(U,\varphi=(x_1,...,x_n))$ with hypersurfaces $x_1=0,...,x_k=0$, by the vector fields
\begin{equation} 
\label{e:vectfields}
x_1\frac{\partial}{\partial x_1},\cdots,x_k\frac{\partial}{\partial x_k},\frac{\partial}{\partial x_{k+1}},\cdots \frac{\partial}{\partial x_n}. 
\end{equation}
It follows that the sheaf of such vector fields is locally free and finitely generated, hence is the sheaf of smooth sections of a vector bundle ${}^bTM$ called the b-tangent bundle. The dual vector bundle ${}^bT^*M$ is called the b-cotangent bundle. The natural map ${}^bTM\rightarrow TM$ and the restriction of the Lie bracket of vector fields makes $A={}^bTM$ into a Lie algebroid. An integration $G$ may be constructed by gluing or by a modified blow up construction (see for example \cite{gualtieri2017tropical, debordskandalisblowup, obster2021blow}). In any integration, there can be no arrows from points of $M\backslash \cup \Z$ to $\cup \Z$.

The desired trace on $\Psi^{-\infty}(G)$ will be based on principal value integration. Let $\I$ be the sheaf of smooth functions on $M$ that vanish along $\Z$. As in the case of ${}^bTM$, $\I$ is the sheaf of smooth sections of a line bundle that we also denote by $\I$. Smooth sections of the dual line bundle $\I^{-1}$ may be thought of as functions on $M$ with simple poles along $\Z$. The principal value integral is a linear functional
\[ \tn{pv}_\Z \int_M \colon \Gamma(\I^{-1}\otimes \Lambda_M)\rightarrow \bC \]
defined as follows. Choose an auxiliary Riemannian metric on $M$ and for $\epsilon>0$ let $M_\epsilon$ be the complement of a neighborhood of $\cup \Z$ of radius $\epsilon$. Then
\[ \tn{pv}_\Z \int_M f\otimes \alpha:=\lim_{\epsilon\rightarrow 0^+} \int_{M_\epsilon} f|_{M_\epsilon}\alpha|_{M_\epsilon} \]
where $f|_{M_\epsilon}$ is regarded as a smooth function on $M_\epsilon$. The definition is independent of the choice of auxiliary Riemannian metric. Using a partition of unity and suitable normal crossing charts, $\tn{pv}_\Z \int_M f\otimes \alpha$ may be expressed as a finite sum of usual principal value integrals of functions of the form $g/x_1\cdots x_k$, where $g$ is smooth and compactly supported. In particular, from the theory of distributions, we see that $\tn{pv}_\Z$ is a continuous linear functional.

Let
\[ \Omega:=\tn{det}({}^bT^*M)\boxtimes \bC, \qquad \Omega_M:=\tn{det}(T^*M)\boxtimes \bC. \]
The b-cotangent bundle ${}^bT^*M$ has a local frame dual to \eqref{e:vectfields}, whose elements are suggestively denoted
\[ \frac{\d x_1}{x_1},\cdots,\frac{\d x_k}{x_k},\d x_{k+1},\cdots, \d x_n.\]
A section of $\Omega$ can be thought of as a top degree form on $M$ with simple poles along $\Z$, i.e. there is an isomorphism:
\begin{equation}
\label{e:Omega}
\Omega \simeq \I^{-1}\otimes \Omega_M.
\end{equation}

From now on assume $TM$, ${}^bTM$ are oriented. On $M\backslash \cup \Z$, the anchor map switches from being orientation-preserving to orientation-reversing (and vice versa) each time a hypersurface is crossed. We say that the integration $G$ is \emph{even} if for every $g \in G$ such that $r(g),s(g)\in M\backslash \cup \Z$, the anchor map is either orientation-preserving or orientation-reversing at both $r(g)$ and $s(g)$; equivalently $r(g)$ can be reached from $s(g)$ by a smooth path crossing an even number of hypersurfaces in $\Z$. For example, for any integration $G'$, the connected component $G$ of $G'$ containing $M$ is an integration of ${}^bTM$ with this property.
%By reversing the orientation of $TM$ if necessary, we may assume that the anchor map ${}^bTM\rightarrow TM$ preserves orientation on some connected component $M_0$ of $M\backslash \cup\Z$.

The orientations on $TM$, ${}^bTM$ determine isomorphisms
\[ \Omega \simeq \Lambda, \qquad \Omega_M\simeq \Lambda_M.\]
Composing these with \eqref{e:Omega} yields an isomorphism
\begin{equation} 
\label{e:mapo}
o\colon \Lambda \xrightarrow{\sim} \I^{-1}\otimes \Lambda_M, \qquad f \mapsto (f)_o.
\end{equation}
In local coordinates, $o$ is the map
\begin{equation} 
\label{e:sgn}
\bigg|\frac{\d x_1 \cdots \d x_n}{x_1\cdots x_k}\bigg|\mapsto \tn{sgn}(\rho,x>0)\frac{|\d x_1\cdots \d x_n|}{x_1\cdots x_k} 
\end{equation}
where $\tn{sgn}(\rho,x>0)=\pm 1$ according to whether the anchor map $\rho$ is orientation-preserving or reversing on the connected component of $M\backslash \cup \Z$ containing the subset of the coordinate patch where $x_1,...,x_k>0$.

Define the continuous linear functional $\tau \colon \Gamma(\Lambda)\rightarrow \bC$ to be the composition of $o$ with principal value integration $\tn{pv}_\Z\int_M$: 
\[ \tau(f)=\tn{pv}_\Z \int_M (f)_o.\]
Note that the corresponding generalized section of $\Lambda_M\otimes \Lambda^{-1}$ is not locally integrable.

\begin{proposition}
\label{p:pvtrace}
Let $G$ be an even integration of ${}^bTM$. Then $\tau$ defines a trace on $\Psi^{-\infty}(G)$.
\end{proposition}
\begin{proof}
We must show that $\tau(s_*f)=\tau(r_*f)$ for all $f \in \Gamma(\delta \Lambda)$ with compact support. By \eqref{e:mapo},
\[ \delta \Lambda=r^*\Lambda \otimes s^*\Lambda\simeq r^*\Lambda \otimes s^*(\I^{-1}\otimes \Lambda_M).\]
Since $r^*A\simeq \ker(Ts)$, the pullback $r^*\Lambda$ is isomorphic to the density bundle along the $s$ fibers. As $s$ is a submersion, there is a short exact sequence of vector bundles
\[ 0 \rightarrow \ker(Ts)\rightarrow TG\rightarrow s^*TM\rightarrow 0.\]
Hence there is a canonical (independent of the choice of splitting) isomorphism $r^*\Lambda \otimes s^*\Lambda_M\simeq \Lambda_G$, the density bundle of $G$. On the other hand $s^*(\I^{-1})=\I_G^{-1}$, where $\I_G$ is the line bundle associated to the normal crossing divisor $\Z_G=\{s^{-1}(Z)\mid Z \in \Z\}$. Note also that $s^{-1}(Z)=r^{-1}(Z)$ since $Z$ is a $G$-invariant subset of $M$, so it does not matter where $s$ or $r$ is used in the definition of $\Z_G$. Let
\[ o_s \colon \delta \Lambda \xrightarrow{\sim} \I^{-1}_G\otimes \Lambda_G \]
be the composition of these two isomorphisms. fiber integration ($s_*$) is compatible with principal value integration,
\[ \tau(s_*f)=\tn{pv}_\Z \int_M (s_*f)_o=\tn{pv}_{\Z_G}\int_G (f)_{o_s}.\]
Replacing $s$ with $r$ throughout yields a similar expression
\[ \tau(r_*f)=\tn{pv}_\Z\int_M(r_*f)_o=\tn{pv}_{\Z_G}\int_G(f)_{o_r}.\]
It remains to compare the isomorphisms $o_s$, $o_r$. Let $g \in G$ be such that $r(g),s(g) \in M\backslash \cup \Z$. Let $x_1,...,x_n$ (resp. $y_1,...,y_n$) be coordinates in a normal crossing chart near $r(g)$ (resp. $s(g)$), and consider
\begin{equation} 
\label{e:LambdaSections}
\bigg|\frac{\d x_1 \cdots \d x_n}{x_1\cdots x_k}\bigg|_{r(g)}\otimes\bigg|\frac{\d y_1 \cdots \d y_n}{y_1\cdots y_l}\bigg|_{s(g)} \in (r^*\Lambda\otimes s^*\Lambda)_g. 
\end{equation}
We may assume the coordinates are chosen such that $r(g)$ (resp. $s(g)$) lies in the subset of the coordinate patch where $x_1,...,x_k>0$ (resp. $y_1,...,y_l>0$). Applying $o_s$, $o_r$ and using equation \eqref{e:sgn} we obtain
\[ \tn{sgn}(\rho,y>0)\bigg|\frac{\d x_1 \cdots \d x_n}{x_1\cdots x_k}\bigg|\otimes\frac{|\d y_1 \cdots \d y_n|}{y_1\cdots y_l}, \qquad \tn{sgn}(\rho,x>0)\frac{|\d x_1 \cdots \d x_n|}{x_1\cdots x_k}\otimes\bigg|\frac{\d y_1 \cdots \d y_n}{y_1\cdots y_l}\bigg| \]
in the two cases. Since $x_1,...,x_k,y_1,...,y_l>0$ at the points in question, we can drop the absolute values in the denominators and hence the two expressions agree except possibly for the sign $\tn{sgn}(\rho,x>0)/\tn{sgn}(\rho,y>0)$, which, since $G$ is an even integration, is $+1$.
\end{proof}
\begin{remark}
For the non-closed setting where $\Z$ is the boundary divisor of a compact manifold with boundary, the definition of $\tau$ above does not work. A variant of $\tau$, the `b-trace functional' \cite[Chapter 4.20]{melrose1993atiyah}, still plays a key role, although it depends on an additional choice and is not in fact a trace.
\end{remark}

We explain the specialization of Theorem \ref{t:fixedptformula} to the case $A={}^bTM$. By Proposition \ref{p:transversefixedpt}, $M^\gamma$ inherits a normal crossing divisor $\Z^\gamma=\{Z\cap M^\gamma \mid Z \in \Z, Z\cap M^\gamma \ne \emptyset\}$. The linear functional
\[ \frac{j^*\tau}{|\nu|}\colon \Gamma(j^*\Lambda)\rightarrow \bC \]
is the composition of 
\[ \tn{pv}_{\Z^\gamma}\int_{M^\gamma} \colon \Gamma(j^*\I^{-1}\otimes \Lambda_{M^\gamma})\rightarrow \bC \]
with the isomorphism
\[ o^\gamma \colon j^*\Lambda \rightarrow j^*\I^{-1}\otimes \Lambda_{M^\gamma}, \qquad o^\gamma=|\nu|^{-1}j^*o. \]
Theorem \ref{t:fixedptformula} yields the following.
\begin{corollary}
\label{c:fixedptb}
Let $M$ be a closed oriented even-dimensional manifold with simple normal crossing divisor $\Z$. Assume ${}^bTM$ is oriented with fiber metric $h$, and let $G$ be an even integration of ${}^bTM$. Let $\sf{K}$ be a compact Lie group that acts on $G$ by bisections preserving the metric and orientation of ${}^bTM$. Let $(E,c,\nabla)$ be a $\sf{K}$-equivariant $\Cl({}^bT^*M)$-module with Clifford ${}^bTM$-connection and let $D=c\circ \nabla$ be the corresponding b-Dirac operator. Let $\gamma \in \sf{K}$ and $j\colon M^\gamma \hookrightarrow M$ the inclusion. Let $\tau(-)=\tn{pv}_\Z\int_M (-)_o$. Then
\[ \tau^\gamma(\tn{ind}(D))=\tn{pv}_{\Z^\gamma}\int_{M^\gamma}\bigg(\frac{\Ahat({}^bTM^\gamma)\tn{\textbf{ch}}^{\gamma}(E/S)}{(2\pi \i)^{n_0/2}\i^{n_1/2}\tn{det}^{1/2}(1-\gamma_1 e^{-R_1})}\bigg)_{[n],o^\gamma}.\]
\end{corollary}

\begin{remark}
Slightly more generally, one gets a similar formula in the presence of an auxiliary foliation. Let $\F$ be a foliation with leaves transverse to all strata of the simple normal crossing divisor $\Z$. There is a Lie algebroid ${}^bT\F$ whose smooth sections are vector fields tangent to the leaves of $\F$ that are also tangent to $\Z$. An $\F$-invariant generalized section of $|\tn{det}(TM/T\F)^*|$ determines a continuous linear functional $\Gamma(\Lambda)\rightarrow \bC$, $\Lambda=|\tn{det}({}^bT\F)^*|$ using principal value integration over $M$ as in the case where $\F$ is trivial discussed above. Then Theorem \ref{t:fixedptformula} gives a fixed-point formula for the trace pairing with an equivariant family of b-Dirac operators on the leaves of a foliation.
\end{remark}

\subsection{A vanishing theorem}
A well-known consequence of the fixed-point formula for the index of a Dirac operator is a vanishing theorem of Atiyah and Hirzebruch: Let $M$ be a closed spin manifold equipped with a non-trivial $S^1$ action. Then the A-roof genus $\Ahat(M)=\int_M \Ahat(TM)=0$ and in fact the $S^1$-equivariant index of the Spin Dirac operator is $0$.

An analogous result can be deduced in our context. Recall that the $\d_A$-cohomology class of $\Ahat(A)$ depends only on $A$ and not on the choice of connection. By a minor extension of \cite[Theorem 5.1]{evens1996transverse}, the pairing $\pair{\tau}{\Ahat(A)_{[n]}}$ depends only on $\tau$ and $A$; this pairing plays the role of the A-roof genus. We say that $\sf{K}$ acts non-trivially if $M^\gamma \ne M$ for some $\gamma \in \sf{K}$.
\begin{corollary}
\label{c:rigidity}
Let $G$ be a Hausdorff Lie groupoid over a compact unit space $M$ with even rank oriented Lie algebroid $A$. Suppose $S^1$ acts non-trivially on $G$ by bisections. Let $\tau \colon \Psi^{-\infty}(G)\rightarrow \bC$ be a continuous trace that factors through restriction to the unit space. Let $D$ be the $A$-Spin Dirac operator associated to an $S^1$-invariant Spin structure on $A$. Then
\[ \pair{\tau}{\tn{ind}(D)}=0 \in C^\infty(S^1) \quad \text{and} \quad \pair{\tau}{\Ahat(A)_{[n]}}=0.\]
\end{corollary}
\begin{proof}
The proof is similar to the classical case, cf. \cite[Section IV.3, pp.293--295]{LawsonMichelsohn}. For $\gamma=z=e^{\i \theta}$ in a dense subset of $S^1=U(1)$, the fixed-point set $M^z=G^z\cap M=G^{S^1}\cap M=F$ does not depend on $z$. For such $z$ the fixed-point formula \eqref{e:fixedptformula} reads
\[ \pair{\tau}{\tn{ind}(D)}(z)=\bigg\la \frac{j^*\tau}{|\nu|},\bigg(\frac{\Ahat(j^!A)}{(2\pi \i)^{n_0/2}\i^{n_1/2}\tn{det}^{1/2}(1-z_1 e^{-R_1})}\bigg)_{[n]}\bigg\ra,\]
where $j\colon F\hookrightarrow M$, the localized relative Chern character factor in the numerator being $1$ in the Spin case. The LHS of this equation is a smooth function $\chi_1(z)$ defined for $z\in U(1)\subset \bC^\times$. By expanding the characteristic classes as in \cite[Section IV.3, p.294]{LawsonMichelsohn} (for example), the RHS admits a meromorphic extension $\chi_2(z)$ for $z \in \bC^\times \backslash R$, with a branch cut along the positive real axis $R$ (used to define a square root $z^{1/2}$). Poles of $\chi_2(z)$ may occur where $\det(1-z_1)=0$, that is, $z_1$ has $1$ as an eigenvalue on a fiber of the normal bundle to $F$, and in particular $z$ is a root of unity. The equality $\chi_1(z)=\chi_2(z)$ for $z \in U(1)$ (and where $\chi_2$ is defined) implies that $\chi_2$ is bounded, hence has no poles. Let $\chi_2^{\pm}$ be the two local holomorphic branches of $\chi_2$ near $z=1$. The equality $\chi_2^{\pm}(z)=\chi_1(z)$ for $z \in U(1)$ near $1$ with $\pm \tn{Im}(z)>0$, forces the $n$-th derivatives of $\chi_2^+$, $\chi_2^-$ at $z=1$ to be equal to each other (both being expressible in terms of $\partial_\theta$-derivatives of $\chi_1$ at $z=1$), for all $n\ge 0$. By analyticity $\chi_2^+=\chi_2^-$, hence $\chi_2$ extends over the cut $R$ to a holomorphic function $\chi$ on $\bC^\times$ extending $\chi_1$. The same argument as in the classical case, based on the identity
\[ \frac{\pm 1}{z^{1/2}-z^{-1/2}}=\frac{\pm z^{1/2}}{z-1}=\frac{\pm z^{-1/2}}{1-z^{-1}}, \]
shows that $\chi(z)$ has limits as $z \rightarrow 0,\infty$ equal to $0$. It follows that $\chi(z)$ extends analytically to the Riemann sphere $\wh{\bC}$ and hence vanishes identically by Liouville's theorem. Thus $\pair{\tau}{\tn{ind}(D)}(z)=0$ for all $z \in S^1$, and evaluating at $z=1$ gives $\pair{\tau}{\Ahat(A)_{[n]}}=0$.
\end{proof}
With additional hypotheses as in Theorem \ref{t:nonHausdorfffixedpt}, a similar statement holds in the non-Hausdorff case. We may apply Corollary \ref{c:rigidity} to the example described in Section \ref{s:normalcrossing}.
\begin{corollary}
Let $M$ be a closed oriented manifold admitting a non-trivial $S^1$-action and an $S^1$-invariant divisor $\Z$ such that the corresponding $b$-tangent bundle ${}^bTM$ is Spin. Then the A-roof genus $\Ahat(M)=\int_M \Ahat(TM)=0$.
\end{corollary}
\begin{remark} 
This result can be deduced by more conventional means: using an argument similar to \cite[Appendix]{BLSbsympl}, if ${}^bTM$ is spin then some finite covering space of $M$ is spin, and applying the classical vanishing theorem to the cover implies $\Ahat(M)=0$. However we find the proof based on Corollary \ref{c:rigidity} interesting and so include it here as well.
\end{remark}
\begin{proof}
We have
\[ \Ahat(M)=\int_M \Ahat(TM)=\int_{M\backslash \cup \Z}\Ahat(TM)=\int_{M\backslash \cup \Z}\Ahat({}^bTM)=\pair{\tau}{\Ahat({}^bTM)_{[n]}},\]
where $\tau$ is the trace defined in Section \ref{s:normalcrossing} and we choose the Chern-Weil representative $\Ahat({}^bTM)$ constructed using the restriction of a connection on $TM$ along the anchor map ${}^bTM\rightarrow TM$ so that $\Ahat({}^bTM)=\Ahat(TM)$ is a non-singular differential form. The right hand side vanishes by Corollary \ref{c:rigidity}.
\end{proof}

\printbibliography

%\bibliographystyle{amsplain}
%\bibliography{../Biblio}
\end{document}